\documentclass[a4paper,11pt]{amsart}
\usepackage{amsfonts}
\usepackage{amssymb}
\usepackage[utf8]{inputenc}
\usepackage{amsmath}
\usepackage{pdflscape}
\usepackage{float}
\usepackage{easyReview}
\usepackage{comment}
\usepackage{graphicx}
\usepackage[colorlinks=true, linkcolor=red, citecolor=blue]{hyperref}
\usepackage[]{epsfig}
\usepackage[]{pstricks}
\usepackage{tikz}
\newcommand{\R}{\mathbb R}
\newcommand{\N}{\mathbb N}

\newtheorem{theorem}{Theorem}[section]
\newtheorem{proposition}[theorem]{Proposition}
\newtheorem{lemma}[theorem]{Lemma}
\newtheorem{definition}{Definition}[section]

\newtheorem{remark}{Remark}
\theoremstyle{remark}

\usepackage{mathrsfs}

\numberwithin{equation}{section}
\makeatletter
\@namedef{subjclassname@2010}{\textup{2020} Mathematics Subject Classification}
\makeatother

\begin{document}
	
	\pagenumbering{arabic}	
\title[Controllability of KdV in half-line]{Another look at the control properties of the Korteweg-de Vries equation}
	\author[Capistrano-Filho]{Roberto de A. Capistrano-Filho*}
	\address{Departamento de Matem\'atica,  Universidade Federal de Pernambuco (UFPE), 50740-545, Recife (PE), Brazil.}
	\email{roberto.capistranofilho@ufpe.br}
	
	\author[Gallego]{Fernando Gallego}
\address{Departamento de Matem\'atica, Universidad Nacional de Colombia (UNAL), Cra 27 No. 64-60, 170003, Manizales, Colombia}
\email{fagallegor@unal.edu.co}	
\date{\today}
		\thanks{*Corresponding author: roberto.capistranofilho@ufpe.br}
		\thanks{Capistrano–Filho was partially supported by CAPES/COFECUB grant number 88887.879175/2023-00, CNPq grant numbers 421573/2023-6, 307808/2021-1, and 301744/2025-4, and Propesqi (UFPE). Gallego was partially supported by the Hermes UNAL project no. 61029}
	
\begin{abstract}
This paper represents a new perspective in understanding the controllability of the Korteweg-de Vries (KdV) equation on unbounded domains. By studying the equation on both the right and left half-line with a single control input, we show that a class of solutions exists for which the KdV equation is exactly controllable. This is accomplished through the introduction of a method for explicitly characterizing controls arising from the Hilbert Uniqueness Method, referred to as \textit{operational controllability}, which yields fundamental insights for proving exact controllability results for the KdV equation. This approach allows for explicitly characterizing both the control input and the controllable solutions. Furthermore, this concept holds significant potential for application to various nonlinear dispersive equations on the half-line and in bounded intervals.
	\end{abstract}

	\maketitle


\section{Introduction}
\subsection{Problem framework} The theory of local well-posedness (LWP), i.e., the existence, uniqueness, and continuity of the data-to-solution map, for the initial boundary value problem (IBVP) associated with the KdV equation \cite{Korteweg}, is well-studied. On the right half-line $(0,+\infty)$ is considered the following IBVP
\begin{equation}\label{h1}
\begin{cases}\partial_t u+ \partial_x u+\partial_x^3 u+u \partial_x u=0, & \text { for }(x, t) \in(0,+\infty) \times(0, T), \\ u(0, t)=f(t), & \text { for } t \in(0, T), \\ u(x, 0)=\phi(x), & \text { for } x \in(0,+\infty),
\end{cases}
\end{equation}
and on the left half-line $(-\infty, 0)$ the following one
\begin{equation}\label{h2}
\begin{cases}\partial_t u+ \partial_x u+\partial_x^3 u+u \partial_x u=0, & \text { for }(x, t) \in(-\infty, 0) \times(0, T), \\ u(0, t)=g_1(t), \quad \partial_x u(0, t)=g_2(t),& \text { for } t \in(0, T), \\ u(x, 0)=\phi(x), & \text { for } x \in(-\infty, 0).\end{cases}
\end{equation}

The rationale behind the presence of one boundary condition in the right half-line problem \eqref{h1} as opposed to two boundary conditions in the left half-line problem \eqref{h2} can be elucidated by integral identities concerning smooth solutions to the linear KdV equation. For such solutions, denoted by $v$, and any arbitrary time $t$, where $0<t<T$, the following holds:
\begin{equation}\label{h1rr}
\int_0^{+\infty} v^2(x, t) d x=\int_0^{+\infty} v^2(x, 0) d x+\int_0^t\left[2 v\left(0, t^{\prime}\right) v_{x x}\left(0, t^{\prime}\right)-v_x^2\left(0, t^{\prime}\right)\right] d t^{\prime}
\end{equation}
and
\begin{equation}\label{h2rr}
\int_{-\infty}^0 v^2(x, t) d x=\int_{-\infty}^0 v^2(x, 0) d x-\int_0^t\left[2 v\left(0, t^{\prime}\right) v_{x x}\left(0, t^{\prime}\right)-v_x^2\left(0, t^{\prime}\right)\right] d t^{\prime}.
\end{equation}

Now, assuming $v(x, 0)=0=v\left(0, t^{\prime}\right)$ for all $x>0$ and $0<t^{\prime}<t$, we deduce from \eqref{h1rr} that $v(x, t)=0$ for all $x>0$. However, the existence of $v(x, t) \neq 0$ for $x<0$, satisfying $v(x, 0)=0=v\left(0, t^{\prime}\right)$ for all $x<0$ and $0<t^{\prime}<t$, is not precluded by \eqref{h2rr}. Indeed, such nonzero solutions do exist, as demonstrated in \cite{Holmer}. Nevertheless, \eqref{h2rr} does reveal that homogeneous conditions $v(x, 0)=v\left(0, t^{\prime}\right)=v_x\left(0, t^{\prime}\right)=0$ for all $x<0$ and $0<t^{\prime}<t$ imply that $v(x, t)=0$ for all $x<0$.

Even though the well-posedness theory is well understood, as we shall see, it is important to present a fundamental problem that forms the basis of our work and that we intend to address:

\vspace{0.1cm}
\noindent \textbf{Exact control problem:} {\em Given an initial state $\phi$ and a terminal state $\psi$ in a certain space, can one determine a suitable control input $f$ (or $g_1$ and $g_2$) such that equation \eqref{h1} (or \eqref{h2}) admits a solution $u$ satisfying $u(x,0)=\phi(x)$ and $u(x,T)=\psi(x)$? }

\vspace{0.1cm}
\noindent If it is always possible to find a control input that drives the solution of the system described by \eqref{h1} or \eqref{h2} from any initial state $\phi$ to any final state $\psi$, then the system is said to be {\bf exactly controllable}. Moreover, if we take the final state $\psi=0$, the system is said to be {\bf null controllable}.

\subsection{Why study control problems on unbounded domains?}  The well-posedness of the system \eqref{h1}, as an initial-boundary value problem, is well established in the literature (see, for example, \cite{Bonahalf1,Bonahalf2,bonahalf3}). In particular, a result by Bona \textit{et al.}  \cite{Bonahalf2} guarantees the following:

\vspace{0.2cm}
\noindent \textbf{Theorem A.}\textit{ Let $\nu>0$ and $s>-1$ be given with $s \neq 3 m+\frac{1}{2},\ m=0,1,2, \cdots$. Set
$$
H_\nu^s\left(\mathbb{R}^{+}\right):=\left\{f \in H^s\left(\mathbb{R}^{+}\right) ; \mathrm{e}^{\nu x} f \in H^s\left(\mathbb{R}^{+}\right)\right\}.
$$
For any given compatible pair\footnote{Refer to \cite{Bonahalf1} for the precise definition of $s$-compatibility in this context.} $(\phi, h) \in H_\nu^s\left(\mathbb{R}^{+}\right) \times H_{l o c}^{\frac{s+1}{3}}\left(\mathbb{R}^{+}\right)$, there exists $T>0$ such that the IBVP \eqref{h1} admits a unique mild solution $u \in C\left([0, T] ; H_\nu^s\left(\mathbb{R}^{+}\right)\right)$.} 

\vspace{0.1cm}

As a direct consequence of Theorem A (see, for instance, \cite[Section 4]{RZsurvey}), they deduced that the system \eqref{h1} is not exactly controllable for a certain class of final states. Specifically, this class is given by 
$$\mathcal{C}:= \left\lbrace \phi_T \in H^s\left(\mathbb{R}^{+}\right): \text{$\phi_T \notin H^{s'}_{\nu}(\mathbb{R}^{+})$ for any $\nu$ and $s' > -1$}\right\rbrace.$$
 This result highlights that the exact controllability is obstructed for states that fail to belong to the more regular weighted Sobolev space \(H^{s'}_{\nu}(\mathbb{R}^{+})\), regardless of the weight parameter \(\nu\) and for any \(s' > -1\).   This naturally raises the question:
\vglue 0.2cm

\vspace{0.2cm}
\noindent\textbf{Question \(\mathcal{A}\):}  \textit{Is there a class of states for which the system \eqref{h1} (or \eqref{h2}) is either null controllable or exactly controllable?}

\vglue 0.2cm

In this direction, Rosier \cite[Theorem 1.2]{RosierSICON2000} shows that for a certain class of solutions, the linear system associated with \eqref{h1} is not null controllable for some states. Specifically, the following result was established:

\vspace{0.1cm}
\noindent \textbf{Theorem B.} \textit{Let \(T > 0\) be given. Then, there exists an initial data \(\phi \in L^2\left(\mathbb{R}^{+}\right)\) such that if \(u \in L^{\infty}\left(0, T; L^2\left(\mathbb{R}^{+}\right)\right)\) solves
$$
\left\{\begin{array}{l}
\partial_t u+\partial_x u+\partial_x^3 u=0, \quad \text{in} \ \mathcal{D}^{\prime}((0,+\infty) \times(0, T)), \\
u\big|_{t=0}=\phi,
\end{array}\right.
$$
then \(u\big|_{t=T} \neq 0.\)}

\vspace{0.1cm}

It means that the bad behavior of the trajectories as $x \rightarrow \infty$ is the price to be paid for getting the exact controllability in $(0,+\infty)$. Indeed, for a certain function $u_0$ in $L^2(0,+\infty)$ and $u_T=0$ a trajectory $u$ as above cannot be found in $L^{\infty}\left(0, T, L^2(0,+\infty)\right)$.  However, when the bounded energy condition \(\left(u \in L^{\infty}\left(0, T; L^2\left(\mathbb{R}^{+}\right)\right)\right)\) is relaxed, the exact boundary controllability of the linear KdV equation holds, as demonstrated also by Rosier in  \cite[Theorem 1.3]{RosierSICON2000}:

\vspace{0.2cm}
\noindent \textbf{Theorem C.} \textit{Let \(T, \epsilon\), and \(b\) be positive numbers, with \(\epsilon < \frac{T}{2}\). Let \(\phi \in H^0\left(\mathbb{R}^{+}\right) = L^2\left(\mathbb{R}^{+}\right)\) and \(\psi \in H_{-b}^0\left(\mathbb{R}^{+}\right)\). Then, there exists a solution
$$
u \in L_{\mathrm{loc}}^2([0, \infty) \times [0, T]) \cap C\left([0, \epsilon]; L^2\left(\mathbb{R}^{+}\right)\right) \cap C\left([T-\epsilon, T]; H_{-b}^0\left(\mathbb{R}^{+}\right)\right)
$$
which satisfies
$$
\left\{\begin{array}{l}
\partial_t u+\partial_x u+\partial_x^3 u=0, \quad \text{in} \ \mathcal{D}^{\prime}((0,+\infty) \times(0, T)) \\
u\big|_{t=0}=\phi, \quad u\big|_{t=T}=\psi.
\end{array}\right.
$$
}

Rosier \cite[Corollary 2]{Rosier2002} constructs a fundamental solution with compact time support for the linear KdV equation, yielding an explicit trajectory. However, with undesirable behavior at infinity, that drives any $u_0$ to zero. Summarizing, \cite[Theorem 1]{Rosier2002} establishes that the linear KdV equation is exactly boundary controllable on $(0,+\infty)$.


\vspace{0.2cm}
\noindent \textbf{Theorem D.} \textit{Let $T>0$ and let $u_0, u_T \in L^2(0,+\infty)$. Then there exists a function $u \in L_{l o c}^2([0, T] \times[0,+\infty))$ fulfilling
$$
\left\{\begin{aligned}
\partial_t u+\partial_x u+\partial_x^3 u & =0, \quad \quad t\in(0,T), \ \ x\in\mathbb{R}^+, \\
u_{\mid t=0} & =u_0, \\
u_{\left.\right|_{t=T}} & =u_T.
\end{aligned}\right.
$$}

It is important to point out that Theorems B, C, and D are related only to the linear equation. In this context, the control theory for the linear and nonlinear KdV equation still needs to be developed if compared with the theory in bounded domains. So, given these previous results, the study of exact controllability for systems \eqref{h1} and \eqref{h2} in the half-line framework becomes an interesting problem. Specifically, the following question arises:

\vspace{0.2cm}
\noindent\textbf{Question \(\mathcal{A'}\):} \textit{Given \(T > 0\) and \(\phi, \phi_T \in L^2(\R^+)\) (or $L^2(\R^-)$). Can one find an appropriate control input \(f(t) \in H^{\frac{1}{3}}(0, T)\) (or \(g_1(t) \in H^{\frac{1}{3}}(0, T)\) or \(g_2(t) \in L^2(0, T)\)) such that there exists a class of corresponding solutions \(u(x, t)\) of system \eqref{h1} (or \eqref{h2}) satisfying $u(x, 0) = \phi(x)$ and $\quad u(x, T) = \phi_T(x)$?}

\subsection{State of the arts}  Research on the IBVP \eqref{h1} began with Ton \cite{Ton}, who established the existence and uniqueness of solutions for smooth initial data with zero boundary conditions. Bona and Winther \cite{BoWi, BoWi1} extended these results, proving global well-posedness in $H^4(\mathbb{R}^{+})$ and continuous dependence, while Faminskii \cite{Faminskii1a} obtained well-posedness in weighted Sobolev spaces. Later, Bona \textit{et al.} \cite{Bonahalf1} proved conditional local well-posedness under $s$-compatibility conditions for $\psi \in H^s(\mathbb{R}^+)$ and $f \in H^{(s+1)/3}(\mathbb{R}^+)$ with $s>\tfrac{3}{4}$, and global results for $1\leq s\leq 3$.

Parallel developments followed with Colliander and Kenig \cite{CK}, who proposed a unified framework for the $\mathrm{gKdV}$ on $\mathbb{R}^+$, later refined by Faminskii \cite{Faminskii}. Holmer \cite{Holmer} and Bona \textit{et al.} \cite{bonahalf3} pushed the theory to the critical Sobolev exponent $s=-\tfrac{3}{4}$, while Faminskii \cite{Faminskii2} established global well-posedness for $s\geq 0$ under natural boundary conditions. Fokas \cite{fokas} introduced the unified transform method (UTM), extending IST techniques to IBVPs. Importantly, these frameworks \cite{Bonahalf1, CK, fokas, Holmer} also apply to the IBVP \eqref{h2}, and Cavalcante \cite{cav} adapted Colliander–Kenig’s approach to star graphs.
With respect to control theory, Russell and Zhang's works initiated the exploration of control strategies for the KdV equation \cite{russell3,russell2,Russell1,zhang2,zhang4}. As for the control issue, Rosier \cite{Rosier} examined boundary control of the KdV equation on the finite domain $(0,L)$ with  boundary conditions
\begin{equation}
\left\{
\begin{array}
[c]{lll}%
\partial_tu+ \partial_xu+\partial^3_xu+u\partial_xu=0, &  & \text{ in } (0,L)\times(0,T),\\
u(0,t)=u(L,t)=0,\text{ }\partial_xu(L,t)=g(t), &  & \text{ in }(0,T),\\
u(x,0)=u_0(x), & & \text{ in }(0,L),
\end{array}
\right.  \label{2}
\end{equation}
where the boundary value function $g(t)$ is considered as a control input, has important phenomena that directly affect the control problem related to them, so-called \textit{critical length phenomenon}. 
This phenomenon suggests that the exact controllability of the linear system associated with \eqref{2} hinges on the length $L$ of the spatial domain $(0,L)$. In other words, the linear system is controllable if and only if
\begin{equation}
L\notin\mathcal{N}:=\left\{  \frac{2\pi}{\sqrt{3}}\sqrt{k^{2}+kl+l^{2}}
\,:k,\,l\,\in\mathbb{N}^{\ast}\right\}  . \label{critical}
\end{equation}

In the case $L\in\mathcal{N}$, Rosier \cite{Rosier} proved that the linear system associated with \eqref{2} is not controllable, since a finite-dimensional subspace $\mathcal{M}(L)\subset L^2(0,L)$ remains unreachable. Specifically, for any nonzero $\psi\in\mathcal{M}$ and $g\in L^2(0,T)$, the solution $u\in C([0,T];L^2(0,L))\cap L^2(0,T;H^1(0,L))$ with $u(\cdot,0)=0$ cannot satisfy $u(\cdot,T)=\psi$. Thus, $(0,L)$ is called a \emph{critical domain} when $L\in\mathcal{N}$. After that, Coron and Crépeau \cite{coron} showed that the nonlinear system \eqref{2} remains locally exactly controllable for certain critical lengths $L=2k\pi$, provided that no integers $(m,n)\in\mathbb{N}^\ast\times\mathbb{N}^\ast$ with $m\neq n$ satisfy $m^2+mn+n^2=3k^2$. Later, Cerpa \cite{cerpa1} and Cerpa–Crépeau \cite{cerpa2} extended this result, proving local controllability for longer times at critical lengths.

Depending on the position where the control is applied at the boundary, the critical set cannot be explicitly characterized \cite{cerpatut,GG1}. Although several works \cite{GG1,CeRiZh} explored the phenomenon for alternative boundary conditions, the exact structure of the critical set remained unknown. Two decades later, Capistrano-Filho and da Silva \cite{CaSi} advanced the theory by analyzing the KdV equation on $[0,L]$ with Neumann-type conditions and one control input. They proved controllability even in the critical case $L\in\mathcal{R}_\beta$, employing the \textit{return method} and a fixed-point argument.

Beyond bounded domains, Muñoz \cite{munoz} studied approximate control of solitons for generalized KdV equations on $\mathbb{R}$, proving local null controllability and the possibility of accelerating solitons to prescribed velocities with quantitative error estimates. Finally, Ozsari \cite{Ozsari2020} addressed the lack of null controllability for the heat equation on the half-line, introducing an approach based on the unified transform method (UTM) \cite{fokas}, which also plays a key role in recent controllability results for KdV-type models.

\subsection{Operational controllability} We study the controllability of the KdV equation on the half-line using the Hilbert Uniqueness Method (HUM) introduced by Lions \cite{Lions}. We introduce a method for explicitly characterizing controls, termed \textit{operational controllability}, based on the derivation of global relations between the final state and the control through integral representations of solutions on the half-line and on finite intervals. This approach not only enables the application of HUM but also offers a framework potentially adaptable to other nonlinear evolution PDEs, extending ideas explored in recent works such as Fokas \cite{fokas}, Kenig \textit{et al.} \cite{CK}, Holmer \cite{Holmer}, and Bona \textit{et al.} \cite{Bonahalf1}. For the sake of simplicity, we will consider the exact controllability problem \eqref{h1}; the other cases in this work can be treated analogously. An important point is that, without any loss of generality, we shall consider only the case when the initial data $\phi=0$.  

\subsubsection{Operational controllability via forcing operator} We will use the representation introduced by Kenig \textit{et al.} \cite{CK}, and after explored by Holmer \cite{Holmer}.  Suppose that the initial data $\phi=0$. Then, the integral representation of the solution of \eqref{h1} given by the forcing operator takes the form
\begin{equation*}
u(x, t)=-\frac{1}{2} \theta(t) \mathscr{D} \partial_{x} u^2(x, t)+\theta(t) \mathscr{L}_{+}^\lambda \left( e^{-\pi i \lambda} f \right)(x,t) +\theta(t) \mathscr{L}_{+}^\lambda \left(\frac{1}{2} \theta(t) \mathscr{D} \partial_x u^2(0, t)\right)(x,t),
\end{equation*}
where the operators $\mathscr{D}$ and $\mathscr{L}^{\lambda}_{+}$ are defined in the \eqref{inhomo} and \eqref{forcingmas}, respectively. (See Appendix \ref{app1}). Now, if we consider the linearized system with initial data zero, the solution could be written as
\begin{align*}
u(x, t)=\theta(t) \mathscr{L}_{+}^\lambda \left( e^{-\pi i \lambda} f\right)(x,t).
\end{align*}
Thus, we define the nature of the control that drives the linearized system associated with \eqref{h1} to satisfy the exact controllability condition.
\begin{definition}\label{def-a}Let $\phi_T\in L^2(\R_x^+)$, thus the system 
\begin{equation}\label{eqnew21}
\begin{cases}\partial_t u+ \partial_x u+\partial_x^3 u=0, & \text { for }(x, t) \in(0,+\infty) \times(0, T), \\ u(0, t)=f(t), & \text { for } t \in(0, T), \\ u(x, 0)=0, & \text { for } x \in(0,+\infty),
\end{cases}
\end{equation}
is operational exact controllable at time $t = T$ (via forcing operator) if and only if the system \eqref{eqnew21} is exactly controllable with a control $f = f(t) \in H^{\frac{1}{3}}(\R^+_t)$ given by the relation
\begin{align}\label{eqnew16}
\theta(T) \mathscr{L}_{+}^\lambda \left( e^{-\pi i \lambda} f(t) \right)(x,T)= \phi_T(x),\quad \text{for every $\lambda\in(-1,1)$.}
\end{align}
\end{definition}

\begin{remark} Concerning the previous definition, we have the following:
\begin{itemize}
\item[i.]From the identities \eqref{eqnew17} and \eqref{forcingmas}, the relation \eqref{eqnew16}  is equivalent to
\begin{align}\label{eqnew19}
 \frac{3}{\Gamma(\lambda)} \int_{x}^{\infty}\int_0^T (y-x)^{\lambda-1}  A\left(\frac{y}{\left(T-t^{\prime}\right)^{1 / 3}}\right) \frac{\mathcal{I}_{-\frac23 - \frac{\lambda}{3}} f\left(t^{\prime}\right)}{\left(T-t^{\prime}\right)^{1 / 3}} d t^{\prime} d y &=\phi_T(x),  
\end{align}
for some $\lambda\in\mathbb{R}$, where $A$ is the Airy function given in \eqref{airyfunction}. It is crucial to emphasize that the presence of the parameter $\lambda$ in the relation above must adhere to the conditions ensuring the well-posedness of the system under consideration. 
\item[ii.] In Definition \ref{def-a}, operational controllability means that, given the final data, we obtain an explicit representation formula \eqref{eqnew19}, in this case through a forcing operator, that satisfies the exact controllability condition, that is, $u(x,T)=\phi_T(x)$.
\end{itemize} 
\end{remark}

\begin{remark}\label{rmk2} Observe that we can extend the previous definition for two classes of operators: the boundary operator \cite{Bonahalf1} and the UTM operator \cite{fokas}. Precisely, we have the following:
\begin{itemize}
\item [(a)] \textbf{Operational controllability via boundary operator:} Using the representation giving by Bona \textit{et al.} \cite{Bonahalf1},  we can establish the operational controllability of the linear system \eqref{eqnew21}. Note that \cite[Proposition 2.2]{Bonahalf1} implies that the solution of the system \eqref{eqnew21} is given by 
$$
u(x, t)=\left[W_b(t) f\right](x)=\left[U_b(t) f\right](x)+\overline{\left[U_b(t) f\right](x)} \quad \text{in $H^s(\R^+_x)$, for $\frac34 < s \leq 3$}.
$$
Here, for $x \geq $, $t \geq 0$,
$$
\left[U_b(t) f\right](x)=\frac{1}{2 \pi} \int_1^{\infty} e^{i \mu^3 t-i \mu t} e^{-\left(\frac{\sqrt{3 \mu^2-4}+i \mu}{2}\right) x}\left(3 \mu^2-1\right) \int_0^{\infty} e^{-\left(\mu^3 i-i \mu\right) \xi} f(\xi) d \xi d \mu.
$$
Thus, the control solving the exact controllability problem for the linearized system \eqref{eqnew21} is defined as follows.
\begin{definition}\label{def-b} Let $\phi_T\in H^s(\R_x^+)$, thus the system is operational exact controllable at time $t = T$ (via boundary operator)  if and only if the system \eqref{eqnew21} is exactly controllable with a control $f = f(t) \in H^{\frac{s+1}{3}}(\R^+_t)$ given by the relation
\begin{equation*}
Re \left(\int_1^{\infty}\int_0^{\infty} e^{i \mu^3 T-i \mu T} e^{-\left(\frac{\sqrt{3 \mu^2-4}+i \mu}{2}\right) x}\left(3 \mu^2-1\right)  e^{-\left(\mu^3 i-i \mu\right) \xi} f(\xi) d \xi d \mu\right)  =\pi \phi_T(x).
\end{equation*}
\end{definition}
\item[(b)]\textbf{Operational controllability via UTM operator:}  Following the Fokas approach \cite{fokas,fokas0,fokas2,himonas2022}, we can consider the case of $H^s(\R^+_x)$, for $-\frac34 < s $ with $\frac{s+1}{3}\notin \mathbb{N}+\frac12$, and we can consider the operational controllability by using the following integral formula\footnote{Note that the terms $(3 k^2-1)$ in the formula \eqref{eqnew20} arises from the dispersion relation associated with the linear KdV operator after applying the unified transform method and solving the corresponding global relation in Fourier space. For details, see for instance \cite[Chapter 1, Proposition 1.2 and Example 1.12]{fokas}.}
\begin{equation}\label{eqnew20}
\begin{aligned}
& u(x, t)=S\left[0, f, 0\right](x, t)  = -\frac{1}{2 \pi} \int_{\partial D^{+}_R} \mathrm{e}^{\mathrm{i} k x-\mathrm{i}( k-k^3) t } (3 k^2-1) \tilde{f}\left(k-k^3, T\right) d k, 
\end{aligned}
\end{equation}
where $\tilde{f}_0\left(k^3, T\right)$ are defined by 
$$
\tilde{f}\left(k^3, T\right)=\int_0^T \mathrm{e}^{\mathrm{i}(k- k^3) t^{\prime}} f\left(t^{\prime}\right)d t^{\prime}, \quad k \in \mathbb{C}, 0<t<T,
$$
and 
$\mathbb{C}^{+}$ and $\mathbb{C}^{-}$ will denote the upper half $(\operatorname{Im} k>0)$ and the lower half $(\operatorname{Im} k<0)$ of the complex $k$-plane. The domain $D$ is defined by
$$
D=\{k \in \mathbb{C}, \ \operatorname{Re} (ik-ik^3)<0\}=\{k \in \mathbb{C}, \   (\operatorname{Im}k)\left( 3 ( \operatorname{Re}k)^2-  (\operatorname{Im}k)^2 -1 \right) <0\} .
$$
Moreover, $D^{+}$ and $D^{-}$ will denote the part of $D$ in $\mathbb{C}^{+}$ and $\mathbb{C}^{-}$, namely
\begin{equation*}
D^{+}=D \cap \mathbb{C}^{+}\quad \text{and} \quad D^{-}=D \cap \mathbb{C}^{-} .
\end{equation*}
The asymptotic form of $D, D^{+}, D^{-}$as $k \rightarrow \infty$  will be denoted by $D_R, D_R^{+}, D_R^{-}$, respectively, i.e.,
\begin{equation}\label{regionD}
\begin{gathered}
D_R=\{k \in D, \quad|w(k)|>R, \quad R \text { large }\}, \quad
D_R^{+}=D_R \cap \mathbb{C}^{+}, \quad D_R^{-}=D_R \cap \mathbb{C}^{-}.
\end{gathered}
\end{equation}
From \cite[Theorem 2]{fokas2}, the function $u(x,t)$ given by \eqref{eqnew20} defines a solution to the system  \eqref{eqnew21} in the space $C([0,T];H^s(\R^+_x))$. Thus, in this case,  the control solving the exact controllability problem for the linearized system \eqref{eqnew21} is defined as follows.

\begin{definition}\label{def-c}Pick $\phi_T\in H^s(\R_x^+)$, the system is operational exact controllable at time $t = T$ (via UTM operator) if and only if (via forcing operator) the system \eqref{eqnew21} is exactly controllable with a control $f = f(t) \in H^{\frac{s+1}{3}}(\R^+_t)$ given by the relation
\begin{equation*}
 \int_{\partial D^{+}_R}\int_0^T \mathrm{e}^{\mathrm{i} k x-\mathrm{i} (k-k^3) (T-t^{\prime})} (3 k^2-1)  f\left(t^{\prime}\right)d t^{\prime} d k=-2\pi \phi_T(x).
\end{equation*}
\end{definition}



 
\item[(c)] Note that in \cite{bonahalf3,fokas2}, the authors establish the uniqueness of the solution to \eqref{eqnew21} when $-3/4<s<1$. Due to this and taking into account the previous definitions, ``operational controllability" provides a framework to determine inputs that ensure exact controllability, characterized by the forcing, boundary, and UTM operators. Since these definitions are equivalent, the desired control can be obtained by combining their relations; for instance, if system \eqref{h1} is controllable through all three operators, the control function $f$ can be written as
$$
\left(  \left\lbrace \text{Forcing operator} \right\rbrace
f
 \right)(x,t) + 
 \left(  \left\lbrace \text{UTM operator} \right\rbrace
f
 \right)(x,t) = 0,$$
 $$
 \left(  \left\lbrace \text{Forcing operator} \right\rbrace
f
 \right)(x,t) -
 \left(  \left\lbrace \text{Boundary operator} \right\rbrace
f
 \right)(x,t) = 0,$$
or
 $$
 \left(  \left\lbrace \text{Boundary operator} \right\rbrace
f
 \right)(x,t) + 
 \left(  \left\lbrace \text{UTM operator} \right\rbrace
f
 \right)(x,t) = 0.$$
 Therefore, if the inverses of the operators introduced above are well-defined, we may conclude that the control input $f(t)$ is given by:
$$
f(t)  \approx  \left[ \left(  \left\lbrace \text{Forcing operator} \right\rbrace
   \right) \right]^{-1}(\phi_T(x))(t),$$
   or
  $$
  f(t)  \approx  \left[ \left(   \left\lbrace \text{Boundary operator} \right\rbrace \right) \right]^{-1}(\phi_T(x))(t),$$
  or
$$
f(t) \approx  \left[ \left(  \left\lbrace \text{UTM operator} \right\rbrace \right) \right]^{-1}(\phi_T(x))(t).
$$
\end{itemize}
\end{remark}
\subsection{Notations and main results} So far, no results address the controllability of the KdV equation on the half-line. Here, we relate the IBVP with controllability and propose an approach applicable not only to the KdV but also to other dispersive equations on unbounded domains, namely $\mathbb{R}^+=(0,\infty)$ and $\mathbb{R}^-=(-\infty,0)$. For the well-posedness and control analysis (via forcing operator), we first consider equation \eqref{h1} in the range $-\tfrac{3}{4}<s<\tfrac{3}{2}$ with $s\neq \tfrac{1}{2}$.

 In this set of conditions over $s$,  we have that
\begin{equation}\label{h3}
\phi \in H^s\left(\mathbb{R}^{+}_x\right), \quad f \in H^{\frac{s+1}{3}}\left(\mathbb{R}^{+}_t\right), \quad \text{and if}\ \  \frac{1}{2}<s<\frac{3}{2}, \ \phi(0)=f(0).
\end{equation}
Moreover, considering the system \eqref{h2}, if $-\frac{3}{4}<s<\frac{3}{2},$ for $s \neq \frac{1}{2}$, we get
\begin{equation}\label{h4}
\phi \in H^s\left(\mathbb{R}^{-}_x\right), g_1 \in H^{\frac{s+1}{3}}\left(\mathbb{R}^{+}_t\right), g_2 \in H^{\frac{s}{3}}\left(\mathbb{R}^{+}_t\right),\quad \text {and if }\ \ \frac{1}{2}<s<\frac{3}{2},\ \phi(0)=g_1(0) .
\end{equation}
Let $X$ denote the modified Bourgain space $X_{s, b} \cap D_\alpha$ with $b<\frac{1}{2}$ and $\alpha>\frac{1}{2}$, where
$$
\|u\|_{X_{s, b}}  =\left(\iint_{\xi, \tau}\langle\xi\rangle^{2 s}\left\langle\tau-\xi^3\right\rangle^{2 b}|\hat{u}(\xi, \tau)|^2 d \xi d \tau\right)^{1 / 2},$$
and
$$
\|u\|_{D_\alpha}  =\left(\iint_{|\xi| \leq 1}\langle\tau\rangle^{2 \alpha}|\hat{u}(\xi, \tau)|^2 d \xi d \tau\right)^{1 / 2} .
$$

With this in hand, we define the solution for the systems \eqref{h1} and \eqref{h2}, respectively, as follows.
\begin{definition}\label{def_1} We will call a solution $u(x, t)$  of \eqref{h1}-\eqref{h3} (resp. \eqref{h2}-\eqref{h4}) on $[0, T]$ if the following holds:
\begin{itemize}
\item[a)] \underline{Well-defined nonlinearity:} The function for some appropriate space $X$, if $u \in X$, then $\partial_x u^2$ is a well-defined in the sense of distribution. Moreover, the function $u(x, t)$ satisfies system \eqref{h1} (resp. \eqref{h2}) in the sense of distributions on the set $(x, t) \in(0,+\infty) \times(0, T)$ (resp. $(x, t) \in(-\infty, 0) \times(0, T)$).
\item[b)] \underline{Space traces:} The function $u \in C\left([0, T] ; H_x^s\right)$ and in this sense $u(\cdot, 0)=\phi$ in $H^s\left(\mathbb{R}^{+}_x\right)$ (resp. $u(\cdot, 0)=\phi$ in $H^s\left(\mathbb{R}^{-}_x\right)$).
\item[c)] \underline{Time traces:} Considering $u \in C\left(\mathbb{R}_x ; H^{\frac{s+1}{3}}(0, T)\right)$ and in this sense $u(0, \cdot)=f$ in $H^{\frac{s+1}{3}}(0, T)$ (resp. $u(0, \cdot)=g_1$ in $H^{\frac{s+1}{3}}(0, T)$).
\item[d)] \underline{Derivative traces:} If $\partial_x u \in C\left(\mathbb{R}_x ; H^{\frac{s}{3}}(0, T)\right)$, considering only the system \eqref{h2}-\eqref{h4} we require that, in this sense, $u(0, \cdot)=g_2$ in $H^{\frac{s}{3}}(0, T)$.
\end{itemize}
\end{definition}

Let us now introduce two sets. For any $\phi \in H^s(\R^+)$, $-\frac{3}{4}<s\leq \frac{3}{2}$, we define the admissible final state class for the linearized systems \eqref{h1} and \eqref{h2} by the sets
\begin{multline*}
    \mathcal{A}^s_r(\phi,T)=\left\lbrace \phi_T \in H^s(\R^+): \text{the solution given by Definition \ref{def_1} for the linear system associated} \right.\\
    \left.\text{with \eqref{h1} satisfies $u(x,T)=\phi_T$, for a boundary control $f \in H^{\frac{s+1}{3}}(\R^+)$}\right\rbrace
\end{multline*}
and
\begin{multline*}
    \mathcal{A}^s_l(\phi,T)=\left\lbrace \phi_T \in H^s(\R^-): \text{the solution given by Definition \ref{def_1} for the linear system associated } \right.\\
    \left.\text{with \eqref{h2} satisfies $u(x,T)=\phi_T$, for a boundary controls $g_1 \in H^{\frac{s+1}{3}}(\R^+), \ g_2 \in H^{\frac{s}{3}}(\R^+)$}\right\rbrace,
\end{multline*}
respectively. Note that both sets are non-empty (see this discussion at Section \ref{sec6}).


Now, consider the following ball in $X:=X_{0,b}\cap D_{\alpha}$ given by 
$$
B_r=\left\lbrace u \in X_{0,b}\cap D_{\alpha}: \|u\|_{X_{0,b}\cap D_{\alpha}} \leq r\right\rbrace.
$$
With all the previous frameworks, we present two results answering Question $\mathcal{A}’$ and, consequently, Question $\mathcal{A}$. 
\begin{theorem}\label{main-a}
Let $T>0$. Then, there exists $\delta>0$ such that for any $\phi \in L^2(\R^+_x)$ and $ \phi_T  \in \mathcal{A}_r^0(\phi,T)$, verifying $
\|\phi\|_{L^2(\R^+_x)}+\|\phi_T\|_{L^2(\R^+_x)}\leq \delta,$  the system \eqref{h1} admits a unique solution $u \in B_r\subset  X=X_{0,b}\cap D_{\alpha}$ operational exactly controllable at time $T$, it means that there exist $f\in H^{\frac{1}{3}}(\R^+_t)$ such that
 \begin{equation*}
\phi_T(x)=e^{-T (\partial_x+\partial_x^3)} \phi(x)-\frac{1}{2}  \mathscr{D} \partial_{x} u^2(x, T)+ \mathscr{L}_{+}^\lambda h(x, T,\phi), 
\end{equation*}
where
\begin{equation*}
h(t,\phi)=e^{-\pi i \lambda}\left[f(t)-\left.\theta(t) e^{-t(\partial_x+ \partial_x^3)} \phi\right|_{x=0}+\frac{1}{2} \theta(t) \mathscr{D} \partial_x u^2(0, t)\right], \quad \text{for every $\lambda\in(-1,1)$,}
\end{equation*}
with $ \mathscr{L}_{+}^{\lambda}$ the forcing operator given by \eqref{forcingmenos}.
\end{theorem}

The second main result considers the control problem \eqref{h2} when only one control input is available. The result is the following. 
\begin{theorem}\label{main-b}
Let $T>0$. Then, there exists $\delta>0$ such that for any $\phi\in L^2(\R^-_x)$ and  $\phi_T  \in \mathcal{A}_l^0(\phi,T)$, verifying
$\|\phi\|_{L^2(\R^-_x)}+\|\phi_T\|_{L^2(\R^-_x)}\leq \delta,$ the system \eqref{h2} admits a unique solution $u \in B_r\subset  X=X_{0,b}\cap D_{\alpha}$ operational exactly controllable at time $T$. Moreover, we have:
\begin{enumerate}
\item[(i)]  If $g_1=0$, one can find  $g_2 \in L^2(\R^+_t)$ such that 
\begin{equation*}
    \phi_T(x)= e^{-T (\partial_x+\partial_x^3)} \phi(x)-\frac{1}{2}  \mathscr{D} \partial_x u^2(x, T)+ \mathscr{L}_{-}^{\lambda_1} h_1(x, T,\phi)+\mathscr{L}_{-}^{\lambda_2} h_2(x, T,\phi)
\end{equation*}
where
\begin{equation*}
\left[\begin{array}{l}
h_1(t,\phi) \\
h_2(t,\phi)
\end{array}\right]=M\left[\begin{array}{c}
-\left.\theta(t) e^{-t (\partial_x+ \partial_x^3)} \phi\right|_{x=0}+\frac{1}{2}  \mathscr{D} \partial_x u^2(0, t) \\
\theta(t) \mathcal{I}_{1 / 3}\left(g_2-\left.\theta \partial_x e^{-t (\partial_x+\partial_x^3)} \phi\right|_{x=0}+\frac{1}{2} \theta \partial_x \mathscr{D} \partial_x u^2(0, \cdot)\right)(t)
\end{array}\right].
\end{equation*}

\item[(ii)]  If $g_2=0$, one can find  $g_1 \in H^{\frac{1}{3}}(\R^+_t)$ such that 
\begin{equation*}
    \phi_T(x)= e^{-T (\partial_x+\partial_x^3)} \phi(x)-\frac{1}{2}  \mathscr{D} \partial_x u^2(x, T)+ \mathscr{L}_{-}^{\lambda_1} h_1(x, T,\phi)+ \mathscr{L}_{-}^{\lambda_2} h_2(x, T,\phi),
\end{equation*}
where
\begin{equation*}
\left[\begin{array}{l}
h_1(t,\phi) \\
h_2(t,\phi)
\end{array}\right]=M\left[\begin{array}{c}
g_1(t)-\left.\theta(t) e^{-t (\partial_x+ \partial_x^3)} \phi\right|_{x=0}+\frac{1}{2}  \mathscr{D} \partial_x u^2(0, t) \\
\theta(t) \mathcal{I}_{1 / 3}\left(-\left.\theta \partial_x e^{-t (\partial_x+\partial_x^3)} \phi\right|_{x=0}+\frac{1}{2} \theta \partial_x \mathscr{D} \partial_x u^2(0, \cdot)\right)(t)
\end{array}\right].
\end{equation*}
\end{enumerate}
Here, $M$ is a matrix given by \eqref{matrixA}.
\end{theorem}

Both theorems follow from classical control tools: the Dolecki–Russell duality \cite{DoRu} in Lions’ framework \cite{Lions}, reducing the problem to an observability inequality, and a fixed-point argument for the full system. Moreover, the explicit solution formulas \cite{Bonahalf1,fokas,Holmer} allow us to precisely characterize the control and controllable solutions, in line with Definitions \ref{def-a}, \ref{def-b}, and \ref{def-c}, as follows:
\begin{remark}\label{main-c} 
Let $\frac{3}{4} < s <1$ and $f \in H^{\frac13}(0,T)$ such that system 
\begin{equation*}
\begin{cases}\partial_t u+ \partial_x u+\partial_x^3 u=0, & \text { for }(x, t) \in(0,+\infty) \times(0, T), \\ u(0, t)=f(t), & \text { for } t \in(0, T), \\ u(x, 0)=0, & \text { for } x \in(0,+\infty), 
\end{cases}
\end{equation*}
is (operational) exactly controllable at time $T$, via forcing operator, boundary operator, and UTM operator (see Definitions \ref{def-a}, \ref{def-b} and \ref{def-c}). Then, the control $f$ satisfies the following relations (see Remark \ref{rmk2} item (c))
\begin{multline*}
\frac{6\pi}{\Gamma(\lambda)} \int_{x}^{\infty}\int_0^T (y-x)^{\lambda-1}  A\left(\frac{y}{\left(T-t^{\prime}\right)^{1 / 3}}\right) \frac{\mathcal{I}_{-\frac23 - \frac{\lambda}{3}} f\left(t^{\prime}\right)}{\left(T-t^{\prime}\right)^{1 / 3}} d t^{\prime} d y \\
-2Re \left(\int_1^{\infty}\int_0^{\infty} e^{i \mu^3 T-i \mu T} e^{-\left(\frac{\sqrt{3 \mu^2-4}+i \mu}{2}\right) x}\left(3 \mu^2-1\right)  e^{-\left(\mu^3 i-i \mu\right) \xi} f(\xi) d \xi d \mu\right)=0
\end{multline*}
or 
\begin{multline*}
2Re \left(\int_1^{\infty}\int_0^{\infty} e^{i \mu^3 T-i \mu T} e^{-\left(\frac{\sqrt{3 \mu^2-4}+i \mu}{2}\right) x}\left(3 \mu^2-1\right)  e^{-\left(\mu^3 i-i \mu\right) \xi} f(\xi) d \xi d \mu\right) \\
+ \int_{\partial D^{+}_R}\int_0^T \mathrm{e}^{\mathrm{i} k x-\mathrm{i} (k-k^3) (T-t^{\prime})} (3 k^2-1)  f\left(t^{\prime}\right)d t^{\prime} d k = 0 
\end{multline*}
or 
\begin{multline*}
 \frac{6\pi}{\Gamma(\lambda)} \int_{x}^{\infty}\int_0^T (y-x)^{\lambda-1}  A\left(\frac{y}{\left(T-t^{\prime}\right)^{1 / 3}}\right) \frac{\mathcal{I}_{-\frac23 - \frac{\lambda}{3}} f\left(t^{\prime}\right)}{\left(T-t^{\prime}\right)^{1 / 3}} d t^{\prime} d y \\
 +   \int_{\partial D^{+}_R}\int_0^T \mathrm{e}^{\mathrm{i} k x-\mathrm{i} (k-k^3) (T-t^{\prime})} (3 k^2-1)  f\left(t^{\prime}\right)d t^{\prime} d k = 0 
\end{multline*}
Here, $A$ is the Airy function given by \eqref{airyfunction} and $\mathcal{I}_{-\frac23 - \frac{\lambda}{3}}$ is the Caputo fractional derivative given by \eqref{airyfunction} and \eqref{eq:IO}, respectively. Moreover, the region $D^+_R$ is defined by \eqref{regionD}. 
\end{remark}
\subsection{Paper's outline}  We conclude the introduction with the paper’s structure: Section \ref{sec2} reviews well-posedness theory; Sections \ref{sec3} and \ref{sec4} prove Theorems \ref{main-a} and \ref{main-b}, establishing exact controllability on the right and left half-lines, respectively; Section \ref{sec6} discusses perspectives and open problems. Appendix \ref{app1} gives boundary forcing operator formulas and key estimates.

\section{Well-posedness theory}\label{sec2}
In this section, we are interested in revisiting the well-posedness theory presented in the classical papers \cite{Bonahalf1,Bonahalf2,bonahalf3,CK,fokas,fokas2,Holmer}. These preliminary analyses will be paramount for us to present the main novelty of this work.
\subsection{Preliminaries} Let us consider $s\geq 0$. In this case,  we say that $\phi \in H^s(\mathbb{R}^+)$ if exists $\tilde{\phi}\in H^s(\mathbb{R})$ such that $\phi=\tilde{\phi}|_{\R+}$.  Moreover, we set $$\|\phi\|_{H^s(\mathbb{R}^+)}:=\inf\limits_{\tilde{\phi}}\|\tilde{\phi}\|_{H^{s}(\mathbb{R})}.$$ Here,  $$H_0^s(\mathbb{R}^+)=\Big\{\phi \in H^{s}(\mathbb{R}^+);\,\text{supp} (\phi) \subset[0,+\infty) \Big\},$$ otherwise, that is, $s<0$, define $H^s(\mathbb{R}^+)$ and $H_0^s(\mathbb{R}^+)$  as the dual space of $H_0^{-s}(\mathbb{R}^+)$ and  $H^{-s}(\mathbb{R}^+)$, respectively. The first results summarize useful properties of the Sobolev spaces on the half-line, and the proofs can be found in \cite{CK}.

\begin{lemma}
	For all $f\in H^s(\mathbb{R})$  with $-\frac{1}{2}<s<\frac{1}{2}$ we have
	\begin{equation*}
	\|\chi_{(0,+\infty)}f\|_{H^s(\mathbb{R})}\lesssim  \|f\|_{H^s(\mathbb{R})}.
	\end{equation*}
\end{lemma}
\begin{lemma}
	If $\frac{1}{2}<s<\frac{3}{2}$ the following statements are valid:
	\begin{enumerate}
		\item [(a)] $H_0^s(\R^+)=\big\{f\in H^s(\R^+);f(0)=0\big\},$\medskip
		\item [(b)] If  $f\in H^s(\R^+)$ with $f(0)=0$, then $\|\chi_{(0,+\infty)}f\|_{H_0^s(\R^+)}\lesssim \|f\|_{H^s(\R^+)}$.
	\end{enumerate}
\end{lemma}

\begin{lemma}
	If $f\in  H_0^s(\mathbb{R}^+)$ with $s\in \R$, we then have
		\begin{equation*}
	\|\psi f\|_{H_0^s(\mathbb{R}^+)}\lesssim \|f\|_{H_0^s(\mathbb{R}^+)}.
	\end{equation*}
	
\end{lemma}

Let us consider the well-known Bourgain theory \cite{Bourgain}. We denote by $X^{s,b}$ the Fourier transform space associated with the linear KdV equation, precisely, the space $X^{s,b}$  is the completion of $S'(\mathbb{R}^2)$ concerning the norm
\begin{equation*}\label{Bourgain-norm}
\|w\|_{X^{s,b}(\phi)}=\|\langle\xi\rangle^s\langle\tau-\xi^3\rangle^b\hat{w}(\xi,\tau) \|_{L_{\tau}^2L^2_{\xi}}.
\end{equation*}
To obtain our results, we also need to define the following auxiliary modified Bougain space. Let  $U^{s,b}$ and $V^{\alpha}$ the completion of $S'(\R^2)$ with respect to the norms
$$
\|w\|_{U^{s,b}}=\left(\int\int \langle \tau\rangle^{2s/3} \langle \tau-\xi^3\rangle^{2b} |\widehat{w}(\xi,\tau)|^2d\xi d\tau\right)^{\frac{1}{2}}\quad \text{and} \quad \|w\|_{V^{\alpha}}=\left(\int\int  \langle \tau\rangle^{2\alpha} |\widehat{w}(\xi,\tau)|^2d\xi d\tau\right)^{\frac{1}{2}}.
$$
Next, nonlinear estimates, in the context of the KdV equation, for $b<\frac{1}{2}$, were derived by Holmer in \cite{Holmer}.
\begin{lemma} The following holds:
	\begin{itemize}
		\item [(a)] For $s>-\frac{3}{4}$, there exists $b=b(s)<\frac12$ such that for all $\alpha>\frac{1}{2}$
		we have
		\begin{equation*}
		\big\|\partial_x (v_1v_2)\big\|_{X^{s,-b}}\lesssim \|v_1\|_{X^{s,b}\cap V^{\alpha}}\|v_2\|_{X^{s,b}\cap V^{\alpha}}.
		\end{equation*}
		\item[(b)] Considering  $s\in\left(-\frac{3}{4},3\right)$, there exists $b=b(s)<\frac12$ such that for all $\alpha>\frac{1}{2}$
		\begin{equation*}
		\big\|\partial_x (v_1v_2)\big\|_{X^{s,-b}}\lesssim\|v_1\|_{X^{s,b}\cap V^{\alpha}}\|v_2\|_{X^{s,b}\cap V^{\alpha}},
		\end{equation*}
is verified.		
	\end{itemize}
\end{lemma}
\subsection{Overview of well-posedness results}  In our analysis, we will consider the KdV equation posed in the unbounded domains, that is, the positive real line and the negative real line. 

Let us start considering the case of the positive real line. We firstly quote the work of Bona \textit{et al.} \cite{Bonahalf1,Bonahalf2} that considers the system \eqref{h1} in $\R^+$.  Notably, the conventional approach of eliminating the term $\partial_x u$ from the equation by transitioning to traveling coordinates comes with a substantial trade-off in the quarter-plane problem. Introducing a change of variables $v(x, t) = u(x + t, t)$ effectively eliminates the problematic term in the evolution equation. However, this transformation alters the landscape of the boundary condition, now expressed as $v(-t, t) = f(t)$ for $t \geq 0$. Consequently, the boundary condition is enforced at a dynamically shifting spatial point, framing the problem within the unconventional domain $\{(x, t): t \geq 0, x + t\}.$ 

Essentially, in \cite{Bonahalf1} the authors establish a Kato smoothing effect in the following form: For $s>\frac{3}{4}$, if $\phi \in H^s\left(\R^{+}\right)$and $f \in H_{l o c}^{\frac{s+1}{3}}\left(\R^{+}\right)$ satisfy certain compatibility conditions at $(x, t)=(0,0)$, then the IBVP \eqref{h1} admits a unique solution
$$
u \in C\left(0, T ; H^s\left(\R^{+}_x\right)\right) \cap L^2\left(0, T ; H_{l o c}^{s+1}\left(\R^{+}_x\right)\right),
$$
which satisfies the following additional properties
$$
\left(\sup _{0<x<+\infty} \int_0^T\left|\partial_x^{s+1} u(x, t)\right|^2 d t\right)^{\frac{1}{2}} \leq C\left(\|\phi\|_{H^s\left(\R^{+}_x\right)}+\|f\|_{H^{\frac{s+1}{3}(0, T)}}\right).
$$
On the other hand, in \cite{bonahalf3}, the authors proved a boundary smoothing property for the corresponding linear problem, where the $u \partial_{x} u$ term is dropped and $\phi(x)=0$, which states
$$
\|u\|_{L^2\left([0, T] ; H^{s+\frac{3}{2}}\left(\mathbb{R}^{+}_X\right)\right)} \leq \mathrm{c}\|f\|_{\mathrm{H}^{\frac{s+1}{3}}\left(\mathbb{R}^{+}_t\right)},
$$
where $c=c(s, T)$. As an application of this property and variants of it, the authors obtain local well-posedness of mild solutions of the nonlinear problem for $\mathrm{s}>-\frac{3}{4}$, where mild solutions are defined as ones that can be appropriately approximated by smoother solutions. Next, we present the results that can be found in \cite{Bonahalf1} and \cite{bonahalf3}, respectively. 

\begin{theorem}
The initial-boundary-value problem \eqref{h1} is locally well-posed for initial data $\phi\in H^s\left(\R^{+}_x\right)$ and boundary data $f\in H_{\text {loc }}^{(s+1) / 3}\left(\R^{+}_t\right)$ satisfying certain compatibility conditions for $s>3 / 4$, whereas global well-posedness holds for $\phi \in H^s\left(\R^{+}_x\right), f \in H^{\frac{7+3 s}{12}}\left(\R^{+}_t\right)$ when $1 \leq s \leq 3$ and for $\phi \in H^s\left(\R^{+}_x\right)$, $f \in H_{\text {loc }}^{(s+1) / 3}\left(\R^{+}_t\right)$when $s \geq 3$. Furthermore, the corresponding solution map is an analytic correspondence between the space of initial and boundary data and the solution space.
\end{theorem}

\begin{theorem}
Let $s \geq-3/2$ and $T>0$ be given. There exists a constant $C$ such that for any $f \in H_0^{(s+1)/3}\left(\R^{+}_t\right)$, the corresponding solution $u$ of the linearized system associated to \eqref{h1} belongs to the space $L^2\left(0, T ; H_0^{s+3/2}\left(\R^{+}_x\right)\right)$ and satisfies
$$
\|u\|_{L^2\left(0, T ; H^{s+\frac{3}{2}}\left(\R^{+}_x\right)\right)} \leq C\|f\|_{H^{\frac{s+1}{3}}\left(\R^{+}_t\right)},
$$
for a constant $C$ depending only on $s$ and $T$.
\end{theorem}

To obtain the boundary controllability for the KdV posed in a half-line (positive or negative), it is necessary to consider a lower regularity in the initial data than  $s>\frac34$ as in the previous results. In this sense, Holmer \cite{Holmer} proves the existence of a solution of the KdV equation \eqref{h1} and \eqref{h2}, with $\beta=0$, posed either on a left half-line and right half-line. The main accomplishment of his work is to show that initial and boundary data may be given in Sobolev spaces of negative index. Indeed, the author demonstrated the existence of solutions for initial data in the Sobolev space $H^{\mathrm{s}}$ as long as $s$ is greater than $-\frac{3}{4}$, and similar restrictions apply to the boundary data. As is shown in the article, the right half-line problem requires one Dirichlet condition, while the left half-line problem requires an additional Neumann condition. On the finite interval, the problem is solved with two Dirichlet conditions and one Neumann condition on the right boundary. Now, to establish the well-posedness of systems \eqref{h1} and \eqref{h2}, for $-\frac34 < s < \frac32$ with $s\neq \frac12$. 

Now on, we will consider the case when $\beta=0$ and the presence of source function $h \in L^1(0,T,H^{s}(\Omega))$, where $\Omega=\R^+$ or $\Omega=\R^{-}$, namely
\begin{equation}\label{h1_1}
\begin{cases}\partial_t u+\partial_x^3 u=h, & \text { for }(x, t) \in(0,+\infty) \times(0, T), \\ u(0, t)=f(t), & \text { for } t \in(0, T), \\ u(x, 0)=\phi(x), & \text { for } x \in(0,+\infty),
\end{cases}
\end{equation}
and
\begin{equation}\label{h2_1}
\begin{cases}\partial_t u+\partial_x^3 u=h, & \text { for }(x, t) \in(-\infty, 0) \times(0, T), \\ u(0, t)=g_1(t), \quad \partial_x u(0, t)=g_2(t), & \text { for } t \in(0, T),\\ u(x, 0)=\phi(x),& \text { for } x \in(-\infty, 0).\end{cases}
\end{equation}
Thus, the well-posedness result for the linear systems, as presented in \cite{Holmer}, can be stated as follows.
\begin{theorem} \label{holmer} Let $-\frac{3}{4}<s<\frac{3}{2}, s \neq \frac{1}{2}$, and consider $h=0$ in the system \eqref{h1_1} and \eqref{h2_1}.
\begin{itemize}
\item[i.] Given $(\phi, f)$ satisfying \eqref{h3}, exist $T>0$, depending only on the norms of $\phi, f$ in \eqref{h3}, and $u(x, t)$ that is a mild and distributional solution to \eqref{h1_1}-\eqref{h3} on $[0, T]$.
\item[ii.] Given $\left(\phi, g_1, g_2\right)$ satisfying \eqref{h4}, exist $T>0$, depending only on the norms of $\phi, g_1, g_2$ in \eqref{h4}, and $u(x, t)$ a mild and distributional solution to  \eqref{h2_1}-\eqref{h4} on $[0, T]$.
\end{itemize}
\end{theorem}

Additionally, the results in \cite{Holmer} also ensure the well-posedness theory for the systems \eqref{h1} and \eqref{h2}, respectively. 
\begin{theorem}\label{localwellposedness} Let $-\frac{3}{4}<s<\frac{3}{2}, s \neq \frac{1}{2}$, it follows that
\begin{itemize}
\item[i.] Given $(\phi, f)$ satisfying \eqref{h3}, exist $T>0$, depending only on the norms of $\phi, f$ in \eqref{h3}, and $u(x, t)$ solution to \eqref{h1}-\eqref{h3} on $[0, T]$.
\item[ii.] Given $\left(\phi, g_1, g_2\right)$ satisfying \eqref{h4}, exist $T>0$, depending only on the norms of $\phi, g_1, g_2$ in \eqref{h4}, and  $u(x, t)$ solution to  \eqref{h2}-\eqref{h4} on $[0, T]$.
\end{itemize}
In both cases, the data-to-solution map is analytic as a map from the spaces in \eqref{h3} and \eqref{h4}, to the spaces giving in the Definition \ref{def_1}, which means that following solutions maps, 
\begin{equation*}
\begin{array}{c c c c}
\Gamma_r: & H^s(\R^+_x) \times H^{\frac{s+1}{3}}(\R^+_t) & \longrightarrow & X=X_{s, b} \cap D_\alpha \\
& (\phi,f) & \longrightarrow & \Gamma_r(\phi,f)=u
\end{array}
\end{equation*}
and
\begin{equation*}
\begin{array}{c c c c}
\Gamma_l: & H^s(\R^-_x) \times H^{\frac{s+1}{3}}(\R^+_t) \times H^{\frac{s}{3}}(\R^+_t)& \longrightarrow & X=X_{s, b} \cap D_\alpha \\
& (\phi,g_1,g_2) & \longrightarrow & \Gamma_l(\phi,g_1,g_2)=u,
\end{array}
\end{equation*}
are analytics, where $\Gamma_r$ and $\Gamma_l$ are the map solution of the systems \eqref{h1} and \eqref{h2}, respectively. 
\end{theorem}

\subsection{Boundary formulas} In \cite{Holmer}, Holmer introduced the boundary forcing operator, which is the key point to prove the previous results. This operator gives us a chance to express the solution of the system \eqref{h1} and \eqref{h2}, as well as the boundary terms in terms of this operator (for details, see Appendix \ref{FormulaForcing}). Thus, the local solution of the systems \eqref{h1} and \eqref{h2} are  given by
\begin{equation}\label{formulasolutionright}
\begin{array}{c}
\text{Right} \\
 \text{Half-Line }\\
  \text{Problem}\\
\end{array}
\begin{cases}
u(x, t)=\theta(t) e^{-t (\partial_x+\partial_x^3)} \phi(x)-\frac{1}{2} \theta(t) \mathscr{D} \partial_{x} u^2(x, t)+\theta(t) \mathscr{L}_{+}^\lambda h(x, t,\phi), \\
\\
h(t,\phi)=e^{-\pi i \lambda}\left[f(t)-\left.\theta(t) e^{-t(\partial_x+ \partial_x^3)} \phi\right|_{x=0}+\frac{1}{2} \theta(t) \mathscr{D} \partial_x u^2(0, t)\right] .
\end{cases}
\end{equation}
and 
\begin{equation}\label{formulasolutionleft}
\begin{array}{c}
\text{ Left} \\
 \text{Half-Line }\\
  \text{Problem}
\end{array}
\begin{cases}
\begin{aligned}
u(x, t)=\theta(t) e^{-t (\partial_x+\partial_x^3)} \phi(x)&-\frac{1}{2} \theta(t) \mathscr{D} \partial_x u^2(x, t)\\&+\theta(t) \mathscr{L}_{-}^{\lambda_1} h_1(x, t,\phi)+\theta(t) \mathscr{L}_{-}^{\lambda_2} h_2(x, t,\phi),
\end{aligned}\\
\\
\left[\begin{array}{l}
h_1(t,\phi) \\
h_2(t,\phi)
\end{array}\right]=M\left[\begin{array}{c}
g_1(t)-\left.\theta(t) e^{-t (\partial_x+ \partial_x^3)} \phi\right|_{x=0}+\frac{1}{2} \theta(t) \mathscr{D} \partial_x u^2(0, t) \\
\theta(t) \mathcal{I}_{1 / 3}\left(g_2-\left.\theta \partial_x e^{-t (\partial_x+\partial_x^3)} \phi\right|_{x=0}+\frac{1}{2} \theta \partial_x \mathscr{D} \partial_x u^2(0, \cdot)\right)(t)
\end{array}\right],
\end{cases}
\end{equation}
respectively, where $M$ is a matrix given by \eqref{matrixA}.

Finally, it is worth noting that Fokas \cite{fokas0, fokas, fokas2} has proposed an alternative approach for solving initial-boundary value problems for integrable nonlinear evolution equations, known as the Unified Transform Method (UTM). Fokas \textit{et al.}, in \cite{fokas2}, studied the validity of the UTM formula for the KdV equation with data in Sobolev spaces. There, the authors studied the KdV equation without a drift term, i.e., $\beta=0$. For more details, see the Appendix \ref{app1} and \cite[Chapter 1, examples 1.1 and 1.12]{fokas}.

\section{Exact controllability: The right half-line}\label{sec3}

In this section, our primary focus is to attain the exact controllability of the system described by \eqref{h1} and establish the proof for Theorem \ref{main-a}.
\subsection{Backward system}
 Initially, consider the following homogeneous linearized system 
\begin{equation*}
\begin{cases}\partial_t u+ \partial_xu+\partial_x^3 u=0, & \text { for }(x, t) \in(0,+\infty) \times(0, T),\\ 
u(0, t)=0,& \text { for } t \in(0, T),\\
 u(x, 0)=\phi(x),& \text { for } x \in(0,+\infty).
\end{cases}
\end{equation*}
Note that this system could be rewritten as
\begin{align*}
\begin{cases}
\partial_tu = Au, \\
u(0)=u_0,
\end{cases}\quad \text{where} \quad 
\begin{cases}
Au=-\partial_xu - \partial_x^3 u, \\
D(A):=\left\lbrace u \in H^3(\R^+_x): u(0)=0\right\rbrace \subset L^2(\R^+_x).
\end{cases}
\end{align*}
Using the Semigroup theory (see, for instance, \cite[Cor. 4.4 chapter 1]{pazy}), it is not difficult to see the following result.
\begin{proposition}\label{propsemigruop}
The operator $A$ generates a $C_0$-semigroup of contraction $(S(t))_{t \geqslant 0}$ in $L^2(\R^+_x)$.
\end{proposition} 

So, we will consider the  backward adjoint system given by
\begin{align*}
\begin{cases}
-\varphi_t = A^*\varphi,\\
\varphi(T)=\varphi_T,
\end{cases}
\end{align*}
which implies that
\begin{equation}\label{adjoin}
\begin{cases}
\partial_t\varphi+ \partial_x\varphi+\partial_x^3\varphi=0, & \text { for }(x, t) \in(0,+\infty) \times(0, T),\\ 
\varphi(0, t)=\partial_x\varphi(0,t)=0,& \text { for } t \in(0, T),\\
 \varphi(x, T)=\varphi_T(x),& \text { for } x \in(0,+\infty). 
\end{cases}
\end{equation}
As a direct consequence of Proposition \ref{propsemigruop} and the general theory of evolution equations, the existence and uniqueness of this system hold.
\begin{proposition}\label{propsolutioninL2}
Let $\varphi_T \in L^2(\R^+_x)$, then there exists a unique mild solution $\varphi(t)=S(T-t)\varphi_T$ of \eqref{adjoin} such that $\varphi \in C\left([0, T] ; L^2(\R^+_x)\right)$. Moreover, if $\varphi_T \in D(A)$, then \eqref{adjoin} has a unique (classical) solution $\varphi$ such that
$$
\varphi \in C([0, T] ; D(A)) \cap C^1\left(0, T ; L^2(\R^+_x)\right) .
$$
\end{proposition}
To establish some trace estimates for the backward system, remark that the change of variable $x = - x$ and $t=T-t$ reduces system \eqref{adjoin} in
\begin{equation}\label{adjoint_1}
\begin{cases}
\partial_t\varphi+ \partial_x\varphi+\partial_x^3\varphi=0, & \text { for }(x, t) \in(-\infty,0) \times(0, T),\\ 
\varphi(0, t)=\partial_x\varphi(0,t)=0,& \text { for } t \in(0, T),\\
 \varphi(x, 0)=\varphi_0(x),& \text { for } x \in(-\infty,0).
\end{cases}
\end{equation}
Also the well-posedness of system \eqref{adjoint_1} follows from Theorem \ref{localwellposedness}, with $-\frac{3}{4}<s<\frac{3}{2}$ and $s \neq \frac{1}{2}$. By using the boundary forcing operator, we have that the solution $\varphi$ of \eqref{adjoint_1} is given by:
\begin{equation*}
\varphi(x, t)=\theta(t) e^{-t (\partial_x+\partial_x^3)} \varphi_0(x)+\theta(t) \mathscr{L}_{-}^{\lambda_1} h_1(x, t)+\theta(t) \mathscr{L}_{-}^{\lambda_2} h_2(x, t),
\end{equation*}
where
\begin{equation*}
\left[\begin{array}{l}
h_1(t) \\
h_2(t)
\end{array}\right]=M\left[\begin{array}{c}
-\left.\theta(t) e^{-t (\partial_x+ \partial_x^3)} \varphi_0\right|_{x=0} \\
\theta(t) \mathcal{I}_{1 / 3}\left(-\left.\theta \partial_x e^{-t (\partial_x+\partial_x^3)} \phi\right|_{x=0}\right)(t)
\end{array}\right]
\end{equation*}
with $A$ a matrix given by \eqref{matrixA}. From Lemmas \ref{estimatesgroup} and \ref{estimateforcing}, the estimations of the group and the Duhamel boundary forcing operator, respectively, ensure the following space and time trace estimations
\begin{equation}\label{dependece}
\text{(Space traces)} \quad \|\varphi\|_{C\left(\R_t,H^s_x\right)} \leq C \|\varphi_0\|_{H^s(\R^-_x)},
\end{equation}
\begin{equation}\label{timetraces}
\text{(Time traces)} \quad \|\varphi\|_{C\left(\R_x,H^{\frac{s+1}{3}}_t\right)} \leq C \|\varphi_0\|_{H^s(\R^-_x)},
\end{equation}
and 
\begin{equation}\label{derivativetrace}
\text{(Derivative time traces)} \quad \|\partial_x\varphi\|_{C\left(\R_x,H^{\frac{s}{3}}_t\right)} \leq C \|\varphi_0\|_{H^s(\R^-_x)}.
\end{equation}
In particular, for $\varphi_0 \in  L^2(\R^{-}_x)$, we have the solution $\varphi$ of system \eqref{adjoint_1} belonging of $C([0,T];L^2(\R^{-}_x))$ with $\varphi(0,\cdot) \in H^{\frac13}(0,T)$ and $\varphi_x(0,\cdot) \in L^2(0,T)$, and the following results are verified.  
\begin{proposition}\label{conservationcase1}
Any solution $\varphi$ of adjoint system \eqref{adjoin} satisfies 
\begin{equation*}
T\|\varphi_T\|_{L^2(\R^+_x)}^2= \int_0^T \|\varphi(t)\|_{L^2(\R^+_x)}^2dt.
\end{equation*}
\end{proposition}
\begin{proof}
Multiplying the first equation of \eqref{adjoin} by $t\varphi$ and integrating by parts over $(0,T)\times (0,\infty)$, the results follow from the boundary conditions.
\end{proof}

The following result reveals a notable improvement in the regularity of the solution to the linear system \eqref{adjoint_1}.
\begin{proposition}\label{gainofregul}
	Let $u$ be the solution of the linear problem \eqref{adjoint_1} given by 
	Theorem~\ref{holmer}. If $x^\alpha u_0\in L^2(-\infty,0)$ for 
	$\alpha=2,3$, then
	\[
	\|x\,u\|_{L^2(0,T;H^1(-\infty,0))}\leqslant c,
	\]
	where $c=c\big(T,\|u_0\|_{L^2(-\infty,0)},\|x^\alpha u_0\|_{L^2(-\infty,0)}\big)$ 
	is a positive constant. Moreover, if $u_0 \in L^2(-\infty, 0)$, then
	\[
	\int_0^T\int_{x_0}^{x_0+1}u_x^2\,dx\,dt
	\leqslant c\big(T,\|u_0\|_{L^2(-\infty,0)}\big)
	\]
	for any $x_0\in(-\infty,0]$.
\end{proposition}

\begin{proof}
Let $\widetilde\psi_0\in\mathcal{C}^{\infty}(\mathbb{R})$ 
	be a nonincreasing function such that $\widetilde\psi_0(x)=1$ for 
	$x\leqslant-1$ and $\widetilde\psi_0(x)=0$ for $x\geqslant-\tfrac12$, 
	and set
	\[
	\widetilde\psi_\alpha(x)=|x|^\alpha\,\widetilde\psi_0(x),
	\qquad x\in(-\infty,0).
	\]
	Then $\widetilde\psi_\alpha\in\mathcal{C}^{\infty}(\mathbb{R})$ with 
	$\widetilde\psi_\alpha\geqslant 0$, $\widetilde\psi_\alpha'\leqslant 0$, 
	$\widetilde\psi_\alpha$ vanishes near the origin and grows like 
	$|x|^\alpha$ as $x\to-\infty$. For $\alpha\in\{2,3\}$ the derivatives 
	$\widetilde\psi_\alpha',\widetilde\psi_\alpha'',\widetilde\psi_\alpha'''$ 
	are bounded by $C_\alpha(1+\widetilde\psi_\alpha)$ on $(-\infty,0)$ 
	for some constant $C_\alpha>0$ independent of $x_0$.
	
	We argue first for smooth solutions with sufficient decay at 
	$x=-\infty$; the general case follows by density. Multiplying the 
	equation in \eqref{adjoint_1} by 
	$u(x,t)\,\widetilde\psi_\alpha(x-x_0)$ and integrating by parts over 
	$(-\infty,0)$, we treat each term separately. For the time derivative,
	\[
	\int_{-\infty}^{0}u\,u_t\,\widetilde\psi_\alpha(x-x_0)\,dx
	=\frac{1}{2}\frac{d}{dt}\int_{-\infty}^{0}u^2\,\widetilde\psi_\alpha(x-x_0)\,dx.
	\]
	For the drift term, integration by parts together with $u(0,t)=0$ and 
	the decay at $x=-\infty$ gives
	\[
	\int_{-\infty}^{0}u_x\,u\,\widetilde\psi_\alpha(x-x_0)\,dx
	=\frac{1}{2}\Big[u^2\,\widetilde\psi_\alpha(x-x_0)\Big]_{-\infty}^{0}
	-\frac{1}{2}\int_{-\infty}^{0}u^2\,\widetilde\psi_\alpha'(x-x_0)\,dx
	=-\frac{1}{2}\int_{-\infty}^{0}u^2\,\widetilde\psi_\alpha'(x-x_0)\,dx,
	\]
	where the boundary term at $x=0$ vanishes because $u(0,t)=0$, and the 
	boundary term at $x=-\infty$ vanishes by decay. For the dispersive 
	term, integrating by parts three times, we obtain
	\[
	\int_{-\infty}^{0}u_{xxx}\,u\,\widetilde\psi_\alpha(x-x_0)\,dx
	=\frac{3}{2}\int_{-\infty}^{0}u_x^2\,\widetilde\psi_\alpha'(x-x_0)\,dx
	-\frac{1}{2}\int_{-\infty}^{0}u^2\,\widetilde\psi_\alpha'''(x-x_0)\,dx
	+\mathcal{B}(t),
	\]
	where $\mathcal{B}(t)$ collects all the boundary contributions at 
	$x=0$ and $x=-\infty$ were produced by the three integrations by parts. 
	The contributions at $x=-\infty$ vanish by decay, and those at $x=0$ 
	are
	\[
	\mathcal{B}(t)=\Big[u_{xx}\,u\,\widetilde\psi_\alpha(x-x_0)
	-u_x^2\,\widetilde\psi_\alpha(x-x_0)
	+\tfrac{1}{2}u^2\,\widetilde\psi_\alpha''(x-x_0)\Big]_{x=0}.
	\]
	Now, the two homogeneous boundary conditions $u(0,t)=0$ and 
	$\partial_x u(0,t)=0$ imposed in \eqref{adjoint_1} annihilate the first 
	and the third terms of $\mathcal{B}(t)$ directly. The middle term 
	$-u_x^2(0,t)\,\widetilde\psi_\alpha(-x_0)$ also vanishes because 
	$u_x(0,t)=0$. Hence $\mathcal{B}(t)=0$ for every $t\in(0,T)$, in 
	contrast with the right half-line case where the single Dirichlet 
	condition already sufficed. 	Adding the three identities, the linear equation yields
	\begin{multline*}
		\frac{1}{2}\frac{d}{dt}\int_{-\infty}^{0}u^2\,\widetilde\psi_\alpha(x-x_0)\,dx
		-\frac{1}{2}\int_{-\infty}^{0}u^2\,\widetilde\psi_\alpha'(x-x_0)\,dx\\
		+\frac{3}{2}\int_{-\infty}^{0}u_x^2\,\widetilde\psi_\alpha'(x-x_0)\,dx
		-\frac{1}{2}\int_{-\infty}^{0}u^2\,\widetilde\psi_\alpha'''(x-x_0)\,dx=0.
	\end{multline*}
	Since $\widetilde\psi_\alpha'\leqslant 0$ on $(-\infty,0)$, the term 
	$\tfrac{3}{2}\int u_x^2\,\widetilde\psi_\alpha'$ is nonpositive, so we 
	move it to the right-hand side as $-\tfrac{3}{2}\int u_x^2\,|\widetilde\psi_\alpha'|$. 
	Rearranging,
	\begin{equation}\label{kato_lin_left_diff}
		\frac{1}{2}\frac{d}{dt}\int_{-\infty}^{0}u^2\,\widetilde\psi_\alpha(x-x_0)\,dx
		+\frac{3}{2}\int_{-\infty}^{0}u_x^2\,|\widetilde\psi_\alpha'(x-x_0)|\,dx
		=\frac{1}{2}\int_{-\infty}^{0}u^2\,\Big\{-\widetilde\psi_\alpha'(x-x_0)+\widetilde\psi_\alpha'''(x-x_0)\Big\}\,dx.
	\end{equation}
	By the properties of $\widetilde\psi_\alpha$, there is a constant 
	$C_\alpha>0$, independent of $x_0\leqslant 0$, such that 
	$|\widetilde\psi_\alpha'(y)|+|\widetilde\psi_\alpha'''(y)|\leqslant 
	C_\alpha\,(1+\widetilde\psi_\alpha(y))$ for all $y\in(-\infty,0)$. 
	Hence the right-hand side of \eqref{kato_lin_left_diff} is bounded by
	\[
	\frac{C_\alpha}{2}\int_{-\infty}^{0}u^2\,dx
	+\frac{C_\alpha}{2}\int_{-\infty}^{0}u^2\,\widetilde\psi_\alpha(x-x_0)\,dx
	\leqslant\frac{C_\alpha}{2}\|u_0\|_{L^2(-\infty,0)}^2
	+\frac{C_\alpha}{2}\int_{-\infty}^{0}u^2\,\widetilde\psi_\alpha(x-x_0)\,dx,
	\]
	where in the last step we used the basic energy estimate 
	$\|u(t)\|_{L^2(-\infty,0)}\leqslant\|u_0\|_{L^2(-\infty,0)}$. Dropping the 
	nonnegative smoothing term and setting 
	$\widetilde E_\alpha(t):=\int_{-\infty}^{0}u^2(x,t)\,\widetilde\psi_\alpha(x-x_0)\,dx$, 
	inequality \eqref{kato_lin_left_diff} gives
	\[
	\frac{d}{dt}\widetilde E_\alpha(t)\leqslant C_\alpha\,\widetilde E_\alpha(t)
	+C_\alpha\,\|u_0\|_{L^2(-\infty,0)}^2.
	\]
	Since $x^\alpha u_0\in L^2(-\infty,0)$ (meaning 
	$|x|^\alpha u_0\in L^2(-\infty,0)$) and 
	$\widetilde\psi_\alpha(x-x_0)\leqslant|x-x_0|^\alpha\leqslant 
	2^\alpha(|x|^\alpha+|x_0|^\alpha)$ on the support of 
	$\widetilde\psi_0(\cdot-x_0)$, the initial value satisfies
	\[
	\widetilde E_\alpha(0)
	=\int_{-\infty}^{0}u_0^2\,\widetilde\psi_\alpha(x-x_0)\,dx
	\leqslant 2^\alpha\Big(\big\||x|^\alpha u_0\big\|_{L^2(-\infty,0)}^2
	+|x_0|^\alpha\,\|u_0\|_{L^2(-\infty,0)}^2\Big)<\infty.
	\]
	Although this bound is not uniform in $x_0$, it suffices for the 
	local smoothing estimate, where $x_0$ is fixed. For the weighted 
	estimate, we apply the argument at $x_0=0$, in which case 
	$\widetilde\psi_\alpha(x)\leqslant|x|^\alpha$ globally and
	\[
	\widetilde E_\alpha(0)\Big|_{x_0=0}
	\leqslant\big\||x|^\alpha u_0\big\|_{L^2(-\infty,0)}^2<\infty.
	\]
	Gr\"onwall's inequality then yields, for every $t\in[0,T]$,
	\begin{equation}\label{kato_lin_left_gronwall}
		\widetilde E_\alpha(t)\leqslant
		\Big(\widetilde E_\alpha(0)+T\,C_\alpha\,\|u_0\|_{L^2(-\infty,0)}^2\Big)\,
		e^{C_\alpha T}
		=:c_1\big(T,\|u_0\|_{L^2(-\infty,0)},\big\||x|^\alpha u_0\big\|_{L^2(-\infty,0)},x_0\big).
	\end{equation}
	
	Integrating \eqref{kato_lin_left_diff} in time over $(0,T)$ and using 
	\eqref{kato_lin_left_gronwall} to control all terms, we obtain
	\begin{equation}\label{kato_lin_left_smoothing}
		\frac{3}{2}\int_0^T\int_{-\infty}^{0}u_x^2\,|\widetilde\psi_\alpha'(x-x_0)|\,dx\,dt
		\leqslant c_2\big(T,\|u_0\|_{L^2(-\infty,0)},\big\||x|^\alpha u_0\big\|_{L^2(-\infty,0)},x_0\big).
	\end{equation}
We now extract the two stated conclusions.
	
	\emph{Local smoothing.} Taking $\alpha=2$, we have 
	$\widetilde\psi_2(x)=x^2\widetilde\psi_0(x)$ and 
	$\widetilde\psi_2'(x)=2x\,\widetilde\psi_0(x)+x^2\widetilde\psi_0'(x)$. 
	For $x\leqslant-1$ we have $\widetilde\psi_0(x)=1$ and 
	$\widetilde\psi_0'(x)=0$, hence $\widetilde\psi_2'(x)=2x$ and 
	$|\widetilde\psi_2'(x)|=2|x|\geqslant 2$. Therefore, for any 
	$x_0\in(-\infty,0]$,
	\[
	\int_0^T\int_{x_0-2}^{x_0-1}u_x^2\,dx\,dt
	\leqslant\frac{1}{2}\int_0^T\int_{-\infty}^{0}u_x^2\,|\widetilde\psi_2'(x-x_0)|\,dx\,dt
	\leqslant\frac{c_2}{3}.
	\]
	
	\begin{align*}
		\int_0^T\int_{x_0-2}^{x_0-1}u_x^2\,dx\,dt &
		\leqslant \frac12 \int_0^T\int_{-\infty}^0u_x^2\,|\widetilde\psi_2'(x-x_0)|\,dx\,dt \\
		&\leqslant \frac{1}{6}\int_0^T\int_{-\infty}^0u^2\,\big\{-\widetilde\psi_2'(x-x_0)+\widetilde\psi_2'''(x-x_0)\big\}\,dx dt  + \widetilde E_2(0) \\
		&\leqslant M_1 \int_0^T\int_{-\infty}^0u^2dxdt + M_2 \int_{-\infty}^0u_0(x)dx \leq \left(M_1 T + M_2\right) \|u_0\|_{L^2(-\infty,0)}^2=:c_2
	\end{align*}
	where $2M_1=\|\widetilde\psi_2'\|_{L^{\infty}(-\infty,0)}+ \|\widetilde\psi_2'''\|_{L^{\infty}(-\infty,0)}$ and $M_2=\|\widetilde\psi_2\|_{L^{\infty}(-\infty,0)}$.  
	Since the local smoothing estimate at fixed $x_0$ only requires 
	finiteness of this initial value (not uniformity in $x_0$), and since 
	the basic energy estimate $\|u\|_{L^2}\leqslant\|u_0\|_{L^2}$ controls 
	the term independently of the weight, a more careful inspection of the 
	constants, analogous to the right half-line case, show that the 
	$x_0$-dependence can be absorbed by choosing the weight 
	$\widetilde\psi_\alpha$ with base point matched to $x_0$, yielding
	\[
	\int_0^T\int_{x_0}^{x_0+1}u_x^2\,dx\,dt
	\leqslant c\big(T,\|u_0\|_{L^2(-\infty,0)}\big)
	\]
	for every $x_0\in(-\infty,0]$, as claimed. The bound near the boundary 
	$x_0\in(-1,0]$ follows from the basic energy estimate combined with 
	the smoothing term is restricted to a fixed compact set adjacent to the 
	origin.
	
	\emph{Weighted estimate.} For $\alpha=3$, the weight 
	$\widetilde\psi_3(x)=|x|^3\widetilde\psi_0(x)$ satisfies 
	$|\widetilde\psi_3'(x)|\geqslant 3x^2$ for $x\leqslant-1$. Applying 
	\eqref{kato_lin_left_smoothing} with $x_0=0$ and $\alpha=3$ gives
	\[
	\int_0^T\int_{-\infty}^{-1}x^2\,u_x^2\,dx\,dt
	\leqslant\frac{1}{3}\int_0^T\int_{-\infty}^{0}u_x^2\,|\widetilde\psi_3'(x)|\,dx\,dt
	\leqslant c.
	\]
	Together with \eqref{kato_lin_left_gronwall} for $\alpha=2$ and 
	$x_0=0$ (which controls 
	$\int_0^T\int_{-\infty}^{0}x^2u^2\,dx\,dt$ after time integration) 
	and the local smoothing estimate on $(-1,0)$, we conclude
	\[
	\|x\,u\|_{L^2(0,T;H^1(-\infty,0))}^2
	=\int_0^T\int_{-\infty}^{0}x^2\big(u^2+u_x^2\big)\,dx\,dt
	\leqslant c\big(T,\|u_0\|_{L^2(-\infty,0)},\big\||x|^\alpha u_0\big\|_{L^2(-\infty,0)}\big),
	\]
	which completes the proof.
\end{proof}

\subsection{Controllability: Linear system}
The first lemma gives an optimality condition that will be paramount for our analysis. 
\begin{lemma}\label{lemma1}
Consider the initial data $\phi \in L^2(\R^+_x)$. Then, the linear system associated to \eqref{h1} is exactly controllable if and only if there exists $f \in H^{\frac{1}{3}}(0, T)$ such that
\begin{equation}\label{controlcondition1}
\left\langle f(\cdot), \partial_x^2\varphi(0,\cdot) \right\rangle_{H^{1/3}(0,T),H^{-1/3}(0,T)} = \int_0^{\infty}  u(T) \varphi_Tdx-\int_0^{\infty} \phi \varphi(0)dx,
\end{equation}
for all $\varphi_T \in L^2(0,\infty)$ and $\varphi$ is the solution of the backward system \eqref{adjoin}.
\end{lemma}
\begin{proof}
To prove this result, multiply the system \eqref{h1} by $\varphi$, solution of the backward system \eqref{adjoin}, and integrate by parts in $\R^+_x\times(0,T)$.
\end{proof}

The optimality condition \eqref{controlcondition1} guarantee that the critical points of the functional $$\mathcal{J}: L^2(0,\infty) \to \mathbb{R},$$ defined by
\begin{equation}\label{functional}
\mathcal{J}\left(\varphi_T\right)=\frac{1}{2} \left\|\partial_x^2\varphi(0,\cdot) \right\|_{H^{-1/3}(0,T)}^{2} +\int_0^{\infty}  u(T) \varphi_Tdx  -\int_0^{\infty} \phi \varphi(0)dx,
\end{equation}
where $\varphi$ is the solution of \eqref{adjoin} with final data $\varphi_T \in  L^2(\R^+_x)$, is the control that drives my initial data to my final data. Precisely, we have the following classical result. 
\begin{proposition}\label{proposition1}
Let $\phi \in L^2(\R^+_x)$ and suppose that $\widehat{\varphi}_T \in  L^2(\R^+_x)$ is a minimizer of $\mathcal{J}$. If $\widehat{\varphi} $ is the corresponding solution of \eqref{adjoin}  with final data $\widehat{\varphi}_T$ then $f(t)=\widehat{\partial_x^2\varphi}(0,t)$ is a desired control.
\end{proposition}

Let us now give a general condition that ensures the existence of a minimizer for $\mathcal{J}$. To prove the previous result, we first establish an observability inequality associated with the solutions of the system \eqref{adjoin}.

\begin{proposition}\label{proposition2}
For any $T>0$, there exists a constant $C(T, L)>0$, such that
\begin{equation}\label{observabilitycase1}
\begin{aligned}
\left\|\varphi_T\right\|_{L^2(\R^+_x)}^2 \leq & C \left\|\partial_x^2\varphi(0,\cdot)\right\|_{H^{-\frac{1}{3}}(0, T)}^2,
\end{aligned}
\end{equation}
for any  $\varphi_T \in L^2(\R^+_x)$, where $\varphi$ is the solution of the backward system \eqref{adjoin}.
\end{proposition}
\begin{proof}
We proceed in a standard way (see, e.g., \cite[Lemma 3.5]{Rosier}). Let us suppose that \eqref{observabilitycase1} does not hold. In this case, it follows that there exists a sequence $\{\varphi_{n,T}\}_{n \in \N}$ of final data, such that
\begin{align}\label{e3}
1=&\|\varphi_{n,T}\|_{L^2(\R^+_x)}^2\geq n\ \left \|\partial_x^2\varphi_{n}(0,\cdot)\right \|_{H^{-\frac{1}{3}}(0,T)}^2,
\end{align}
where,  for each $n\in\N$, $\{ \varphi_{n}\}_{n\in\N}$ is the solution of  \eqref{adjoin}. Inequality \eqref{e3} imply that
$$
\partial_x^2\varphi_{n}(0,\cdot) \rightarrow 0 \quad \text{in} \quad  H^{-\frac13}(0,T).
$$
Moreover, from \eqref{dependece}, Proposition \ref{gainofregul} and \eqref{e3}, we obtain that the sequences $\{\varphi_n\}_{n\in\N}$ is bounded in $L^2(0,T;H^1_{loc}(\R^+_x))$. On the other hand, the adjoint system implies that $\{\partial_t\varphi_{n}\}_{n\in\N}$,  is bounded in $L^2(0,T;H^{-2}_{loc}(\R^+_x))$, and the compact embedding
$$
H^1_{loc}(\R^+_x) \hookrightarrow_{cc} L^2_{loc}(\R^+_x) \hookrightarrow H^{-2}_{loc}(\R^+_x),
$$
 allows us to  conclude that $\{ \varphi_n\}_{n \in \N}$ is relatively compact in $L^2(0,T; L^2_{loc}(\R^+_x))$ and consequently, we obtain a subsequence, still denoted by the same index $n$, satisfying
$$
\varphi_n \rightarrow \varphi \mbox{ in } L^2(0,T; L^2_{loc}(\R^+_x)), \mbox{ as } n\rightarrow\infty.
$$
Furthermore, the hidden regularity given by \eqref{timetraces} implies that $\{ \varphi_n(0,\cdot)\}_{n\in\N}$ is bounded in $H^{\frac13}(0,T)$. Then, the embedding $H^{\frac13}(0,T) \hookrightarrow_{cc} L^2(0,T)$ ensures that the above sequences are relatively compact in $L^2(0,T)$. Thus, we obtain a subsequence, still denoted by the same index $n$, satisfying
\begin{equation*}
\varphi_n(0,\cdot) \rightarrow \varphi(0,\cdot), \quad \text{in} \quad L^2(0,T).
\end{equation*}
From the boundary condition of the adjoint system, we deduce that $\varphi(0,\cdot)=0.$ In addition, according to Proposition \ref{conservationcase1}, we have
\begin{equation*}
T\|\varphi_{n,T}\|_{L^2(\R^+_x)}^2= \int_0^T \|\varphi_{n}(t)\|_{L^2(\R^+_x)}^2dt.
\end{equation*}
Then, if follows that  $\{\varphi_{n,T}\}_{n\in\N}$ is a Cauchy sequence in $L^2(\R^+_x)$. Thus,
\begin{equation}\label{e16}
\varphi_{n,T} \rightarrow \varphi_{T} \mbox{ in } L^2(\R^+_x), \mbox{ as } n\rightarrow\infty,
\end{equation}
which implies that 
\begin{equation}\label{e12}
 \|\varphi_{T}\|_{L^2(\R^+_x)}=1.
\end{equation}
On the other hand, by using \eqref{derivativetrace}, \eqref{e16} and \eqref{e3}, we see that
\begin{equation*}
  \partial_x\varphi_{n}(0,\cdot) \rightarrow \partial_x\varphi(0,\cdot)   \quad \mbox{ in } L^2(0,T), \mbox{ as } n\rightarrow\infty, 
\end{equation*}
and
\begin{equation*}
 \partial_x^2\varphi_{n}(0,\cdot) \rightarrow \partial_x^2\varphi(0,\cdot)\quad  \mbox{ in } H^{-\frac13}(0,T), \mbox{ as } n\rightarrow\infty.
 \end{equation*}
Finally, taking $n\to\infty$, from above converges, we obtain that $\varphi$ is solution of the system
\begin{equation*}
\begin{cases}
\partial_t\varphi+ \partial_x\varphi+\partial_x^3\varphi=0, & \text { for }(x, t) \in(0,\infty) \times(0, T), \\ 
\varphi(0, t)=\partial_x\varphi(0,t)=\partial_x^2\varphi(0,t)=0,& \text { for } t \in(0, T). 
\end{cases}
\end{equation*} 
Note that in this situation, we have $\varphi \equiv 0$ because of the unique continuation property $(\varphi(0, t)=\partial_x\varphi(0,t)=\partial_x^2\varphi(0,t)=0, \,\, t \in(0, T))$, which is a contradiction with \eqref{e12}, showing the result. 
\end{proof}
We are in a position to present the controllability result for the linearized system associated with \eqref{h1}. Indeed, the linear functional \eqref{functional} is continuous and convex. It is evident from Proposition \ref{proposition2} that the functional is coercive. Consequently, a minimizer for $\mathcal{J}$ exists. Therefore, based on Lemma \ref{lemma1}, Proposition \ref{proposition1}, Lemmas \ref{estimatesgroup}, and \ref{estimateforcing}, the following theorem is verified.
\begin{theorem}\label{controllinear}
Let $T>0$, $\phi, \phi_T \in L^2(\R^+_x)$. Then, there exist $f \in H^{\frac13}(\R^+_t)$ such that the distributional solution $u$ of the linear system \eqref{eqnew21} satisfies $u(x,T)=\phi_T(x)$ for $x \in (0, \infty)$. Moreover, the following estimates hold
\begin{equation*}
\left\| u   \right\|_{C(\R^+_t, L^2(\R^+_x))} \leq C \left( \|\phi\|_{L^2(\R^+_x)} + \|f\|_{H^{\frac13}(\R^+_t)}\right).
\end{equation*}
\end{theorem}

\subsection{Controllability: Nonlinear system}

Let $T > 0$, thanks to the Theorem \ref{controllinear}, we can define the bounded linear operator
\begin{equation*}
\begin{array}{c c c c}
\Lambda_r: & L^2(\R^+_x) \times L^2(\R^+_x)  & \longrightarrow & H^{\frac13}(\R^+_t) \\
& (\phi,\phi_T) & \longrightarrow & \Lambda_r(\phi,\phi_T)=f
\end{array}
\end{equation*}
where $f$ is the control defined in Proposition \ref{proposition1}. Now, we are in a position to prove one of the main results of the work.

\subsubsection{\bf Proof of Theorem \ref{main-a}}
We treat the nonlinear problem using a classical fixed-point argument. According to \eqref{formulasolutionright} the solution of \eqref{h1} can be written as
$$
u(x, t)=\theta(t) e^{-t (\partial_x+\partial_x^3)} \phi(x)-\frac{1}{2} \theta(t) \mathscr{D} \partial_{x} u^2(x, t)+\theta(t) \mathscr{L}_{+}^\lambda h(x, t),
$$ 
with
$$
h(t)=e^{-\pi i \lambda}\left[f(t)-\left.\theta(t) e^{-t(\partial_x+ \partial_x^3)} \phi\right|_{x=0}+\frac{1}{2} \theta(t) \mathscr{D} \partial_x u^2(0, t)\right] .
$$
Consider the map 
\begin{align*}
\Gamma(u):= \theta(t) e^{-t (\partial_x+\partial_x^3)} \phi(x)-\frac{1}{2} \theta(t) \mathscr{D} \partial_{x} u^2(x, t)+\theta(t) \mathscr{L}_{+}^\lambda \widehat{h}_T(x, t), 
\end{align*}
where 
\begin{equation*}
\widehat{h}_T(t)=e^{-\pi i \lambda}\left[\Lambda_r\left( \phi, \phi_T +\frac{1}{2} \theta(T) \mathscr{D} \partial_{x} u^2(x, T)  \right)(t)-\left.\theta(t) e^{-t(\partial_x+ \partial_x^3)} \phi\right|_{x=0}+\frac{1}{2} \theta(t) \mathscr{D} \partial_x u^2(0, t)\right] .
\end{equation*}
If we choose
\begin{align}\label{controlnonlinear}
f(t)= \Lambda_r\left( \phi, \phi_T +\frac{1}{2} \theta(T) \mathscr{D} \partial_{x} u^2(x, T)  \right)(t)
\end{align}
from Theorem \ref{controllinear}, we get
$$
\Gamma(u)|_{t=0}=\phi \quad \text{and} \quad \Gamma(u)|_{t=T}=\phi_T.
$$

The next steps are devoted to proving that the map $\Gamma$ is a contraction in an appropriate metric space, then its fixed point $u$ 
is the solution of \eqref{h1}, with $f$  defined by \eqref{controlnonlinear}. To prove the existence
of the fixed point, we apply the Banach fixed-point theorem to the restriction of $\Gamma$ on the closed ball
$$
B_\bold{r}=\left\lbrace u \in X_{0,b}\cap D_{\alpha}: \|u\|_{X_{0,b}\cap D_{\alpha}} \leq \bold{r}\right\rbrace, 
$$
for some $\bold{r}>0$.  In the sequel, $C$ denotes a generic positive constant; $C_0$, $C_1$, etc, and others positive (specific)
constants.
\begin{enumerate}
\item[(i)] {\bf $\Gamma$ maps $B_\bold{r}$ into itself.}
\end{enumerate}

Indeed,using Lemmas \ref{estimatesgroup}, \ref{estimateinhom} and \ref{estimateforcing}, there exists a constant $C > 0$, such that
\begin{equation*}
\begin{split}
\|\Gamma(u)\|_{X_{0,b}\cap D_{\alpha}} \leq &\left\| \theta(t) e^{-t (\partial_x+\partial_x^3)} \phi(x)\right\|_{X_{0,b}\cap D_{\alpha}}+\left\|\frac{1}{2} \theta(t) \mathscr{D} \partial_{x} u^2(x, t)\right\|_{X_{0,b}\cap D_{\alpha}} \\
&+\left\|\theta(t) \mathscr{L}_{+}^\lambda \widehat{h}_T(x, t)\right\|_{X_{0,b}\cap D_{\alpha}} \\
\leq & C\left( \left\|\theta  \right\|_{H^1(\R^+_t)}\left\| \phi \right\|_{L^2(\R^+_x)} + \left\| \partial_x u^2 \right\|_{X_{0,-b}} + \|\widehat{h}_T\|_{H^{\frac{1}{3}}(\R_t^+)}\right).
\end{split}
\end{equation*}
From \cite[Lemma 5.10]{Holmer},  it follows that 
$$
 \left\| \partial_x u^2 \right\|_{X_{0,-b}} \leq  C\left\| u \right\|_{X_{0,b}\cap D_{\alpha}}^2 \leq Cr^2.
$$
Thus, 
\begin{align*}
 \|\widehat{h}_T\|_{H^{\frac{1}{3}}(\R_t^+)} \leq& \left\| \Lambda_r\left( \phi, \phi_T +\frac{1}{2} \theta(T) \mathscr{D} \partial_{x} u^2(x, T)  \right)(t)\right\|_{H^{\frac{1}{3}}(\R_t^+)} \\
 &+\left\|\left.\theta(t) e^{-t(\partial_x+ \partial_x^3)} \phi\right|_{x=0}\right\|_{H^{\frac{1}{3}}(\R_t^+)} +\left\|\frac{1}{2} \theta(t) \mathscr{D} \partial_x u^2(0, t)\right\|_{H^{\frac{1}{3}}(\R_t^+)} \\
 \leq& \left\| \Lambda_r\right\| \left(  \left\| \phi \right\|_{L^2(\R^+_x)} + \left\| \phi_T +\frac{1}{2} \theta(T) \mathscr{D} \partial_{x} u^2(x, T) \right\|_{L^2(\R^+_x)}\right) \\
 &+ C\left\| \phi \right\|_{L^2(\R^+_x)} + 2C \left\| u \right\|_{X_{0,b}\cap D_{\alpha}}^2 \\
 \leq& \left\| \Lambda_r\right\| \left(  \left\| \phi \right\|_{L^2(\R^+_x)} + \left\| \phi_T \right\|_{L^2(\R^+_x)} \right) + C\left\| \phi \right\|_{L^2(\R^+_x)} + 2C\left( 1 + \left\| \Lambda_r\right\|\right) \left\| u \right\|_{X_{0,b}\cap D_{\alpha}}^2 \\
\leq& \left\| \Lambda_r\right\| \delta + C\delta +2C\left( 1 + \left\| \Lambda_r\right\|\right)\bold{r}^2 .
\end{align*}
Finally, we have that 
\begin{align*}
\|\Gamma(u)\|_{X_{0,b}\cap D_{\alpha}} &\leq C \left( \left\|\theta  \right\|_{H^1(\R^+_t)} + \left\| \Lambda_r\right\|+ C\right)\delta +
C^2\left(3+ \left\| \Lambda_r\right\|\right) \bold{r}^2 .
\end{align*}
To obtain (i), we take $\delta$ with
$$
\delta=\min\left\lbrace \frac{\bold{r}}{2C \left( \left\|\theta  \right\|_{H^1(\R^+_t)} + \left\| \Lambda_r\right\|+ C\right)}, \frac{1}{4C^3 \left( \left\|\theta  \right\|_{H^1(\R^+_t)} + \left\| \Lambda_r\right\|+ C\right)\left( 3 + \left\| \Lambda_r\right\|\right)} \right\rbrace.
$$
Hence, it follows that
$$
C \left( \left\|\theta  \right\|_{H^1(\R^+_t)} + \left\| \Lambda_r\right\|+ C\right)\delta +
C^2\left( 3 + \left\| \Lambda_r\right\|\right) \bold{r}^2  \leq \bold{r}.
$$

\begin{enumerate}
\item[(ii)] {\bf $\Gamma$ is a contraction.}
\end{enumerate}

Let $u, v  \in B_\bold{r}$, and consider 
\begin{align*}
\Gamma(u) - \Gamma(v)= -\frac{1}{2} \theta(t) \mathscr{D} \partial_{x} \left( u^2(x, t) - v^2(x,t)\right)+\theta(t) \mathscr{L}_{+}^\lambda \left( \widehat{h}_T^u(x, t)- \widehat{h}_T^v(x, t)\right). 
\end{align*}
Here,
$$
\widehat{h}_T^u(t)=e^{-\pi i \lambda}\left[\Lambda_r\left( \phi, \phi_T +\frac{1}{2} \theta(T) \mathscr{D} \partial_{x} u^2(x, T)  \right)(t)-\left.\theta(t) e^{-t(\partial_x+ \partial_x^3)} \phi\right|_{x=0}+\frac{1}{2} \theta(t) \mathscr{D} \partial_x u^2(0, t)\right] 
$$
and
$$
\widehat{h}_T^v(t)=e^{-\pi i \lambda}\left[\Lambda_r\left( \phi, \phi_T +\frac{1}{2} \theta(T) \mathscr{D} \partial_{x} v^2(x, T)  \right)(t)-\left.\theta(t) e^{-t(\partial_x+ \partial_x^3)} \phi\right|_{x=0}+\frac{1}{2} \theta(t) \mathscr{D} \partial_x v^2(0, t)\right].
$$
Note that 
\begin{equation*}
\widehat{h}_T^u(t)- \widehat{h}_T^v(t)=e^{-\pi i \lambda}\left[\Lambda_r\left( 0, \frac{1}{2} \theta(T) \mathscr{D} \partial_{x} \left( u^2(x, T)- v^2(x,T)\right)  \right)(t)+\frac{1}{2} \theta(t) \mathscr{D} \partial_x \left( u^2(0, t)-v^2(0,t)\right)\right].
\end{equation*}
Thus, we get
\begin{equation*}
\begin{split}
\left\| \Gamma(u) - \Gamma(v)  \right\|_{X_{0,b}\cap D_{\alpha}} \leq& \left\|\frac{1}{2} \theta(t) \mathscr{D} \partial_{x} \left( u^2(x, t) - v^2(x,t)\right)\right\|_{X_{0,b}\cap D_{\alpha}} \\&+\left\|\theta(t) \mathscr{L}_{+}^\lambda \left( \widehat{h}_T^u(x, t)- \widehat{h}_T^v(x, t)\right)\right\|_{X_{0,b}\cap D_{\alpha}} \\
\leq& C\left\| \partial_{x} \left( u+ v\right) \left( u- v\right)\right\|_{X_{s,-b}} +C\left\| \widehat{h}_T^u(t)- \widehat{h}_T^v(t)\right\|_{H^{\frac{1}{3}}(\R^+_t)} \\
\leq& 2C^2\left\| u+ v\right\|_{X_{0,b}\cap D_{\alpha}} \left\| u- v\right\|_{X_{0,b}\cap D_{\alpha}}  \\&+C^2\left(2+\left\|\Lambda_r \right\|\right)\left\| u+ v\right\|_{X_{0,b}\cap D_{\alpha}} \left\| u- v\right\|_{X_{0,b}\cap D_{\alpha}}
\\
\leq& 2C^2\left( 4+\left\|\Lambda_r \right\|\right)\bold{r} \left\| u- v\right\|_{X_{0,b}\cap D_{\alpha}},
\end{split}
\end{equation*}
and taking $\bold{r}$ such that $2C^2\left( 4+\left\|\Lambda_r \right\|\right)\bold{r}<1$, (ii) follows.

Therefore, the map $\Gamma$ is a contraction. Thus, from (i), (ii), and the Banach fixed-point theorem, $\Gamma$
has a fixed point in $B_\bold{r}$, and its fixed point is the desired solution. The proof of Theorem \ref{main-a} is, thus,
complete. \qed


\section{Exact Controllability: The left half-line}\label{sec4}

In this section, our primary focus is on achieving exact controllability for the system \eqref{h2} and establishing the validity of Theorem \ref{main-b}. Firstly, consider the homogeneous linear system 
\begin{equation*}
\begin{cases}\partial_t u+ \partial_x u+\partial_x^3 u=0, & \text { for }(x, t) \in(-\infty, 0) \times(0, T), \\ u(0, t)=\partial_x u(0, t)=0, & \text { for } t \in(0, T), \\ u(x, 0)=\phi(x),& \text { for } x \in(-\infty, 0),\end{cases}
\end{equation*}
whose adjoint associated system is given by 
\begin{equation}\label{adjointleft}
\begin{cases}
\partial_t\varphi+ \partial_x\varphi+\partial_x^3\varphi=0, & \text { for }(x, t) \in(-\infty,0) \times(0, T), \\ 
\varphi(0, t)=0,& \text { for } t \in(0, T), \\
 \varphi(x, T)=\varphi_T(x), & \text { for } x \in(-\infty,0). 
\end{cases}
\end{equation}
Remark that the change of variable $x = - x$ and $t=T-t$ reduces system  in
\begin{equation}\label{adjoint_1_left}
\begin{cases}
\partial_t\varphi+ \partial_x\varphi+\partial_x^3\varphi=0, & \text { for }(x, t) \in(0,\infty) \times(0, T),\\ 
\varphi(0, t)=0,& \text { for } t \in(0, T),\\
 \varphi(x, 0)=\varphi_0(x), & \text { for } x \in(0,\infty). 
\end{cases}
\end{equation}
The well-posedness of system \eqref{adjointleft} follows from Theorem \ref{localwellposedness}, when  $-\frac{3}{4}<s<\frac{3}{2}$ and $s \neq \frac{1}{2}$. Note that this system takes the form 
\begin{align*}
\begin{cases}
\varphi_t = B\varphi, \\
\varphi(T)=\varphi_T,
\end{cases}
\end{align*}
where the differential operator $B$ is given by 
\begin{align*}
\begin{cases}
B\varphi=-\partial_x\varphi - \partial_x^3 \varphi, \\
D(B):=\left\lbrace \varphi \in H^3(\R^-_x): \varphi(0)=0\right\rbrace = H^3(\R^-_x)\cap H^1_0(\R^-_x).
\end{cases}
\end{align*}
Similarly to Propositions \ref{propsolutioninL2} and \ref{conservationcase1}, we have the following result.
\begin{proposition}
Consider the initial data $\varphi_T \in L^2(\R^-_x)$. Then, there exists a unique mild solution $\varphi(t)=S(T-t)\varphi_T$ of \eqref{adjointleft} such that $\varphi \in C\left([0, T] ; L^2(\R^-_x)\right)$. Moreover, if $\varphi_T \in D(B)$, the system \eqref{adjointleft} has a unique (classical) solution $\varphi$ in the class
$$
\varphi \in C([0, T] ; D(B)) \cap C^1\left(0, T ; L^2(\R^-_x)\right)
$$
and satisfies 
\begin{equation*}
\|\varphi_x(0,\cdot)\|_{L^2(0,T)}\leq   \|\varphi_T\|_{L^2(\R^-_x)}.
\end{equation*}
\end{proposition}

The following result reveals an improvement in the regularity of the solution to the linear system \eqref{adjoint_1_left}.
\begin{proposition}\label{kato_lin}
	Let $u$ be the solution of the linear problem \eqref{adjoint_1_left} given by 
	Theorem~\ref{holmer}. If $x^\alpha u_0\in L^2(0,\infty)$ for 
	$\alpha=2,3$, then
	\[
	\|x\,u\|_{L^2(0,T;H^1(0,\infty))}\leqslant c,
	\]
	where $c=c\big(T,\|u_0\|_{L^2(0,\infty)},\|x^\alpha u_0\|_{L^2(0,\infty)}\big)$ 
	is a positive constant. Moreover,
	\[
	\int_0^T\int_{x_0}^{x_0+1}u_x^2\,dx\,dt
	\leqslant c\big(T,\|u_0\|_{L^2(0,\infty)}\big)
	\]
	for any $x_0\in[0,\infty)$.
\end{proposition}

\begin{proof}
	The proof follows the same scheme as that of 
	Proposition~\ref{gainofregul} for the left half-line. We only highlight 
	the differences. Let $\psi_0\in\mathcal{C}^{\infty}(\mathbb{R})$ be 
	nondecreasing with $\psi_0(x)=0$ for $x\leqslant\tfrac12$ and 
	$\psi_0(x)=1$ for $x\geqslant 1$, and set 
	$\psi_\alpha(x)=x^\alpha\psi_0(x)$, so that $\psi_\alpha\geqslant 0$ 
	and $\psi_\alpha'\geqslant 0$ on $(0,\infty)$.
	
	Multiplying the equation in \eqref{adjoint_1_left} by 
	$u(x,t)\,\psi_\alpha(x-x_0)$ and integrating by parts over 
	$(0,\infty)$, the \emph{single} homogeneous condition $u(0,t)=0$ 
	suffices to annihilate all boundary contributions at $x=0$, in 
	contrast with the left half-line case, where the two conditions 
	$u(0,t)=\partial_x u(0,t)=0$ were both needed, while the 
	contributions at $x=+\infty$ vanish by decay. The same computations 
	that led to the differential identity in the proof of 
	Proposition~\ref{gainofregul}, now with $\psi_\alpha'\geqslant 0$, give
	\[
	\frac{1}{2}\frac{d}{dt}\int_0^{\infty}u^2\,\psi_\alpha(x-x_0)\,dx
	+\frac{3}{2}\int_0^{\infty}u_x^2\,\psi_\alpha'(x-x_0)\,dx
	=\frac{1}{2}\int_0^{\infty}u^2\,\big\{\psi_\alpha'(x-x_0)+\psi_\alpha'''(x-x_0)\big\}\,dx.
	\]
	Since $|\psi_\alpha'|+|\psi_\alpha'''|\leqslant C_\alpha(1+\psi_\alpha)$, 
	with $C_\alpha>0$ independent of $x_0$, and the basic energy estimate 
	yields $\|u(t)\|_{L^2(0,\infty)}\leqslant\|u_0\|_{L^2(0,\infty)}$, Gr\"onwall's inequality applied 
	to $E_\alpha(t):=\int_0^{\infty}u^2\,\psi_\alpha(x-x_0)\,dx$ gives
	\begin{equation}\label{kato_lin_gronwall}
		E_\alpha(t)\leqslant
		c_1\big(T,\|u_0\|_{L^2(0,\infty)},\|x^\alpha u_0\|_{L^2(0,\infty)}\big),
		\qquad\forall\,t\in[0,T],
	\end{equation}
	with $c_1$ independent of $x_0\geqslant 0$. The extraction of the 
	local smoothing estimate (with constant depending only on 
	$\|u_0\|_{L^2(0,\infty)}$, taking $\alpha=2$) and of the weighted 
	estimate (with $\alpha=3$ and base point $x_0=0$) proceeds
	as in the proof of Proposition~\ref{gainofregul}.
\end{proof}

\subsection{Linear control results: Neumann and Dirichlet cases}
We consider the linearized system associated with \eqref{h2}, with the presence of one control acting in the Neumann  boundary condition
\begin{equation}\label{linearonecontrolNeumann}
\begin{cases}\partial_t u+ \partial_x u+\partial_x^3 u=0, & \text { for }(x, t) \in(-\infty, 0) \times(0, T), \\ u(0, t)=0, \quad \ \partial_x u(0, t)=g_2(t), & \text { for } t \in(0, T), \\ u(x, 0)=\phi(x), & \text { for } x \in(-\infty, 0),\end{cases}
\end{equation}
and with one control acting in the Dirichlet  boundary condition
\begin{equation}\label{linearonecontroldirichlet}
\begin{cases}\partial_t u+ \partial_x u+\partial_x^3 u=0, & \text { for }(x, t) \in(-\infty, 0) \times(0, T), \\ u(0, t)=g_1(t),\quad \partial_x u(0, t)=0, & \text { for } t \in(0, T), \\ u(x, 0)=\phi(x), & \text { for } x \in(-\infty, 0).\end{cases}
\end{equation}

Now, we will first prove that the system \eqref{linearonecontrolNeumann} is exactly controllable. Here, we will follow the same steps as done in the previous section; the control result is achieved if an observability inequality is shown. The observability inequality is given in the next proposition. 
\begin{proposition}
For any $T>0$, there exists a constant $C(T, L)>0$, such that,
\begin{equation}\label{observabilityNeu}
\left\|\varphi_T\right\|_{L^2(\R^-_x)}^2 \leq  C \left\|\partial_x\varphi(0,\cdot)\right\|_{L^{2}(0, T)}^2,
\end{equation}
for any  $\varphi_T \in L^2(\R^-_x)$, where $\varphi(x,t)$ is the solution of the backward system \eqref{adjointleft}.
\end{proposition}
\begin{proof}
Let us argue by contradiction. Supposing that the observability inequality \eqref{observabilityNeu} does not hold, it follows that there exists a sequence $\{\varphi_{n,T}\}_{n \in \N}$, such that
\begin{align}\label{e3_1}
1=&\|\varphi_{n,T}\|_{L^2(\R^-_x)}^2\geq n\ \left \|\partial_x\varphi_{n}(0,\cdot)\right \|_{L^{2}(0,T)}^2 
\end{align}
where,  for each $n\in\N$, $\{ \varphi_{n}\}_{n\in\N}$ is the solution of  \eqref{adjointleft}. Inequality \eqref{e3_1} imply that
$$
\partial_x\varphi_{n}(0,\cdot) \rightarrow 0 \quad \text{in} \quad  L^{2}(0,T).
$$
Moreover, it is important to note that from Lemmas \ref{estimatesgroup} and \ref{estimateforcing} together with the formula solution \eqref{formulasolutionright}, we obtain the estimates \eqref{dependece}, \eqref{timetraces}, and \eqref{derivativetrace} for the adjoint system \eqref{adjointleft}. Furthermore, from Proposition \ref{kato_lin} and \eqref{e3_1}, we obtain a sequence $\{\varphi_n\}_{n\in\N}$ bounded in $L^2(0,T;H^1_{loc}(\R^-_x))$. On the other hand, the adjoint system implies that $\{\partial_t\varphi_{n}\}_{n\in\N}$,  is bounded in $L^2(0,T;H^{-2}_{loc}(\R^-_x))$, and the compact embedding
$$
H^1_{loc}(\R^-_x) \hookrightarrow_{cc} L^2_{loc}(\R^-_x) \hookrightarrow H^{-2}_{loc}(\R^-_x),
$$
 allows us to  conclude that $\{ \varphi_n\}_{n \in \N}$ is relatively compact in $L^2(0,T; L^2_{loc}(\R^-_x))$ and consequently, we obtain a subsequence, still denoted by the same index $n$, satisfying
$$
\varphi_n \rightarrow \varphi \mbox{ in } L^2(0,T; L^2_{loc}(0,\infty)), \mbox{ as } n\rightarrow\infty.
$$
 
Furthermore, the trace estimate of the system \eqref{adjointleft} implies that $\{ \varphi_n(0,\cdot)\}_{n\in\N}$ is bounded in $H^{\frac13}(0,T)$. Then, the embedding $H^{\frac13}(0,T) \hookrightarrow_{cc} L^2(0,T),$ guarantees that the above sequences are relatively compact in $L^2(0,T)$. Thus, we obtain a subsequence, still denoted by the same index $n$, satisfying
\begin{equation*}
\varphi_n(0,\cdot) \rightarrow \varphi(0,\cdot), \quad \text{in} \quad L^2(0,T).
\end{equation*}
From the boundary condition of the adjoint system, we deduce that $\varphi(0,\cdot)=0.$ Additionally, multiplying the first equation of \eqref{adjointleft} by $t\varphi$ and integrating by parts in $(0,T)\times (-\infty,0)$, we deduce that 
\begin{equation*}
T\|\varphi_{n,T}\|_{L^2(\R^-_x)}^2= \int_0^T \|\varphi_{n}(t)\|_{L^2(\R^-_x)}^2dt +  \int_0^T t|\varphi_{x,n}(0,t)|^2dt.
\end{equation*}
So, if follows that  $\{\varphi_{n,T}\}_{n\in\N}$ is a Cauchy sequence in $L^2(0,\infty)$. Thus,
\begin{equation}\label{e16_1}
\varphi_{n,T} \rightarrow \varphi_{T} \mbox{ in } L^2(\R^-_x), \mbox{ as } n\rightarrow\infty,
\end{equation}
which implies that 
\begin{equation}\label{e12_1}
 \|\varphi_{T}\|_{L^2(\R^-_x)}=1.
\end{equation}
On the other hand, by using the derivative traces estimation of the system \eqref{adjointleft}, \eqref{e16_1}, and \eqref{e3_1}, we see that
\begin{equation*}
 \partial_x \varphi_{n}(0,\cdot) \rightarrow \partial_x\varphi(0,\cdot)   \quad \mbox{ in } L^2(0,T), \mbox{ as } n\rightarrow\infty, 
\end{equation*}
and
\begin{equation*}
 \partial_x^2\varphi_{n}(0,\cdot) \rightarrow \partial_x^2\varphi(0,\cdot)\quad  \mbox{ in } H^{-\frac13}(0,T), \mbox{ as } n\rightarrow\infty.
 \end{equation*}
Finally, taking $n\to\infty$, from above converges, we obtain that $\varphi$ is solution of the system
$$
\begin{cases}
\partial_t\varphi+ \partial_x\varphi+\partial_x^3\varphi=0, & \text { for }(x, t) \in(-\infty,0) \times(0, T),\\ 
\varphi(0, t)=\partial_x\varphi(0,t)=0,& \text { for } t \in(0, T), \\
\varphi(x, T)=\varphi_T(x), & \text { for } x \in(-\infty,0),
\end{cases}
$$
or equivalently, 
\begin{equation*}
\begin{cases}
\partial_t\varphi+ \partial_x\varphi+\partial_x^3\varphi=0, & \text { for }(x, t) \in(0,\infty) \times(0, T),\\ 
\varphi(0, t)=\partial_x\varphi(0,t)=0,& \text { for } t \in(0, T),\\
\varphi(x, 0)=\varphi_0(x), & \text { for } x \in(0,\infty).
\end{cases}
\end{equation*} 
Notice that \eqref{e12_1} implies that the solutions can not be identically zero. However, from the following Lemma, one can conclude that $\varphi=0$, which leads us to a contradiction with \eqref{e12_1}. 
\end{proof}

\begin{lemma}\label{lemaespectralcaso1}
For any $T>0$, let $N_T$ denote the space of the initial states $\varphi_0 \in L^2(\R^-_x)$, such that the solution of \eqref{adjoint_1_left} satisfies $\partial_x\varphi(0,\cdot)=0$. Then, $N_T=\{0\}$.
\end{lemma}

\begin{proof}
The proof uses the same arguments as those given in \cite{Rosier}, that is if $N_T \neq\{0\}$, the map $\varphi_T \in N_T \rightarrow B\left(\varphi_T\right) \in \mathbb{C} N_T$ (where $\mathbb{C} N_T$ denote the complexification of $N_T$) has (at least) one eigenvalue. Hence, there exist $\lambda \in \mathbb{C}$ and $\varphi_0 \in H^3(\R^-_x) \backslash\{0\}$, such that
\begin{equation}\label{spectralcase1}
\begin{cases}
\lambda\varphi_T+\varphi_T^{\prime}+\varphi_T^{\prime  \prime \prime}=0,& \text{in $ \R^-_x$},\\
\varphi_T(0)=\varphi_T^{\prime}(0)=0.
\end{cases}
\end{equation}
To conclude the proof of the Lemma \ref{lemaespectralcaso1}, we prove that this does not hold. To simplify the notation, henceforth, we denote $\varphi_T:=\varphi$.  Consider
$$
\begin{cases}
\lambda\varphi+\varphi^{\prime}+\varphi^{\prime  \prime \prime}=0, \quad \, \text{in $ (-\infty,0)$},\\
\varphi(0)=\varphi^{\prime}(0)=0,
\end{cases}
$$
with $\varphi\neq 0$. Note that if $\varphi^{\prime\prime}(0)=0$, then $\varphi \equiv 0$. Otherwise,  we use an argument similar to the one used in  \cite[Lemma 3.5]{Rosier}. Let us introduce the notation $\hat{\varphi}(\xi)=\int_{-\infty}^0  \mathrm{e}^{-i x \xi} \varphi(x) \mathrm{d} x. $ Note that the above representation has the following properties:
\begin{align*}
&\widehat{\varphi^{\prime}}(\xi)= \int_{-\infty}^0 \mathrm{e}^{-i x \xi} \varphi^{\prime}(x) \mathrm{d} x =  i\xi \widehat{\psi} + [e^{-ix \xi}\varphi]_{x=-\infty}^{x=0}= i\xi \widehat{\varphi} + \varphi(0)=i\xi \widehat{\varphi},
\end{align*}
and
\begin{align*}
\widehat{\varphi^{\prime\prime\prime}}(\xi) &= \int_{-\infty}^{0} \mathrm{e}^{-i x \xi} \varphi^{\prime\prime\prime}(x) \mathrm{d} x = i \xi^3 \widehat{\varphi} + [-\xi^2e^{-ix \xi}\varphi+i\xi e^{-ix \xi}\varphi^{\prime}+e^{-ix \xi}\varphi^{\prime\prime}]_{x=-\infty}^{x=0} \\
&= - i \xi^3 \widehat{\varphi} -\xi^2\varphi(0)+i\xi \varphi^{\prime}(0)+\varphi^{\prime\prime}(0)= - i \xi^3 \widehat{\varphi} +\varphi^{\prime\prime}(0).
\end{align*}
 Then, multiplying the equation in \eqref{spectralcase1} by $\mathrm{e}^{-i x \xi}$ and integrating by part in $(-\infty,0)$ yields
 \begin{equation}\label{rrrr}
\widehat{\varphi}(\xi)=  \dfrac{\varphi^{\prime\prime}(0)}{\lambda -i\xi- i\xi^3}.
 \end{equation}
Using Paley-Wiener theorem (see, for instance, \cite[Section 4, p. 161]{yosida}) and the usual characterization of $H^3(-\infty,0)$ functions using their Fourier transforms, we see that a nontrivial solution of \eqref{spectralcase1} is equivalent to the existence of $\lambda \in \mathbb{C}$, such that
 \begin{enumerate}
 \item[(i)] $\widehat{\varphi}$ is a entire functions in $\mathbb{C}$,
 \item[(ii)]  $\int_{\mathbb{R}}|\widehat{\varphi}(\xi)|^2\left(1+|\xi|^2\right)^2 \mathrm{~d} \xi<\infty$, 
 \item[(iii)] $\forall \xi \in \mathbb{C}$, we have that $\widehat{\varphi}|(\xi)| \leq c_1(1+|\xi|)^k \mathrm{e}^{L|\operatorname{Im} \xi|}$, for some positive constant $c_1$.
 \end{enumerate}
Notice that if (i) holds, then (ii) and (iii) are satisfied; however, since $\widehat{\varphi}$ is given by \eqref{rrrr}, it can not be an entire function. 
\end{proof}
As usual in control theory (see the previous section), with the previous observability inequality in hand, the following controllability result for the linearized system holds. 
\begin{theorem}\label{controllinearonecontrolneuman}
Let $T>0$, $\phi, \phi_T \in L^2(\R^-_x)$. Then, there exist $g_2 \in L^{2}(\R^+_t)$ such that the distributional solution $u$ of \eqref{linearonecontrolNeumann}
satisfies that $u(x,T)=\phi_T(x)$ for $x \in \R^-_x$. Moreover, the following estimates hold
\begin{equation*}
\left\| u   \right\|_{C(\R^+_t, L^2(\R^-_x))} \leq C \left( \|\phi\|_{L^2(\R^-_x)} + \|g_2\|_{L^{2}(\R^+_t)}\right).
\end{equation*}
\end{theorem}

\begin{remark}Note that similar results can be obtained for the system \eqref{linearonecontroldirichlet}.  Indeed, we have the following observability inequality for the solution of the adjoint system associated with the system \eqref{linearonecontroldirichlet}: 
$$
\left\|\varphi_T\right\|_{L^2(\R^-_x)}^2 \leq  C \left\|\partial_x^2\varphi(0,\cdot)\right\|_{H^{-\frac{1}{3}}(0, T)}^2,
$$
for any  $\varphi_T \in L^2(\R^-_x)$, where $\varphi$ is the solution of the backward system \eqref{adjointleft}. So, with this in hand, the following holds:
\begin{theorem}\label{controllinearonecontroldirichelt}
Let $T>0$, $\phi, \phi_T \in L^2(\R^-_x)$. Then, there exist $g_1 \in H^{\frac{1}{3}}(\R^+_t)$ such that the distributional solution $u$ of \eqref{linearonecontroldirichlet} satisfies that $u(x,T)=\phi_T(x)$ for $x \in \R^-_x$. Moreover, the following estimates hold
\begin{equation*}
\left\| u   \right\|_{C(\R^+_t, L^2(\R^-_x))} \leq C \left( \|\phi\|_{L^2(\R^-_x)} + \|g_1\|_{H^{\frac{1}{3}}(\R^+_t)}\right).
\end{equation*}
\end{theorem}
\end{remark}

\subsection{Controllability of the nonlinear systems}
Let $T > 0$, from Theorems \ref{controllinearonecontrolneuman} and \ref{controllinearonecontroldirichelt}, we can define the bounded linear operators
\begin{equation*}
\begin{array}{c c c c}
\Lambda_{l,D}: & L^2(\R^-_x) \times L^2(\R^-_x)  & \longrightarrow &H^{\frac{1}{3}}(\R^+_t) \\
& (\phi,\phi_T) & \longrightarrow & \Lambda_{l,D}(\phi,\phi_T)=g_1
\end{array}
\end{equation*}
and
\begin{equation*}
\begin{array}{c c c c}
\Lambda_{l,N}: & L^2(\R^-_x) \times L^2(\R^-_x)  & \longrightarrow & L^{2}(\R^+_t) \\
& (\phi,\phi_T) & \longrightarrow & \Lambda_{l,N}(\phi,\phi_T)=g_2,
\end{array}
\end{equation*}
where $g_1$ and $g_2$ are the controls of the linear system  \eqref{linearonecontroldirichlet} and \eqref{linearonecontrolNeumann}, respectively. Now, we are in a position to prove Theorem \ref{main-b}.

\subsubsection{\bf Proof of Theorem \ref{main-b}} 
Since the proof of this theorem is analogous to what was done before in Theorem \ref{main-a}, that is, the nonlinear problem is treated using a classical fixed-point argument, for the sake of completeness, we will give the sketch of the proof. According to \eqref{formulasolutionleft} the solution of \eqref{h2} can be written as
$$
u(x, t)=\theta(t) e^{-t (\partial_x+\partial_x^3)} \phi(x)-\frac{1}{2} \theta(t) \mathscr{D} \partial_x u^2(x, t)+\theta(t) \mathscr{L}_{-}^{\lambda_1} h_1(x, t,\phi)+\theta(t) \mathscr{L}_{-}^{\lambda_2} h_2(x, t,\phi),
$$
with 
$$
\left[\begin{array}{l}
h_1(t,\phi) \\
h_2(t,\phi)
\end{array}\right]=M\left[\begin{array}{c}
-\left.\theta(t) e^{-t (\partial_x+ \partial_x^3)} \phi\right|_{x=0}+\frac{1}{2} \theta(t) \mathscr{D} \partial_x u^2(0, t) \\
\theta(t) \mathcal{I}_{1 / 3}\left(g_2-\left.\theta \partial_x e^{-t (\partial_x+\partial_x^3)} \phi\right|_{x=0}+\frac{1}{2} \theta \partial_x \mathscr{D} \partial_x u^2(0, \cdot)\right)(t)
\end{array}\right],
$$
where $M$ is a matrix given by \eqref{matrixA}. 
 Consider the map 
\begin{align*}
\Gamma(u):= \theta(t) e^{-t (\partial_x+\partial_x^3)} \phi(x)-\frac{1}{2} \theta(t) \mathscr{D} \partial_{x} u^2(x, t)+\theta(t) \mathscr{L}_{-}^\lambda \widehat{h}_{1,T}(x, t)+\theta(t) \mathscr{L}_{-}^\lambda \widehat{h}_{2,T}(x, t). 
\end{align*}
In this case, $\widehat{h}^u_{1,T}(t)=h_1(t,\varphi)$ and $\widehat{h}^u_{2,T}(t)=h_2(t,\varphi)$ with 
$$
g_2(t)= \Lambda_{l,N}\left( \phi, \phi_T +\frac{1}{2} \theta(T) \mathscr{D} \partial_{x} u^2(x, T)  \right)(t). 
$$
Thanks to the Theorem \ref{controllinearonecontrolneuman}, we get that
$$
\Gamma(u)|_{t=0}=\phi \quad \text{and} \quad \Gamma(u)|_{t=T}=\phi_T.
$$
Remember the closed ball in X, 
$$
B_r=\left\lbrace u \in X_{0,b}\cap D_{\alpha}: \|u\|_{X_{0,b}\cap D_{\alpha}} \leq r\right\rbrace, 
$$
for some $r>0$.

\vspace{0.2cm}

Note, first, that $\Gamma$ maps $B_r$ into itself. Indeed, using Lemmas \ref{estimatesgroup}, \ref{estimateinhom} and \ref{estimateforcing}, there exists a constant $C > 0$, such that
\begin{align*}
\|\Gamma(u)\|_{X_{0,b}\cap D_{\alpha}} 
&\leq C\left( \left\|\theta  \right\|_{H^1(\R^+_t)}\left\| \phi \right\|_{L^2(\R^-_x)} + \left\| \partial_x u^2 \right\|_{X_{0,-b}} + \|\widehat{h}^u_{1,T}\|_{H^{\frac{1}{3}}(\R_t^+)}+\|\widehat{h}^u_{2,T}\|_{H^{\frac{1}{3}}(\R_t^+)}\right).
\end{align*}
From \cite[Lemma 5.10]{Holmer},  it follows that 
\begin{align*}
 \left\| \partial_x u^2 \right\|_{X_{0,-b}} \leq  C\left\| u \right\|_{X_{0,b}\cap D_{\alpha}}^2 \leq Cr^2.
\end{align*}
Note that from definition of $\Gamma$, the function $\widehat{h_1}$ and $\widehat{h}_2$ are given by
\begin{equation}\label{formula1}
\begin{split}
\widehat{h}^u_{1,T}(t)= &a_{1,1}\left( \left.\theta(t) e^{-t (\partial_x+ \partial_x^3)} \phi\right|_{x=0}+\frac{1}{2} \theta(t) \mathscr{D} \partial_x u^2(0, t) \right) \\
&+a_{1,2}\left(\theta \mathcal{I}_{1 / 3}\left(\Lambda_{l,N}\left( \phi, \phi_T +\frac{1}{2} \theta(T) \mathscr{D} \partial_{x} u^2(x, T)  \right)(t) \right.\right. \\
&\left.\left.-\left.\theta \partial_x e^{-t (\partial_x+\partial_x^3)} \phi\right|_{x=0}+\frac{1}{2} \theta \partial_x \mathscr{D} \partial_x u^2(0, \cdot)\right)(t)\right)
\end{split}
\end{equation}
and
\begin{equation}\label{formula2}
\begin{split}
\widehat{h}^u_{2,T}(t)= &a_{2,1}\left( \left.\theta(t) e^{-t (\partial_x+ \partial_x^3)} \phi\right|_{x=0}+\frac{1}{2} \theta(t) \mathscr{D} \partial_x u^2(0, t) \right) \\
&+a_{2,2}\left(\theta\mathcal{I}_{1 / 3}\left(\Lambda_{l,N}\left( \phi, \phi_T +\frac{1}{2} \theta(T) \mathscr{D} \partial_{x} u^2(x, T)  \right)(t) \right.\right. \\
&\left.\left.-\left.\theta \partial_x e^{-t (\partial_x+\partial_x^3)} \phi\right|_{x=0}+\frac{1}{2} \theta \partial_x \mathscr{D} \partial_x u^2(0, \cdot)\right)(t)\right),
\end{split}
\end{equation}
where 
\begin{equation*}
\begin{cases}
a_{1,1} = \frac{\sin \left(\frac{\pi}{3} \lambda_2-\frac{\pi}{6}\right)}{2 \sqrt{3} \sin \left[\frac{\pi}{3}\left(\lambda_2-\lambda_1\right)\right]} & a_{1,2}=- \frac{\sin \left(\frac{\pi}{3} \lambda_2+\frac{\pi}{6}\right)}{2 \sqrt{3} \sin \left[\frac{\pi}{3}\left(\lambda_2-\lambda_1\right)\right]} \\
a_{2,1} = \frac{\sin \left(\frac{\pi}{3} \lambda_1-\frac{\pi}{6}\right)}{2 \sqrt{3} \sin \left[\frac{\pi}{3}\left(\lambda_2-\lambda_1\right)\right]} & a_{2,2}=- \frac{\sin \left(\frac{\pi}{3} \lambda_1+\frac{\pi}{6}\right)}{2 \sqrt{3} \sin \left[\frac{\pi}{3}\left(\lambda_2-\lambda_1\right)\right]}.
\end{cases}
\end{equation*}
Thus, for $i=1,2$, we get
\begin{align*}
 \|\widehat{h}_{i,T}\|_{H^{\frac{1}{3}}(\R_t^+)} &\leq C\left( \left\|\left.\theta(t) e^{-t(\partial_x+ \partial_x^3)} \phi\right|_{x=0}\right\|_{H^{\frac{1}{3}}(\R_t^+)}+\left\|\frac{1}{2} \theta(t) \mathscr{D} \partial_x u^2(0, t)\right\|_{H^{\frac{1}{3}}(\R_t^+)}  \right.  \\
 &+  \left\| \mathcal{I}_{1 / 3}\Lambda_{l,N}\left( \phi, \phi_T +\frac{1}{2} \theta(T) \mathscr{D} \partial_{x} u^2(x, T)  \right)(t) \right\|_{H^{\frac{1}{3}}(\R_t^+)}  \\
  &\left. +  \left\| \mathcal{I}_{1 / 3} \left(-\theta\left. \partial_x e^{-t (\partial_x+\partial_x^3)} \phi\right|_{x=0}+\frac{1}{2} \theta \partial_x \mathscr{D} \partial_x u^2(0, \cdot)\right)(t)\right\|_{H^{\frac{1}{3}}(\R_t^+)}  \right) .
\end{align*}
Note that 
\begin{align*}
 \left\|\left.\theta(t) e^{-t(\partial_x+ \partial_x^3)} \phi\right|_{x=0}\right\|_{H^{\frac{1}{3}}(\R_t^+)}+\left\|\frac{1}{2} \theta(t) \mathscr{D} \partial_x u^2(0, t)\right\|_{H^{\frac{1}{3}}(\R_t^+)} 
 & \leq C\delta +2Cr^2 
\end{align*}
and
\begin{equation*}
\left\| \mathcal{I}_{1 / 3}\Lambda_{l,N}\left( \phi, \phi_T +\frac{1}{2} \theta(T) \mathscr{D} \partial_{x} u^2(x, T)  \right)(t) \right\|_{H^{\frac{1}{3}}(\R_t^+)} 
  \leq \left\| \Lambda_{l,N}\right\| \delta  +2C  \left\| \Lambda_{l,N}\right\|r^2. 
\end{equation*}
Therefore, we have that 
\begin{align*}
 \|\widehat{h}_{i,T}\|_{H^{\frac{1}{3}}(\R_t^+)} &\leq \left( C  +\left\| \Lambda_{l,N}\right\|\right) \delta  +2C \left(1+ \left\| \Lambda_{l,N}\right\|\right) r^2. 
\end{align*}
Finally, we have that 
\begin{align*}
\|\Gamma(u)\|_{X_{s,b}\cap D_{\alpha}} &\leq C \left( \left\|\theta  \right\|_{H^1(\R^+_t)} + \left\| \Lambda_{l,N}\right\|+ 2C\right)\delta +
C^2\left(6+ \left\| \Lambda_{l,N}\right\|\right) r^2 ,
\end{align*}
showing that $\Gamma$ maps $B_r$ into itself.

\vspace{0.2cm}

Now on, let us show that $\Gamma$ is a contraction. To do that, let $u, v  \in B_r$, and consider 
\begin{equation*}
\begin{split}
\Gamma(u) - \Gamma(v)= &-\frac{1}{2} \theta(t) \mathscr{D} \partial_{x} \left( u^2(x, t) - v^2(x,t)\right)+\theta(t) \mathscr{L}_{-}^\lambda \left( \widehat{h}_{1,T}^u(x, t)- \widehat{h}_{1,T}^v(x, t) \right)\\&+\theta(t) \mathscr{L}_{-}^\lambda \left( \widehat{h}_{2,T}^u(x, t)- \widehat{h}_{2,T}^v(x, t)\right).
\end{split}
\end{equation*}
Here  $ \widehat{h}_{i,T}^u$ and $ \widehat{h}_{i,T}^v$ for $i=1,2$ are given by \eqref{formula1} and \eqref{formula2}, respectively.
Observing that 
\begin{equation*}
\begin{split}
\widehat{h}_{i,T}^u(t)- \widehat{h}_{i,T}^v(t)=& a_{i,1}\left( \frac{1}{2} \theta(t) \mathscr{D} \partial_x \left( u^2(0, t)-v^2(0,t)\right)\right) \\
&+a_{i,2}\left(\theta\mathcal{I}_{1 / 3}\left(\Lambda_{l,N}\left( 0, \frac{1}{2} \theta(T) \mathscr{D} \partial_{x} \left( u^2(x, T)- v^2(x,T)\right)  \right)(t) \right.\right. \\
&\left.\left. +\frac{1}{2} \theta \partial_x \mathscr{D} \partial_x \left( u^2(0, t)-v^2(0,t)\right)\right)(t)\right).
\end{split}
\end{equation*}
we have 
$$
\left\| \Gamma(u) - \Gamma(v)  \right\|_{X_{0,b}\cap D_{\alpha}} 
\leq 4C^2\left( 5+\left\|\Lambda_{l,N} \right\|\right)r \left\| u- v\right\|_{X_{0,b}\cap D_{\alpha}},
$$
and taking $r$ such that $2C^2\left( 4+\left\|\Lambda_{l,n} \right\|\right)r<1$, $\Gamma$ is a contractive. Therefore, the results are obtained by using 
Banach fixed-point theorem.
The proof of the Theorem is complete. \qed

\section{Further comments and perspectives}\label{sec6}
Let us present some comments and perspectives on our analysis in this work. 

\subsection{Comments and novelties} There are important facts about the theorems presented in this work. We mention some of them below.
\begin{itemize}
\item[i.] Given that operational controllability can be extended to encompass bounded intervals, with all operators effectively operating within these constraints, we can similarly extend Theorems \ref{main-a} and \ref{main-b} to the bounded cases, provided suitable boundary conditions are met.
\item[ii.] Regarding the critical set phenomenon, it is notable that several authors have explored this aspect when considering the KdV equation on a bounded domain. We highlight two specific works. In references \cite{Rosier} and \cite{GG1}, the authors delve into this phenomenon for two sets of boundary conditions. However, in our case, where fewer boundary conditions are imposed, and the drift term ($\partial_xu$) is included in the equation, the critical set phenomenon does not manifest. This represents a novel observation for the KdV equation. Specifically, we can maintain the same control positioning, either $\partial_xu(L,t)=f(t)$ or $u(0,t)=g_1(t)$.
\item[iii.] Still concerning the critical set phenomenon, in a recent work \cite{CaSi}, the authors developed a refined analysis of the critical length phenomenon for the nonlinear Korteweg–de Vries equation under Neumann-type boundary conditions. The paper combines sharp well-posedness theory, hidden regularity estimates, fixed-point arguments, and the return method in order to recover controllability in situations where the linearized system fails due to the presence of critical lengths. We also emphasize that several other contributions have addressed related questions from different perspectives. For a broader overview of the subject and further developments concerning controllability and critical length phenomena for dispersive equations, we refer the reader to the surveys and related works \cite{Capistrano,cerpatut}.
\end{itemize}

The scenario presented in this work is completely new in terms of control theory. Employing a comprehensive approach enables us to formulate a broader relationship, thereby achieving precise control, denoted by $f$, which is essential for ensuring the exact controllability of the system under consideration, as stated in remark \ref{main-c}. Our analysis is illustrated in the scheme below (see Figure \ref{figurateorema} below), which shows the range of $s$ in the control space $H^s$ for which the relation presented in remark \ref{main-c} remains valid. 

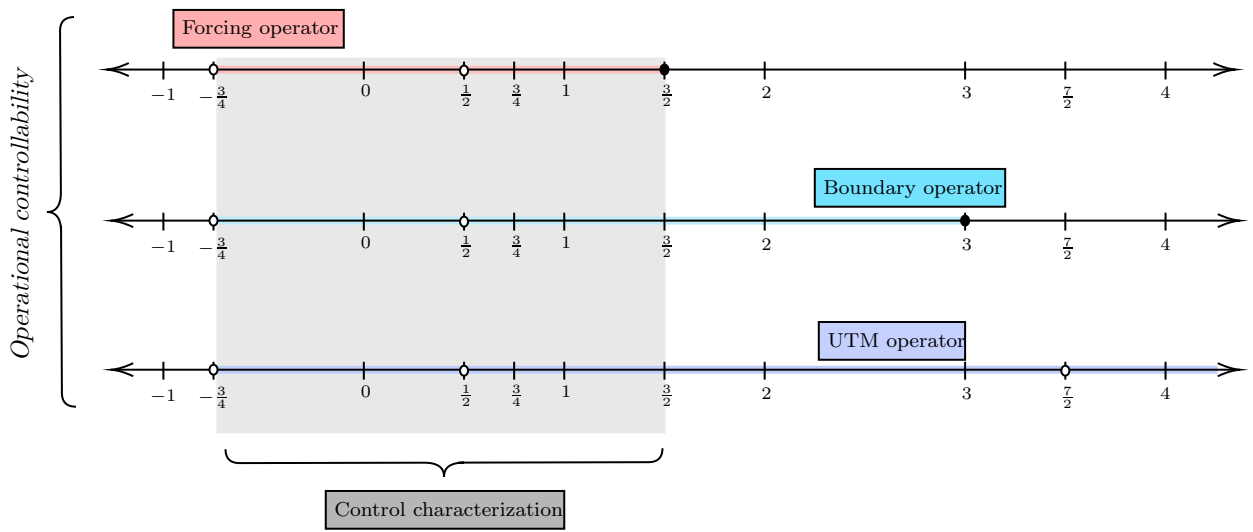
\begin{figure}[hbt]

\tikzset{every picture/.style={line width=0.75pt}} 

\begin{tikzpicture}[x=0.75pt,y=0.75pt,yscale=-1,xscale=1]

\draw  [color={rgb, 255:red, 255; green, 255; blue, 255 }  ,draw opacity=1 ][fill={rgb, 255:red, 232; green, 232; blue, 232 }  ,fill opacity=0.57 ] (112.5,32) -- (288,32) -- (288,222) -- (112.5,222) -- cycle ;
\draw [color={rgb, 255:red, 192; green, 206; blue, 251 }  ,draw opacity=1 ][line width=3]    (112.5,187) -- (613.5,187) ;
\draw    (63.5,187) -- (622.5,187) ;
\draw [shift={(624.5,187)}, rotate = 180] [color={rgb, 255:red, 0; green, 0; blue, 0 }  ][line width=0.75]    (10.93,-3.29) .. controls (6.95,-1.4) and (3.31,-0.3) .. (0,0) .. controls (3.31,0.3) and (6.95,1.4) .. (10.93,3.29)   ;
\draw [shift={(61.5,187)}, rotate = 0] [color={rgb, 255:red, 0; green, 0; blue, 0 }  ][line width=0.75]    (10.93,-3.29) .. controls (6.95,-1.4) and (3.31,-0.3) .. (0,0) .. controls (3.31,0.3) and (6.95,1.4) .. (10.93,3.29)   ;
\draw [color={rgb, 255:red, 192; green, 237; blue, 251 }  ,draw opacity=1 ][line width=3]    (112.5,112) -- (487.5,112) ;
\draw    (63.5,112) -- (622.5,112) ;
\draw [shift={(624.5,112)}, rotate = 180] [color={rgb, 255:red, 0; green, 0; blue, 0 }  ][line width=0.75]    (10.93,-3.29) .. controls (6.95,-1.4) and (3.31,-0.3) .. (0,0) .. controls (3.31,0.3) and (6.95,1.4) .. (10.93,3.29)   ;
\draw [shift={(61.5,112)}, rotate = 0] [color={rgb, 255:red, 0; green, 0; blue, 0 }  ][line width=0.75]    (10.93,-3.29) .. controls (6.95,-1.4) and (3.31,-0.3) .. (0,0) .. controls (3.31,0.3) and (6.95,1.4) .. (10.93,3.29)   ;
\draw [color={rgb, 255:red, 251; green, 192; blue, 192 }  ,draw opacity=1 ][line width=3]    (112.5,36) -- (337.5,36) ;
\draw    (61.29,36) -- (115.71,36) -- (620.29,36) ;
\draw [shift={(622.29,36)}, rotate = 180] [color={rgb, 255:red, 0; green, 0; blue, 0 }  ][line width=0.75]    (10.93,-3.29) .. controls (6.95,-1.4) and (3.31,-0.3) .. (0,0) .. controls (3.31,0.3) and (6.95,1.4) .. (10.93,3.29)   ;
\draw [shift={(59.29,36)}, rotate = 0] [color={rgb, 255:red, 0; green, 0; blue, 0 }  ][line width=0.75]    (10.93,-3.29) .. controls (6.95,-1.4) and (3.31,-0.3) .. (0,0) .. controls (3.31,0.3) and (6.95,1.4) .. (10.93,3.29)   ;
\draw    (187.5,32) -- (187.5,41) ;
\draw    (287.5,32) -- (287.5,41) ;
\draw    (387.5,32) -- (387.5,41) ;
\draw    (487.5,32) -- (487.5,41) ;
\draw    (587.5,32) -- (587.5,41) ;
\draw    (87.5,32) -- (87.5,41) ;
\draw    (112.5,32) -- (112.5,41) ;
\draw    (237.5,32) -- (237.5,41) ;
\draw    (537.5,32) -- (537.5,41) ;
\draw    (337.5,32) -- (337.5,41) ;
\draw    (262.5,32) -- (262.5,41) ;
\draw    (187.5,108) -- (187.5,117) ;
\draw    (287.5,108) -- (287.5,117) ;
\draw    (387.5,108) -- (387.5,117) ;
\draw    (487.5,108) -- (487.5,117) ;
\draw    (587.5,108) -- (587.5,117) ;
\draw    (87.5,108) -- (87.5,117) ;
\draw    (112.5,108) -- (112.5,117) ;
\draw    (237.5,108) -- (237.5,117) ;
\draw    (537.5,108) -- (537.5,117) ;
\draw    (337.5,108) -- (337.5,117) ;
\draw    (262.5,108) -- (262.5,117) ;
\draw    (187.5,183) -- (187.5,192) ;
\draw    (287.5,183) -- (287.5,192) ;
\draw    (387.5,183) -- (387.5,192) ;
\draw    (487.5,183) -- (487.5,192) ;
\draw    (587.5,183) -- (587.5,192) ;
\draw    (87.5,183) -- (87.5,192) ;
\draw    (112.5,183) -- (112.5,192) ;
\draw    (237.5,183) -- (237.5,192) ;
\draw    (537.5,183) -- (537.5,192) ;
\draw    (337.5,183) -- (337.5,192) ;
\draw    (262.5,183) -- (262.5,192) ;
\draw  [fill={rgb, 255:red, 255; green, 255; blue, 255 }  ,fill opacity=1 ] (239.5,36.5) .. controls (239.5,35.12) and (238.6,34) .. (237.5,34) .. controls (236.4,34) and (235.5,35.12) .. (235.5,36.5) .. controls (235.5,37.88) and (236.4,39) .. (237.5,39) .. controls (238.6,39) and (239.5,37.88) .. (239.5,36.5) -- cycle ;
\draw  [fill={rgb, 255:red, 3; green, 0; blue, 0 }  ,fill opacity=1 ] (339.5,36) .. controls (339.5,34.62) and (338.6,33.5) .. (337.5,33.5) .. controls (336.4,33.5) and (335.5,34.62) .. (335.5,36) .. controls (335.5,37.38) and (336.4,38.5) .. (337.5,38.5) .. controls (338.6,38.5) and (339.5,37.38) .. (339.5,36) -- cycle ;
\draw  [fill={rgb, 255:red, 255; green, 255; blue, 255 }  ,fill opacity=1 ] (239.5,112.5) .. controls (239.5,111.12) and (238.6,110) .. (237.5,110) .. controls (236.4,110) and (235.5,111.12) .. (235.5,112.5) .. controls (235.5,113.88) and (236.4,115) .. (237.5,115) .. controls (238.6,115) and (239.5,113.88) .. (239.5,112.5) -- cycle ;
\draw  [fill={rgb, 255:red, 3; green, 0; blue, 0 }  ,fill opacity=1 ] (489.5,112) .. controls (489.5,110.62) and (488.6,109.5) .. (487.5,109.5) .. controls (486.4,109.5) and (485.5,110.62) .. (485.5,112) .. controls (485.5,113.38) and (486.4,114.5) .. (487.5,114.5) .. controls (488.6,114.5) and (489.5,113.38) .. (489.5,112) -- cycle ;
\draw  [fill={rgb, 255:red, 255; green, 255; blue, 255 }  ,fill opacity=1 ] (239.5,187.5) .. controls (239.5,186.12) and (238.6,185) .. (237.5,185) .. controls (236.4,185) and (235.5,186.12) .. (235.5,187.5) .. controls (235.5,188.88) and (236.4,190) .. (237.5,190) .. controls (238.6,190) and (239.5,188.88) .. (239.5,187.5) -- cycle ;
\draw  [fill={rgb, 255:red, 255; green, 255; blue, 255 }  ,fill opacity=1 ] (539.5,187.5) .. controls (539.5,186.12) and (538.6,185) .. (537.5,185) .. controls (536.4,185) and (535.5,186.12) .. (535.5,187.5) .. controls (535.5,188.88) and (536.4,190) .. (537.5,190) .. controls (538.6,190) and (539.5,188.88) .. (539.5,187.5) -- cycle ;
\draw   (42.9,9.6) .. controls (38.23,9.63) and (35.91,11.97) .. (35.94,16.64) -- (36.35,97.64) .. controls (36.38,104.31) and (34.07,107.65) .. (29.4,107.67) .. controls (34.07,107.65) and (36.42,110.97) .. (36.45,117.64)(36.43,114.64) -- (36.86,198.64) .. controls (36.88,203.31) and (39.22,205.63) .. (43.89,205.6) ;
\draw   (118.5,227) .. controls (118.5,231.67) and (120.83,234) .. (125.5,234) -- (191.75,234) .. controls (198.42,234) and (201.75,236.33) .. (201.75,241) .. controls (201.75,236.33) and (205.08,234) .. (211.75,234)(208.75,234) -- (278,234) .. controls (282.67,234) and (285,231.67) .. (285,227) ;
\draw  [fill={rgb, 255:red, 255; green, 255; blue, 255 }  ,fill opacity=1 ] (114.5,36) .. controls (114.5,34.62) and (113.6,33.5) .. (112.5,33.5) .. controls (111.4,33.5) and (110.5,34.62) .. (110.5,36) .. controls (110.5,37.38) and (111.4,38.5) .. (112.5,38.5) .. controls (113.6,38.5) and (114.5,37.38) .. (114.5,36) -- cycle ;
\draw  [fill={rgb, 255:red, 255; green, 255; blue, 255 }  ,fill opacity=1 ] (114.5,112) .. controls (114.5,110.62) and (113.6,109.5) .. (112.5,109.5) .. controls (111.4,109.5) and (110.5,110.62) .. (110.5,112) .. controls (110.5,113.38) and (111.4,114.5) .. (112.5,114.5) .. controls (113.6,114.5) and (114.5,113.38) .. (114.5,112) -- cycle ;
\draw  [fill={rgb, 255:red, 255; green, 255; blue, 255 }  ,fill opacity=1 ] (114.5,187) .. controls (114.5,185.62) and (113.6,184.5) .. (112.5,184.5) .. controls (111.4,184.5) and (110.5,185.62) .. (110.5,187) .. controls (110.5,188.38) and (111.4,189.5) .. (112.5,189.5) .. controls (113.6,189.5) and (114.5,188.38) .. (114.5,187) -- cycle ;
\draw  [fill={rgb, 255:red, 255; green, 255; blue, 255 }  ,fill opacity=1 ] (289.5,36.5) .. controls (289.5,35.12) and (288.6,34) .. (287.5,34) .. controls (286.4,34) and (285.5,35.12) .. (285.5,36.5) .. controls (285.5,37.88) and (286.4,39) .. (287.5,39) .. controls (288.6,39) and (289.5,37.88) .. (289.5,36.5) -- cycle ;
\draw  [fill={rgb, 255:red, 255; green, 255; blue, 255 }  ,fill opacity=1 ] (289.5,112.5) .. controls (289.5,111.12) and (288.6,110) .. (287.5,110) .. controls (286.4,110) and (285.5,111.12) .. (285.5,112.5) .. controls (285.5,113.88) and (286.4,115) .. (287.5,115) .. controls (288.6,115) and (289.5,113.88) .. (289.5,112.5) -- cycle ;
\draw  [fill={rgb, 255:red, 255; green, 255; blue, 255 }  ,fill opacity=1 ] (289.5,187.5) .. controls (289.5,186.12) and (288.6,185) .. (287.5,185) .. controls (286.4,185) and (285.5,186.12) .. (285.5,187.5) .. controls (285.5,188.88) and (286.4,190) .. (287.5,190) .. controls (288.6,190) and (289.5,188.88) .. (289.5,187.5) -- cycle ;

\draw (103.5,43.4) node [anchor=north west][inner sep=0.75pt]  [font=\tiny]  {$-\frac{3}{4}$};
\draw (284.5,42.4) node [anchor=north west][inner sep=0.75pt]  [font=\tiny]  {$1$};
\draw (233.5,42.4) node [anchor=north west][inner sep=0.75pt]  [font=\tiny]  {$\frac{1}{2}$};
\draw (484.5,43.4) node [anchor=north west][inner sep=0.75pt]  [font=\tiny]  {$3$};
\draw (533.5,43.4) node [anchor=north west][inner sep=0.75pt]  [font=\tiny]  {$\frac{7}{2}$};
\draw (333.5,42.4) node [anchor=north west][inner sep=0.75pt]  [font=\tiny]  {$\frac{3}{2}$};
\draw (258.5,42.4) node [anchor=north west][inner sep=0.75pt]  [font=\tiny]  {$\frac{3}{4}$};
\draw (79.5,44.4) node [anchor=north west][inner sep=0.75pt]  [font=\tiny]  {$-1$};
\draw (384.5,43.4) node [anchor=north west][inner sep=0.75pt]  [font=\tiny]  {$2$};
\draw (184.5,42.4) node [anchor=north west][inner sep=0.75pt]  [font=\tiny]  {$0$};
\draw (583.5,43.4) node [anchor=north west][inner sep=0.75pt]  [font=\tiny]  {$4$};
\draw (103.5,119.4) node [anchor=north west][inner sep=0.75pt]  [font=\tiny]  {$-\frac{3}{4}$};
\draw (284.5,118.4) node [anchor=north west][inner sep=0.75pt]  [font=\tiny]  {$1$};
\draw (233.5,118.4) node [anchor=north west][inner sep=0.75pt]  [font=\tiny]  {$\frac{1}{2}$};
\draw (484.5,119.4) node [anchor=north west][inner sep=0.75pt]  [font=\tiny]  {$3$};
\draw (533.5,119.4) node [anchor=north west][inner sep=0.75pt]  [font=\tiny]  {$\frac{7}{2}$};
\draw (333.5,118.4) node [anchor=north west][inner sep=0.75pt]  [font=\tiny]  {$\frac{3}{2}$};
\draw (258.5,118.4) node [anchor=north west][inner sep=0.75pt]  [font=\tiny]  {$\frac{3}{4}$};
\draw (79.5,120.4) node [anchor=north west][inner sep=0.75pt]  [font=\tiny]  {$-1$};
\draw (384.5,119.4) node [anchor=north west][inner sep=0.75pt]  [font=\tiny]  {$2$};
\draw (184.5,118.4) node [anchor=north west][inner sep=0.75pt]  [font=\tiny]  {$0$};
\draw (583.5,119.4) node [anchor=north west][inner sep=0.75pt]  [font=\tiny]  {$4$};
\draw (103.5,194.4) node [anchor=north west][inner sep=0.75pt]  [font=\tiny]  {$-\frac{3}{4}$};
\draw (284.5,193.4) node [anchor=north west][inner sep=0.75pt]  [font=\tiny]  {$1$};
\draw (233.5,193.4) node [anchor=north west][inner sep=0.75pt]  [font=\tiny]  {$\frac{1}{2}$};
\draw (484.5,194.4) node [anchor=north west][inner sep=0.75pt]  [font=\tiny]  {$3$};
\draw (533.5,194.4) node [anchor=north west][inner sep=0.75pt]  [font=\tiny]  {$\frac{7}{2}$};
\draw (333.5,193.4) node [anchor=north west][inner sep=0.75pt]  [font=\tiny]  {$\frac{3}{2}$};
\draw (258.5,193.4) node [anchor=north west][inner sep=0.75pt]  [font=\tiny]  {$\frac{3}{4}$};
\draw (79.5,195.4) node [anchor=north west][inner sep=0.75pt]  [font=\tiny]  {$-1$};
\draw (384.5,194.4) node [anchor=north west][inner sep=0.75pt]  [font=\tiny]  {$2$};
\draw (184.5,193.4) node [anchor=north west][inner sep=0.75pt]  [font=\tiny]  {$0$};
\draw (583.5,194.4) node [anchor=north west][inner sep=0.75pt]  [font=\tiny]  {$4$};
\draw  [fill={rgb, 255:red, 255; green, 175; blue, 175 }  ,fill opacity=1 ]  (318.5,9) -- (403.5,9) -- (403.5,28) -- (318.5,28) -- cycle  ;
\draw (321.5,13) node [anchor=north west][inner sep=0.75pt]  [font=\scriptsize] [align=left] {Forcing operator};
\draw  [fill={rgb, 255:red, 115; green, 227; blue, 255 }  ,fill opacity=1 ]  (319.5,84) -- (414.5,84) -- (414.5,103) -- (319.5,103) -- cycle  ;
\draw (322.5,88) node [anchor=north west][inner sep=0.75pt]  [font=\scriptsize] [align=left] {Boundary operator};
\draw  [fill={rgb, 255:red, 198; green, 208; blue, 255 }  ,fill opacity=1 ]  (323.5,157) -- (396.5,157) -- (396.5,176) -- (323.5,176) -- cycle  ;
\draw (326.5,161) node [anchor=north west][inner sep=0.75pt]  [font=\scriptsize] [align=left] {UTM operator};
\draw  [fill={rgb, 255:red, 182; green, 182; blue, 182 }  ,fill opacity=1 ]  (142.5,248) -- (261.5,248) -- (261.5,267) -- (142.5,267) -- cycle  ;
\draw (145.5,252) node [anchor=north west][inner sep=0.75pt]  [font=\scriptsize] [align=left] {Operational Control};
\draw (10,181.63) node [anchor=north west][inner sep=0.75pt]  [font=\small,rotate=-270,xslant=0.25] [align=left] {Operational Controllability};
\end{tikzpicture}
\caption{Operational controllability and control characterization relations}\label{figurateorema}
\end{figure}

\subsection{Non-emptiness of $\mathcal{A}^0_r(\phi,T)$ and $\mathcal{A}^0_l(\phi,T)$} As mentioned in the introduction, the admissible final state class $\mathcal{A}^0_r(\phi,T)$ is non-empty for any given $\phi\in L^2(\mathbb{R}^+)$ and $T>0$. 
In fact, to see this, consider the linear initial-boundary value problem with homogeneous boundary conditions:
\begin{equation}\label{eq:linear_homog}
	\begin{cases}
		\partial_t v + \partial_x v + \partial_x^3 v = 0, & (x,t)\in(0,+\infty)\times(0,T),\\
		v(0,t) = 0, & t\in(0,T),\\
		v(x,0) = \phi(x), & x\in(0,+\infty).
	\end{cases}
\end{equation}
By the well-posedness theory for the linear Korteweg--de Vries equation (including the drift term $\partial_x v$) on the right half-line developed by Bona, Sun, and Zhang \cite{Bonahalf1,Bonahalf2}, Colliander and Kenig \cite{CK}, and Holmer \cite{Holmer} (see also Fokas \cite{fokas,fokas0,fokas2} for the unified transform approach), problem \eqref{eq:linear_homog} admits a unique solution $v\in C([0,T];L^2(\mathbb{R}^+))$. The present approach follows Holmer's method of boundary forcing operators, which directly accommodates the drift term without altering the analysis. More generally, for every boundary datum $f\in H^{1/3}(0,T)$, the non-homogeneous problem associated with \eqref{h1} admits a unique solution $u\in C([0,T];L^2(\mathbb{R}^+))$, and the data-to-solution map
\[
(\phi,f)\longmapsto u
\]
is continuous from $L^2(\mathbb{R}^+)\times H^{1/3}(0,T)$ into $C([0,T];L^2(\mathbb{R}^+))$. Returning to \eqref{eq:linear_homog}, the solution $v$ satisfies Definition~\ref{def_1} with $s=0$ and boundary control $f\equiv 0\in H^{1/3}(0,T)$. The temporal trace at $t=T$ yields
\[
\phi_T := v(\cdot,T) \in L^2(\mathbb{R}^+),
\]
and by construction, $\phi_T\in\mathcal{A}^0_r(\phi,T)$. This proves that $\mathcal{A}^0_r(\phi,T)\neq\emptyset$.

An analogous argument applies to the left half-line case associated with \eqref{h2}, where two boundary conditions are required at $x=0$. Consider the linear problem with homogeneous boundary data ($g_1=g_2=0$). For any $\psi\in L^2(\mathbb{R}^-)$, Holmer's well-posedness theory 
\cite{Holmer} (with the drift term $-\partial_x w$ explicitly included in the present work) ensures the existence of a unique solution  $w\in C([0,T];L^2(\mathbb{R}^-))$, and the data-to-solution map $(\psi,g_1,g_2)\longmapsto w $ is continuous from $L^2(\mathbb{R}^-)\times H^{1/3}(0,T)\times H^{-1/3}(0,T)$ into $C([0,T];L^2(\mathbb{R}^-))$. It is worth noting that Bona, Sun, and Zhang's approach does not cover the left half-line case; Holmer's method is particularly well suited to this configuration. Taking $g_1\equiv g_2\equiv 0$, the solution $w$ of linearized system of  \eqref{h2} 
yields $\psi_T:=w(\cdot,T)\in L^2(\mathbb{R}^-)$, and thus $\psi_T\in\mathcal{A}^0_\ell(\psi,T)$, proving that $\mathcal{A}^0_\ell(\psi,T)\neq\emptyset$.

\begin{remark} It is worth placing this result in historical context. Rosier's pioneering work \cite{RosierSICON2000} established exact boundary controllability for the linear KdV equation on the half-line in the weaker framework $L^2_{\mathrm{loc}}((0,+\infty)\times(0,T))$ with boundary controls in distribution spaces of negative order, showing that any $u_0,u_T\in L^2(\mathbb{R}^+)$ can be connected by such a solution. However, if one requires solutions to belong to $L^\infty(0,T;L^2(\mathbb{R}^+))$, controllability fails: Rosier showed there exists $\tilde{\varphi}\in L^2(\mathbb{R}^+)$ such that  $0\notin\mathcal{A}^0_r(\tilde{\varphi},T)$ when the admissible final state class is defined using solutions in $L^\infty(0,T;L^2(\mathbb{R}^+))$. This illustrates the delicate dependence of controllability properties on the functional framework. The well-posedness theory cited above was developed after Rosier's work and provides the natural functional setting $\phi\in L^2(\mathbb{R}^+)$, $f\in H^{1/3}(0,T)$, and $u\in C([0,T];L^2(\mathbb{R}^+))$, which is consistent with the equation's scaling. Within this modern framework, the non-emptiness of $\mathcal{A}^0_r(\phi,T)$ (and similarly $\mathcal{A}^0_\ell(\psi,T)$ for the left half-line) follows from the well-posedness of the homogeneous boundary problem, and the continuous dependence on data ensures the structure needed for fixed-point arguments in the nonlinear problem.
\end{remark}

\subsection{Noncritical length phenomenon} Rosier \cite{Rosier} showed that considering $L\notin\mathcal{N}$, where $\mathcal{N}$ is defined by \eqref{critical}, that the associated linear system \eqref{2}  posed on the bounded interval 
\begin{equation}
\left\{
\begin{array}
[c]{lll}%
\partial_tu+\partial_xu+\partial_x^3u=0, &  & \text{ in } (0,L)\times(0,T),\\
u(0,t)=u(L,t)=0,\text{ }\partial_xu(L,t)=g(t),&  & \text{ in }(0,T),\\
u(x,0)=u_0(x),& & \text{ in }(0,L),
\end{array}
\right.  \label{2a}
\end{equation}
is controllable; roughly speaking, if $L\in\mathcal{N}$ system \eqref{2a} is not controllable, that is,  there exists a finite-dimensional subspace of $L^2(0,L)$, denoted by $\mathcal{M}=\mathcal{M}(L)$, which is unreachable from $0$ for the linear system.  More precisely, for every nonzero state $\psi\in\mathcal{M}$, $g\in L^2(0,T)$  and $u\in C([0,T];L^2(0,L))\cap L^2(0,T;H^1(0,L))$ satisfying \eqref{2a} and $u(\cdot,0)=0$, one has $u(\cdot,T)\neq\psi$. 

\begin{definition}
A spatial domain $(0,L)$ is called \textit{critical}  for the system \eqref{2a} if its domain length $L$ belongs to the critical set $\mathcal{N}$.
\end{definition}

Following the work of Rosier \cite{Rosier}, the boundary control system of the KdV equation posed on the finite interval $(0,L)$ with various control inputs has been intensively studied  (cf.  \cite{CaCaZh,cerpa1,cerpa2,CoCre,crepeau,GG,GG1,Gui}  and see \cite{cerpatut, RZsurvey} for more complete reviews).
Essentially, the critical length phenomenon arises from employing HUM, utilizing a compactness-uniqueness argument to establish certain observability inequalities. Certainly, through a contradiction argument, we have determined that the adjoint system is observable, which allows us to explore a complex function. The critical length phenomenon manifests when this complex function is potentially entire.

Thus, we are interested in analyzing whether the critical length phenomenon appears for the KdV equation in unbounded domains. Let us consider the three cases treated in this work.

\subsubsection{{\bf The case of the left-half line}} We already know that the critical length phenomenon naturally arises in the case of Neumann boundary control within a bounded domain. However, when considering the left half-line, we can further examine the following control problem
\begin{equation}\label{linearleft_2}
\begin{cases}
\partial_t u+\partial_x u+ \partial_x^3u=0, & \text { for }(x, t) \in(-\infty, L) \times(0, T), \\ 
u(L, t)=0,\quad\ \partial_x u(L, t)=g_2(t), & \text { for } t \in(0, T), \\ u(x, 0)=\phi(x), & \text { for } x \in(-\infty, L).\end{cases}
\end{equation}
By using the HUM, let us recall that the backward system associated with \eqref{linearleft_2} is
\begin{equation}\label{adjointforever}
\begin{cases}
\partial_t\varphi+ \partial_x\varphi+\partial_x^3\varphi=0,& \text { for }(x, t) \in(-\infty,L) \times(0, T),\\ 
\varphi(L, t)=0,& \text { for } t \in(0, T),\\
 \varphi(x, T)=\varphi_T(x),& \text { for } x \in(-\infty,L). 
\end{cases}
\end{equation}
and its observability inequality is given by 
$$
\left\|\varphi_T\right\|_{L^2(-\infty,L)}^2 \leq  C \left\|\partial_x\varphi(L,\cdot)\right\|_{L^{2}(0, T)}^2.
$$
Through the application of the compactness-uniqueness argument, the function in the complex plane associated with the observability inequality is 
 \begin{equation*}
\widehat{\varphi}(\xi)=  \dfrac{e^{-iL\xi}\varphi^{\prime\prime}(L)}{\lambda -i\xi- i\xi^3}.
 \end{equation*}
Note that for any $L>0$, the function $\widehat{\varphi}$ can not be entire. Thus, we do not have any restriction on the length $L$. 

Also, in the case of the control acting in the Dirichlet condition
\begin{equation*}
\begin{cases}
\partial_t u+ \partial_xu+ \partial_x^3u=0,& \text { for }(x, t) \in(-\infty, L) \times(0, T),\\ 
u(L, t)=g_1(t), \quad \partial_x u(L, t)=0,& \text { for } t \in(0, T), \\ u(x, 0)=\phi(x),& \text { for } x \in(-\infty, L),\end{cases}
\end{equation*}
the backward system and its observability inequality are given by \eqref{adjointforever} and 
$$
\left\|\varphi_T\right\|_{L^2(-\infty,L)}^2 \leq  C \left\|\partial_x^2\varphi(L,\cdot)\right\|_{H^{-\frac{1}{3}}(0, T)}^2,
$$
respectively.  In this case, the function in the plane complex is given by
\begin{equation*}
\widehat{\varphi}(\xi)=  \dfrac{\xi e^{-iL\xi} \varphi^{\prime}(L)}{\xi^3+\xi+i\lambda},
 \end{equation*}
and the function $\widehat{\varphi}$ can not be entire. Again, no restriction on the length $L$ is necessary. 

\subsubsection{{\bf The case of the  right-half line}} Consider the control problem posed on the right half-line
\begin{equation}\label{linearright_1}
\begin{cases}
 \partial_tu+ \partial_xu+ \partial_x^3u=0, & \text { for }(x, t) \in(L,+\infty) \times(0, T), \\
 u(L, t)=f(t), & \text { for } t \in(0, T), \\ u(x, 0)=\phi(x),& \text { for } x \in(L,+\infty).
\end{cases}.
\end{equation}
Again, using HUM, the backward system associated with \eqref{linearright_1} is
\begin{equation*}
\begin{cases}
\partial_t\varphi+\partial_x \varphi+\partial_x^3\varphi=0, & \text { for }(x, t) \in(L,+\infty) \times(0, T), \\ 
\varphi(L, t)=\partial_x\varphi(L,t)=0,& \text { for } t \in(0, T), \\
 \varphi(x, T)=\varphi_T(x), & \text { for } x \in(L,+\infty),
\end{cases}
\end{equation*}
and its observability inequality is given by 
\begin{equation*}
\begin{aligned}
\left\|\varphi_T\right\|_{L^2(L,\infty)}^2 \leq & C \left\|\partial_x^2\varphi(L,\cdot)\right\|_{H^{-\frac{1}{3}}(0, T)}^2.
\end{aligned}
\end{equation*}
Through the application of the compactness-uniqueness argument, we do not associate any function in the complex plane, as we directly derive a contradiction from the spectral analysis of the operator.

\vspace{0.2cm}
To summarize, considering all the boundary controllability on the half-line in our work, the critical length phenomenon does not exist, suggesting a behavior completely different compared with the bounded interval, as in the case of the following works \cite{Rosier} and \cite{GG1}.

\subsection{Controllability in $H^s$ \textit{via} forcing operator} From the preceding analysis, considering the case $-\frac34 < s \leq \frac{3}{2}$, with $s\neq \frac12$, we can extend our understanding of controllability for the KdV equation for this range of $s$. By leveraging the well-posedness principles observed in half-line scenarios, we can derive the following scheme:
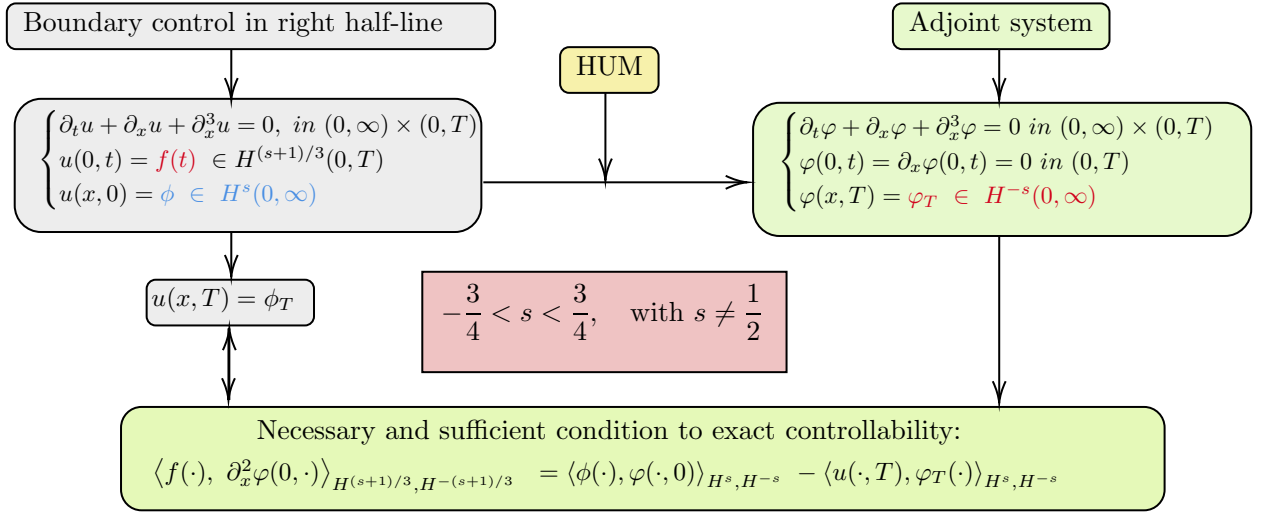
\vglue 0.5cm
\tikzset{every picture/.style={line width=0.75pt}} 
\begin{center}
\begin{tikzpicture}[x=0.75pt,y=0.75pt,yscale=-1,xscale=0.98]

\draw  [fill={rgb, 255:red, 229; green, 249; blue, 184 }  ,fill opacity=1 ] (74.3,224.4) .. controls (74.3,218.66) and (78.96,214) .. (84.7,214) -- (595.2,214) .. controls (600.95,214) and (605.6,218.66) .. (605.6,224.4) -- (605.6,255.6) .. controls (605.6,261.34) and (600.95,266) .. (595.2,266) -- (84.7,266) .. controls (78.96,266) and (74.3,261.34) .. (74.3,255.6) -- cycle ;

\draw  [fill={rgb, 255:red, 237; green, 237; blue, 237 }  ,fill opacity=1 ] (20.17,72.4) .. controls (20.17,65) and (26.17,59) .. (33.57,59) -- (244.35,59) .. controls (251.75,59) and (257.75,65) .. (257.75,72.4) -- (257.75,112.6) .. controls (257.75,120) and (251.75,126) .. (244.35,126) -- (33.57,126) .. controls (26.17,126) and (20.17,120) .. (20.17,112.6) -- cycle ;

\draw  [fill={rgb, 255:red, 231; green, 249; blue, 204 }  ,fill opacity=1 ] (396.09,74.4) .. controls (396.09,67) and (402.09,61) .. (409.49,61) -- (636.31,61) .. controls (643.71,61) and (649.71,67) .. (649.71,74.4) -- (649.71,114.6) .. controls (649.71,122) and (643.71,128) .. (636.31,128) -- (409.49,128) .. controls (402.09,128) and (396.09,122) .. (396.09,114.6) -- cycle ;

\draw    (258.76,101) -- (391.08,101) ;
\draw [shift={(393.08,101)}, rotate = 180] [color={rgb, 255:red, 0; green, 0; blue, 0 }  ][line width=0.75]    (10.93,-3.29) .. controls (6.95,-1.4) and (3.31,-0.3) .. (0,0) .. controls (3.31,0.3) and (6.95,1.4) .. (10.93,3.29)   ;
\draw  [fill={rgb, 255:red, 247; green, 241; blue, 171 }  ,fill opacity=1 ] (297.85,37.6) .. controls (297.85,35.06) and (299.91,33) .. (302.45,33) -- (343.37,33) .. controls (345.91,33) and (347.97,35.06) .. (347.97,37.6) -- (347.97,51.4) .. controls (347.97,53.94) and (345.91,56) .. (343.37,56) -- (302.45,56) .. controls (299.91,56) and (297.85,53.94) .. (297.85,51.4) -- cycle ;

\draw    (320.91,56) -- (320.91,98) ;
\draw [shift={(320.91,100)}, rotate = 270] [color={rgb, 255:red, 0; green, 0; blue, 0 }  ][line width=0.75]    (10.93,-3.29) .. controls (6.95,-1.4) and (3.31,-0.3) .. (0,0) .. controls (3.31,0.3) and (6.95,1.4) .. (10.93,3.29)   ;
\draw  [fill={rgb, 255:red, 231; green, 249; blue, 204 }  ,fill opacity=1 ] (467.27,15.6) .. controls (467.27,13.06) and (469.33,11) .. (471.87,11) -- (572.94,11) .. controls (575.48,11) and (577.54,13.06) .. (577.54,15.6) -- (577.54,29.4) .. controls (577.54,31.94) and (575.48,34) .. (572.94,34) -- (471.87,34) .. controls (469.33,34) and (467.27,31.94) .. (467.27,29.4) -- cycle ;

\draw [fill={rgb, 255:red, 231; green, 249; blue, 204 }  ,fill opacity=1 ]   (521.4,34) -- (521.4,57.5) ;
\draw [shift={(521.4,59.5)}, rotate = 270] [color={rgb, 255:red, 0; green, 0; blue, 0 }  ][line width=0.75]    (10.93,-3.29) .. controls (6.95,-1.4) and (3.31,-0.3) .. (0,0) .. controls (3.31,0.3) and (6.95,1.4) .. (10.93,3.29)   ;

\draw  [fill={rgb, 255:red, 237; green, 237; blue, 237 }  ,fill opacity=1 ] (16.16,15.6) .. controls (16.16,13.06) and (18.22,11) .. (20.76,11) -- (260.17,11) .. controls (262.71,11) and (264.77,13.06) .. (264.77,15.6) -- (264.77,29.4) .. controls (264.77,31.94) and (262.71,34) .. (260.17,34) -- (20.76,34) .. controls (18.22,34) and (16.16,31.94) .. (16.16,29.4) -- cycle ;

\draw    (130.44,34) -- (130.44,57.5) ;
\draw [shift={(130.44,59.5)}, rotate = 270] [color={rgb, 255:red, 0; green, 0; blue, 0 }  ][line width=0.75]    (10.93,-3.29) .. controls (6.95,-1.4) and (3.31,-0.3) .. (0,0) .. controls (3.31,0.3) and (6.95,1.4) .. (10.93,3.29)   ;
\draw  [fill={rgb, 255:red, 237; green, 237; blue, 237 }  ,fill opacity=1 ] (87.34,154.6) .. controls (87.34,152.06) and (89.4,150) .. (91.94,150) -- (167.95,150) .. controls (170.49,150) and (172.55,152.06) .. (172.55,154.6) -- (172.55,168.4) .. controls (172.55,170.94) and (170.49,173) .. (167.95,173) -- (91.94,173) .. controls (89.4,173) and (87.34,170.94) .. (87.34,168.4) -- cycle ;

\draw    (130.44,126) -- (130.44,146.5) ;
\draw [shift={(130.44,148.5)}, rotate = 270] [color={rgb, 255:red, 0; green, 0; blue, 0 }  ][line width=0.75]    (10.93,-3.29) .. controls (6.95,-1.4) and (3.31,-0.3) .. (0,0) .. controls (3.31,0.3) and (6.95,1.4) .. (10.93,3.29)   ;
\draw    (521.4,128) -- (521.4,210) ;
\draw [shift={(521.4,212)}, rotate = 270] [color={rgb, 255:red, 0; green, 0; blue, 0 }  ][line width=0.75]    (10.93,-3.29) .. controls (6.95,-1.4) and (3.31,-0.3) .. (0,0) .. controls (3.31,0.3) and (6.95,1.4) .. (10.93,3.29)   ;
\draw    (129.44,210.5) -- (129.44,176.5) ;
\draw [shift={(129.44,174.5)}, rotate = 90] [color={rgb, 255:red, 0; green, 0; blue, 0 }  ][line width=0.75]    (10.93,-3.29) .. controls (6.95,-1.4) and (3.31,-0.3) .. (0,0) .. controls (3.31,0.3) and (6.95,1.4) .. (10.93,3.29)   ;
\draw  [fill={rgb, 255:red, 239; green, 195; blue, 195 }  ,fill opacity=1 ] (227.68,146) -- (413.13,146) -- (413.13,196.5) -- (227.68,196.5) -- cycle ;

\draw    (129,174.5) -- (129,208.5) ;
\draw [shift={(129,210.5)}, rotate = 270] [color={rgb, 255:red, 0; green, 0; blue, 0 }  ][line width=0.75]    (10.93,-3.29) .. controls (6.95,-1.4) and (3.31,-0.3) .. (0,0) .. controls (3.31,0.3) and (6.95,1.4) .. (10.93,3.29)   ;

\draw (24.46,62.4) node [anchor=north west][inner sep=0.75pt]  [font=\small,xscale=0.90,yscale=0.90]  {$ \begin{array}{l}
\begin{cases}\partial_tu +\partial_xu +\partial_x^3u =0,\ in\ ( 0,\infty ) \times ( 0,T)\\
u( 0,t) =\textcolor[rgb]{0.82,0.01,0.11}{f( t)} \ \in H^{( s+1) /3}( 0,T)\\
u( x,0) =\textcolor[rgb]{0.29,0.56,0.89}{\phi \ \in \ H}\textcolor[rgb]{0.29,0.56,0.89}{^{s}( 0,\infty )}
\end{cases}
\end{array}$};
\draw (401.4,63.4) node [anchor=north west][inner sep=0.75pt]  [font=\small,xscale=0.90,yscale=0.90]  {$ \begin{array}{l}\begin{cases}
\partial_t\varphi  +\partial_x\varphi  +\partial_x^3\varphi=0\ in\ (0,\infty) \times ( 0,T)\\
\varphi ( 0,t) =\partial_x\varphi( 0,t) =0\ in\ ( 0,T)\\
\varphi ( x,T) =\textcolor[rgb]{0.82,0.01,0.11}{\varphi _{T} \ \in \ H^{-s}( 0,\infty }\textcolor[rgb]{0.82,0.01,0.11}{)}
\end{cases}
\end{array}$};
\draw (304.51,36) node [anchor=north west][inner sep=0.75pt]  [xscale=0.95,yscale=0.95] [align=left] {HUM};
\draw (474,14) node [anchor=north west][inner sep=0.75pt]  [xscale=0.95,yscale=0.95] [align=left] {Adjoint system};
\draw (23.07,14) node [anchor=north west][inner sep=0.75pt]  [xscale=0.95,yscale=0.95] [align=left] {Boundary control in right half-line};
\draw (88.44,152.4) node [anchor=north west][inner sep=0.75pt]  [font=\small,xscale=0.95,yscale=0.95]  {$u( x,T) =\phi _{T}$};
\draw (88.96,238.4) node [anchor=north west][inner sep=0.75pt]  [font=\small,xscale=0.95,yscale=0.95]  {$\left< f( \cdot ) ,\ \partial_x^2\varphi( 0,\cdot )\right> _{H^{( s+1) /3} ,H^{-( s+1) /3} \ } \ =\left< \phi ( \cdot ) ,\varphi ( \cdot ,0)\right> _{H^{s} ,H^{-s}} \ -\left< u( \cdot ,T) ,\varphi _{T}( \cdot )\right> _{H^{s} ,H^{-s}}$};
\draw (141.94,219) node [anchor=north west][inner sep=0.75pt]  [xscale=0.95,yscale=0.95] [align=left] {Necessary and sufficient condition to exact controllability:};
\draw (235.9,151) node [anchor=north west][inner sep=0.75pt]  [xscale=0.95,yscale=0.95] [align=left] {$\displaystyle -\frac{3}{4} < s< \frac{3}{4},\quad \text{with} \ s\neq \frac{1}{2}$};
\end{tikzpicture}
\end{center}

This scheme represents the method of controllability in the right half-line (the left half-line can be done similarly). Note that the necessary and sufficient condition to obtain exact controllability may be seen as an optimality condition for the critical points of the functional $\mathcal{J}: H^{-s}(\R^+_x) \longrightarrow \mathbb{R}$, defined by
$$
\mathcal{J}\left(\varphi_T\right)=\frac{1}{2} \left\|\partial_x^2\varphi(0,\cdot) \right\|_{H^{-(s+1)/3}(0,T)}^{2} - \left< \phi ( \cdot ) ,\varphi ( \cdot ,0)\right> _{H^{s} ,H^{-s}} +\left<  u( \cdot ,T) ,\varphi _{T}( \cdot )\right> _{H^{s} ,H^{-s}},
$$
where $\varphi$ is the solution of \eqref{adjoin}  with final data $\varphi_T \in  H^{-s}(\R^+_x)$. Moreover, we already known that if  $\widehat{\varphi}_T \in  H^{-s}(\R^+_x)$ is a minimizer of $\mathcal{J}$, with $\widehat{\varphi} $ the corresponding solution of \eqref{adjoin}, with final data $\widehat{\varphi}_T$, then $f(t)=\widehat{\partial_x^2\varphi}(0,t)$ is a desired control. Thus, the observability inequality is given by 
$$
\left\|\varphi_T\right\|_{H^{-s}(\R^+_x)}^2 \leq  C \left\|\partial_x^2\varphi(0,\cdot)\right\|_{H^{-\frac{s+1}{3}}(0, T)}^2,
$$
for any  $\varphi_T \in H^{-s}(\R^+_x)$, where $\varphi$ is the solution of the backward system \eqref{adjoin}. To prove the observability inequality, several approaches can be utilized. Our intuition regarding the observability inequality is that it is valid in $-\frac34 < s \leq \frac{3}{2}$, with $s\neq \frac12$, given the favorable trace estimates and smoothing effects exhibited by the solution of the backward system, however, this issue of showing the observability inequality in $H^s$, when $-\frac34 < s \leq \frac{3}{2}$, with $s\neq \frac12$, is still an open problem. 

\subsection*{Acknowledgment} The authors are grateful to the referee for the careful reading of this paper and for the valuable suggestions and comments, which substantially improved the results presented herein. This research was conducted after several visits by the authors to the National University of Colombia (sede Manizales) and the Federal University of Pernambuco, and it was completed during the visit of the authors to the Department of Mathematics at Ewha Womans University, Korea, in September 2024. The authors sincerely thank the host institutions for their generous hospitality.

\subsection*{Data Availability} It does not apply to this article as no new data were created or analyzed in this study.
\subsection*{Conflict of interest} This work has no conflicts of interest.

\appendix

\section{Initial boundary value problem: Half-line cases }\label{app1}

\subsection{Boundary forcing operator}\label{FormulaForcing}
Following the ideas in \cite{Holmer}, we can derive a formula for the solution of the initial-value problem \eqref{h1}. The local existence and uniqueness of solutions of the systems \eqref{adjoin} and \eqref{adjoint_1} is established through Theorem \ref{localwellposedness}. Consider a cut-off function $\theta(t):=\theta$, $\theta \in C_0^{\infty}(\mathbb{R})$. Denote $\theta(t)=\frac{1}{T} \psi\left(\frac{t}{T}\right)$, for $T>0$, such that 
\begin{align*}
\begin{cases}
0 \leq \psi \leq 1, \quad \theta \equiv 1&\text{on $[0,1]$,} \\ 
\psi \equiv 0 & \text{for $|t| \geq 2$},
\end{cases}
\end{align*}
Let us now provide a summary of the Riemann-Liouville fractional integral operator; the reader can refer to \cite{CK, Holmer} for more details.  

Define the function $t_+$ as follows
\[t_+ = t \quad \mbox{if} \quad t > 0, \qquad t_+ = 0  \quad \mbox{if} \quad t \le 0.\]
The tempered distribution $\frac{t_+^{\alpha-1}}{\Gamma(\alpha)}$ is defined as a locally integrable function for Re $\alpha>0$ by
\begin{equation*}
	\left \langle \frac{t_+^{\alpha-1}}{\Gamma(\alpha)},\ f \right \rangle=\frac{1}{\Gamma(\alpha)}\int_0^{\infty} t^{\alpha-1}f(t)dt.
\end{equation*}
It follows that
\begin{equation}\label{gamma}
	\frac{t_+^{\alpha-1}}{\Gamma(\alpha)}=\partial_t^k\left( \frac{t_+^{\alpha+k-1}}{\Gamma(\alpha+k)}\right),
\end{equation}
for all $k\in\mathbb{N}$. Expression  \eqref{gamma} can be used to extend the definition of $\frac{t_+^{\alpha-1}}{\Gamma(\alpha)}$ to all $\alpha \in \mathbb{C}$ in the sense of distributions. A change of contour shows the Fourier transform of $\frac{t_+^{\alpha-1}}{\Gamma(\alpha)}$ is the following one
\begin{equation}\label{transformada}
	\left(\frac{t_+^{\alpha-1}}{\Gamma(\alpha)}\right)^{\widehat{}}(\tau)=e^{-\frac{1}{2}\pi i \alpha}(\tau-i0)^{-\alpha},
\end{equation}
where $(\tau-i0)^{-\alpha}$ is the distributional limit. For $\alpha \notin \mathbb{Z}$, let us rewrite \eqref{transformada} on the following way
\begin{equation}\label{transformada1}
	\left(\frac{t_+^{\alpha-1}}{\Gamma(\alpha)}\right)^{\widehat{}}(\tau)=e^{-\frac12 \alpha \pi i}|\tau|^{-\alpha}\chi_{(0,\infty)}+e^{\frac12 \alpha \pi i}|\tau|^{-\alpha}\chi_{(-\infty,0)}.
\end{equation}
Note that from \eqref{transformada} and \eqref{transformada1}, we have that
$$
(\tau-i0)^{-\alpha} = |\tau|^{-\alpha}\chi_{(0,\infty)}+e^{\alpha \pi i}|\tau|^{-\alpha}\chi_{(-\infty,0)}.
$$
For $f\in C_0^{\infty}(\mathbb{R}^+)$, define $	\mathcal{I}_{\alpha}f$ as
\begin{equation*}
	\mathcal{I}_{\alpha}f=\frac{t_+^{\alpha-1}}{\Gamma(\alpha)}*f,
\end{equation*}
so, for $\mbox{Re }\alpha>0$, follows that
\begin{equation}\label{eq:IO}
	\mathcal{I}_{\alpha}f(t)=\frac{1}{\Gamma(\alpha)}\int_0^t(t-s)^{\alpha-1}f(s) \; ds.
\end{equation}

The following properties easily holds $\mathcal{I}_0f=f$, $\mathcal{I}_1f(t)=\int_0^tf(s) \; ds$, $\mathcal{I}_{-1}f=f'$ and $\mathcal{I}_{\alpha}\mathcal{I}_{\beta}=\mathcal{I}_{\alpha+\beta}$. Moreover, the lemmas below can be found in \cite{Holmer}, and we will omit their proofs.
\begin{lemma}\cite[Lemma 2.1]{Holmer}
	If $f\in C_0^{\infty}(\mathbb{R}^+)$, then $\mathcal{I}_{\alpha}f\in C_0^{\infty}(\mathbb{R}^+)$, for all $\alpha \in \mathbb{C}$.
\end{lemma}
\begin{lemma}\cite[Lemma 5.3]{Holmer}
	If $0\leq \mathrm{Re} \ \alpha <\infty$ and $s\in \mathbb{R}$, then $\|\mathcal{I}_{-\alpha}h\|_{H_0^s(\mathbb{R}^+)}\leq c \|h\|_{H_0^{s+\alpha}(\mathbb{R}^+)}$, where $c=c(\alpha)$.
\end{lemma}
\begin{lemma}\cite[Lemma 5.4]{Holmer}
	If $0\leq \mathrm{Re}\ \alpha <\infty$, $s\in \mathbb{R}$ and $\mu\in C_0^{\infty}(\mathbb{R})$, then
	$\|\mu\mathcal{I}_{\alpha}h\|_{H_0^s(\mathbb{R}^+)}\leq c \|h\|_{H_0^{s-\alpha}(\mathbb{R}^+)},$ where $c=c(\mu, \alpha)$.
\end{lemma}

\subsubsection{\bf{Oscillatory integral}} In this subsection, we will define the oscillatory integral, which is the key to defining, in the next section, the Duhamel boundary forcing operator. The Airy function is
\begin{equation}\label{airyfunction}
A(x)=\frac{1}{2 \pi} \int_{\xi} e^{i x \xi} e^{i \xi^3} d \xi
\end{equation}
From \cite{Holmer}, the Airy function has the following properties:
\begin{itemize}
\item[i.] $A(x)$ is a smooth function with the asymptotic properties
$$
\begin{aligned}
A(x) & \sim c_1 x^{-1 / 4} e^{-c_2 x^{3 / 2}}\left(1+O\left(x^{-3 / 4}\right)\right) \quad \text { as } x \rightarrow+\infty
\end{aligned}
$$
and
$$
\begin{aligned}
A(-x) & \sim c_2 x^{-1 / 4} \cos \left(c_2 x^{3 / 2}-\frac{\pi}{4}\right)\left(1+O\left(x^{-3 / 4}\right)\right) \quad \text { as } x \rightarrow+\infty,
\end{aligned}
$$
where $c_1,c_2>0$.
\item[ii.] We can compute:
\begin{equation*}
\begin{split}
A(0)&=\frac{1}{2 \pi} \int_{\xi} e^{i \xi^3} d \xi=\frac{1}{6 \pi} \int_\eta \eta^{-2 / 3} e^{i \eta} d \eta=\frac{\frac{\sqrt{3}}{2} \Gamma\left(\frac{1}{3}\right)}{3 \pi}=\frac{1}{3 \Gamma\left(\frac{2}{3}\right)},\\
A^{\prime}(0)&=\frac{1}{2 \pi} \int_{\xi} i \xi e^{i \xi^3} d \xi=-\frac{1}{3 \Gamma\left(\frac{1}{3}\right)},\quad \text{and}\\
&\int_0^{+\infty} A(y) d y=\frac{1}{3}.
\end{split}
\end{equation*}
\end{itemize}
On the other hand, consider the following group as
\begin{equation}\label{group}
e^{-t (\partial+\partial_x^3)} \phi(x)=\frac{1}{2 \pi} \int_{\xi} e^{i x \xi} e^{i t (\xi+\xi^3)} \hat{\phi}(\xi) d \xi,
\end{equation}
so that
\begin{equation}\label{equationgroup}
\begin{cases}\left(\partial_t+\partial_x+\partial_x^3\right)\left[e^{-t \partial_x^3} \phi\right](x, t)=0, & \text { for }(x, t) \in \mathbb{R} \times \mathbb{R},\\ {\left[e^{-t \partial_x^3} \phi\right](x, 0)=\phi(x)}, & \text { for } x \in \mathbb{R}.\end{cases}
\end{equation}
On the other hand, we define the Duhamel inhomogeneous solution operator $\mathscr{D}$ as
\begin{equation}\label{inhomo}
\mathscr{D} w(x, t)=\int_0^t e^{-\left(t-t^{\prime}\right) (\partial_x + \partial_x^3)} w\left(x, t^{\prime}\right) d t^{\prime},
\end{equation}
so that
\begin{equation}\label{inhomoequation}
\begin{cases}\left(\partial_t+\partial_x+\partial_x^3\right) \mathscr{D} w(x, t)=w(x, t), & \text { for }(x, t) \in \mathbb{R} \times \mathbb{R} \\ \mathscr{D} w(x, 0)=0, & \text { for } x \in \mathbb{R} .\end{cases}
\end{equation}
We now introduce the Duhamel boundary forcing operator (see \cite{CK}). For $f \in C_0^{\infty}\left(\mathbb{R}^{+}\right)$, let
\begin{equation}\label{eqnew17}
\begin{aligned}
\mathscr{L}^0 f(x, t) & =3 \int_0^t e^{-\left(t-t^{\prime}\right) \partial_x^3} \delta_0(x)\mathcal{I}_{-2 / 3} f\left(t^{\prime}\right) d t^{\prime} \\
& =3 \int_0^t A\left(\frac{x}{\left(t-t^{\prime}\right)^{1 / 3}}\right) \frac{\mathcal{I}_{-2 / 3} f\left(t^{\prime}\right)}{\left(t-t^{\prime}\right)^{1 / 3}} d t^{\prime}
\end{aligned}
\end{equation}
so that
$$
\begin{cases}\left(\partial_t+\partial_x^3\right) \mathscr{L}^0 f(x, t)=3 \delta_0(x) \mathcal{I}_{-2 / 3} f(t), & \text { for }(x, t) \in \mathbb{R} \times \mathbb{R}, \\ \mathscr{L}^0 f(x, 0)=0, & \text { for } x \in \mathbb{R} .\end{cases}
$$
We define {\em the Duhamel boundary forcing operator class} for $-1<\lambda<1$, and $f \in C_0^{\infty}\left(\mathbb{R}^{+}\right)$ as
\begin{equation}\label{forcingmenos}
\begin{aligned}
\mathscr{L}_{-}^\lambda f(x, t) & =\left[\frac{x_{+}^{\lambda-1}}{\Gamma(\lambda)} * \mathscr{L}^0\left(\mathcal{I}_{-\lambda / 3} f\right)(-, t)\right](x) \\
&=\frac{1}{\Gamma(\lambda)} \int_{-\infty}^x(y-x)^{\lambda-1} \mathscr{L}^0\left(\mathcal{I}_{-\lambda / 3} f\right)(y, t) d y
\\
&=3 \frac{x_{+}^{(\lambda+3)-1}}{\Gamma(\lambda+3)} \mathcal{I}_{-\frac{2}{3}-\frac{\lambda}{3}} f(t)-\int_{-\infty}^x \frac{(y-x)^{(\lambda+3)-1}}{\Gamma(\lambda+3)} \mathscr{L}^0\left(\partial_t \mathcal{I}_{-\frac{\lambda}{3}} f\right)(y, t) d y
\end{aligned}
\end{equation}
and, with $\frac{x_{-}^{\lambda-1}}{\Gamma(\lambda)}=e^{i \pi \lambda(-x)_{+}^{\lambda-1}} \frac{\Gamma(\lambda)}{(\lambda)}$ define
\begin{equation}\label{forcingmas}
\begin{aligned}
\mathscr{L}_{+}^\lambda f(x, t) & =\left[\frac{x_{-}^{\lambda-1}}{\Gamma(\lambda)} * \mathscr{L}^0\left(\mathcal{I}_{-\lambda / 3} f\right)(-, t)\right](x) \\
&=\frac{1}{\Gamma(\lambda)} \int_x^{\infty}(y-x)^{\lambda-1} \mathscr{L}^0\left(\mathcal{I}_{-\lambda / 3} f\right)(y, t) d y
\\
& =3 \frac{x_{-}^{(\lambda+3)-1}}{\Gamma(\lambda+3)}\mathcal{I}_{-\frac{2}{3}-\frac{\lambda}{3}} f(t)+e^{i \pi \lambda} \int_{-\infty}^x \frac{(y-x)^{(\lambda+3)-1}}{\Gamma(\lambda+3)} \mathscr{L}^0\left(\partial_t\mathcal{I}_{-\frac{\lambda}{3}} f\right)(y, t) d y .
\end{aligned}
\end{equation}
It is straightforward
from these definitions, in the sense of distributions
$$
\left(\partial_t+\partial_x+\partial_x^3\right) \mathscr{L}_{-}^\lambda f(x, t)=3 \frac{x_{+}^{\lambda-1}}{\Gamma(\lambda)} \mathcal{I}_{-\frac{2}{3}-\frac{\lambda}{3}} f(t) + 3( \lambda+2)\frac{x_{+}^{\lambda+1}}{\Gamma(\lambda+3)} \mathcal{I}_{-\frac{2}{3}-\frac{\lambda}{3}} f(t)
$$
and
$$
\left(\partial_t+\partial_x+\partial_x^3\right) \mathscr{L}_{+}^\lambda f(x, t)=3 \frac{x_{-}^{\lambda-1}}{\Gamma(\lambda)} \mathcal{I}_{-\frac{2}{3}-\frac{\lambda}{3}} f(t)+3(\lambda+2) \frac{x_{-}^{\lambda+1}}{\Gamma(\lambda+3)} \mathcal{I}_{-\frac{2}{3}-\frac{\lambda}{3}} f(t) .
$$

\subsubsection{\bf{Solution of the systems}} To address the nonlinear problem \eqref{h1} with given data $f$ and $\phi$, take $-1<\lambda<1$, we set
$$
u(x, t)=\theta(t) e^{-t (\partial_x+\partial_x^3)} \phi(x)-\frac{1}{2} \theta(t) \mathscr{D \partial _ { x }} u^2(x, t)+\theta(t) \mathscr{L}_{+}^\lambda h(x, t),
$$
where
\begin{equation}\label{boundaryforcing2}
h(t)=e^{-\pi i \lambda}\left[f(t)-\left.\theta(t) e^{-t(\partial_x+ \partial_x^3)} \phi\right|_{x=0}+\frac{1}{2} \theta(t) \mathscr{D} \partial_x u^2(0, t)\right] .
\end{equation}
Then
\begin{equation*}
\begin{split}
\left(\partial_t+\partial_x+\partial_x^3\right) u(x, t)=&-\frac{1}{2} \partial_x u^2(x, t)+3 \frac{x_{-}^{\lambda-1}}{\Gamma(\lambda)}\theta(t) \mathcal{I}_{-\frac{2}{3}-\frac{\lambda}{3}} f(t)\\&+3(\lambda+2) \frac{x_{-}^{\lambda+1}}{\Gamma(\lambda+3)} \theta(t) \mathcal{I}_{-\frac{2}{3}-\frac{\lambda}{3}} f(t).
\end{split}
\end{equation*}
Due to the support properties of  $x_{-}^{\lambda\pm 1}$ and \eqref{boundaryforcing2}, we have that $u$ is the solution of \eqref{h1}.

Additionally, if we take $-1<\lambda_1, \lambda_2<1, \lambda_1 \neq \lambda_2$, and set
$$
u(x, t)=\theta(t) e^{-t (\partial_x+\partial_x^3)} \phi(x)-\frac{1}{2} \theta(t) \mathscr{D} \partial_x u^2(x, t)+\theta(t) \mathscr{L}_{-}^{\lambda_1} h_1(x, t)+\theta(t) \mathscr{L}_{-}^{\lambda_2} h_2(x, t),
$$
where
\begin{equation}\label{bondaryforcing1}
\left[\begin{array}{l}
h_1(t) \\
h_2(t)
\end{array}\right]=M\left[\begin{array}{c}
g_1(t)-\left.\theta(t) e^{-t (\partial_x+ \partial_x^3)} \phi\right|_{x=0}+\frac{1}{2} \theta(t) \mathscr{D} \partial_x u^2(0, t) \\
\theta(t) \mathcal{I}_{1 / 3}\left(g_2-\left.\theta \partial_x e^{-t (\partial_x+\partial_x^3)} \phi\right|_{x=0}+\frac{1}{2} \theta \partial_x \mathscr{D} \partial_x u^2(0, \cdot)\right)(t)
\end{array}\right] .
\end{equation}
with 
\begin{equation}\label{matrixA}
M=\frac{1}{2 \sqrt{3} \sin \left[\frac{\pi}{3}\left(\lambda_2-\lambda_1\right)\right]}\left[\begin{array}{cc}
\sin \left(\frac{\pi}{3} \lambda_2-\frac{\pi}{6}\right) & -\sin \left(\frac{\pi}{3} \lambda_2+\frac{\pi}{6}\right) \\
-\sin \left(\frac{\pi}{3} \lambda_1-\frac{\pi}{6}\right) & \sin \left(\frac{\pi}{3} \lambda_1+\frac{\pi}{6}\right),
\end{array}\right].
\end{equation}
We have that, in the sense of distribution, $u$ satisfies 
\begin{equation*}
\begin{split}
(\partial_t + \partial_x+\partial_x^3)u(x,t)= &-\frac12 \partial_x u^2(x,t) + 3 \frac{x_{+}^{\lambda_1-1}}{\Gamma(\lambda_1)} \mathcal{I}_{-\frac{2}{3}-\frac{\lambda_1}{3}} h_1(t) + 3( \lambda_1+2)\frac{x_{+}^{\lambda_1+1}}{\Gamma(\lambda_1+3)} \mathcal{I}_{-\frac{2}{3}-\frac{\lambda_1}{3}} h_1(t) \\
&+3 \frac{x_{+}^{\lambda_2-1}}{\Gamma(\lambda_2)} \mathcal{I}_{-\frac{2}{3}-\frac{\lambda_2}{3}} h_2(t) + 3( \lambda_2+2)\frac{x_{+}^{\lambda_2+1}}{\Gamma(\lambda_2+3)} \mathcal{I}_{-\frac{2}{3}-\frac{\lambda_2}{3}} d_2(t).
\end{split}
\end{equation*}
Due to the support properties of  $x_{+}^{\lambda\pm 1}$ and \eqref{bondaryforcing1}, we have that $u$ is the solution of \eqref{h2}.

\subsubsection{\bf{Main estimates}} The operator $e^{-t (\partial_x + \partial_x^3)}$ was defined above in \eqref{group} satisfying \eqref{equationgroup}.
Thus, the following lemma holds.
\begin{lemma}\cite[Lemma 5.5]{Holmer}\label{estimatesgroup}
Let $s \in \mathbb{R}$. Then:
\begin{enumerate}
\item[(a)] (Space traces) $\left\|e^{-t (\partial_x + \partial_x^3)} \phi(x)\right\|_{C\left(\mathbb{R}_t ; H_x^s\right)} \leq c\|\phi\|_{H^s}$;
\item[(b)] (Time traces) $\left\|\theta(t) e^{-t (\partial_x + \partial_x^3)} \phi(x)\right\|_{C\left(\mathbb{R}_x ; H_t^{\frac{s+1}{3}}\right)} \leq c\|\phi\|_{H^s}$;
\item[(c)] (Derivative time traces) $\left\|\theta(t) \partial_x e^{-t (\partial_x + \partial_x^3)} \phi(x)\right\|_{C\left(\mathbb{R}_x ; H_t^{\frac{s}{3}}\right)} \leq c\|\phi\|_{H^s}$;
\item[(d)] (Bourgain space estimate) If $0<b<1$ and $0<\alpha<1$, then $$\left\|\theta(t) e^{-t (\partial_x + \partial_x^3)} \phi(x)\right\|_{X_{s, b} \cap D_\alpha} \leq c\|\theta\|_{H^1}\|\phi\|_{H^s},$$ where $c$ is independent of $\theta$.
\end{enumerate}
\end{lemma}

In addition to the previous lemma, and taking into account that the operator $\mathscr{D}$ was defined above in \eqref{inhomo}, satisfying \eqref{inhomoequation}. Consider
$$
\|u\|_{Y_{s, b}}=\left(\iint_{\xi, \tau}\langle\tau\rangle^{2 s / 3}\left\langle\tau-\xi^3\right\rangle^{2 b}|\hat{u}(\xi, \tau)|^2 d \xi d \tau\right)^{1 / 2}.
$$
Thus, the next result is as follows.
\begin{lemma}\cite[Lemma 5.6]{Holmer}\label{estimateinhom}
 Let $s \in \mathbb{R}$. Then:
 \begin{enumerate}
 \item[(a)] (Space traces) If $0 \leq b<\frac{1}{2}$, then
$$
\|\theta(t) \mathscr{D} w(x, t)\|_{C\left(\mathbb{R}_t ; H_x^s\right)} \leq c\|w\|_{X_{s,-b}} ;
$$
\item[(b)] (Time traces) If $0<b<\frac{1}{2}$, then
$$
\|\theta(t) \mathscr{D} w(x, t)\|_{C\left(\mathbb{R}_x ; H_t^{\frac{s+1}{3}}\right)} \leq \begin{cases}c\|w\|_{X_{s,-b}}, & \text { if }-1 \leq s \leq \frac{1}{2}, \\ c\left(\|w\|_{X_{s,-b}}+\|w\|_{Y_{s,-b}}\right), & \text { for any } s .\end{cases}
$$

If $s<\frac{7}{2}$, then $\|\theta(t) \mathscr{D} w(x, t)\|_{C\left(\mathbb{R}_x ; H_0^{\frac{s+1}{3}}\left(\mathbb{R}_t^{+}\right)\right)}$has the same bound.
\item[(c)] (Derivative time traces) If $0<b<\frac{1}{2}$, then
$$
\left\|\theta(t) \partial_x \mathscr{D} w(x, t)\right\|_{C\left(\mathbb{R}_x ; H_t^{\frac{s}{3}}\right)} \leq \begin{cases}c\|w\|_{X_{s,-b}},& \text { if } 0 \leq s \leq \frac{3}{2},\\ c\left(\|w\|_{X_{s,-b}}+\|w\|_{Y_{s,-b}}\right), & \text { for any } s .\end{cases}
$$

If $s<\frac{9}{2}$, then $\left\|\theta(t) \partial_x \mathscr{D} w(x, t)\right\|_{C\left(\mathbb{R}_x ; H_0^{\frac{s}{3}}\left(\mathbb{R}_t^{+}\right)\right)}$has the same bound.
\item[(d)] (Bourgain space estimate) If $0 \leq b<\frac{1}{2}$ and $\alpha \leq 1-b$, then $$\|\theta(t) \mathscr{D} w(x, t)\|_{X_{s, b} \cap D_\alpha} \leq c\|w\|_{X_{s,-b}}.$$
\end{enumerate}
\end{lemma}

Finally, the operators $\mathscr{L}_{ \pm}^\lambda$  defined in \eqref{forcingmenos} and \eqref{forcingmas} have the following properties.
\begin{lemma}\cite[Lemma 5.8]{Holmer}\label{estimateforcing}
Let $s \in \mathbb{R}$. Then:
\begin{enumerate}
\item[(a)] (Space traces) If $s-\frac{5}{2}<\lambda<s+\frac{1}{2}, \quad \lambda<\frac{1}{2}$, and supp $f \subset[0,1]$, then $$\left\|\mathscr{L}_{ \pm}^\lambda f(x, t)\right\|_{C\left(\mathbb{R}_t ; H_x^s\right)} \leq c\|f\|_{ H_0^{\frac{s+1}{3}\left(\mathbb{R}^{+}\right)}}.$$
\item[(b)] (Time traces) If $-2<\lambda<1$, then
$$
\left\|\theta(t) \mathscr{L}_{ \pm}^\lambda f(x, t)\right\|_{C\left(\mathbb{R}_x ; H_0^{\frac{s+1}{3}}\left(\mathbb{R}_t^{+}\right)\right)} \leq c\|f\|_{H_0^{\frac{s+1}{3}}\left(\mathbb{R}^{+}\right)}
$$
\item[(c)] (Derivative time traces) If $-1<\lambda<2$, then
$$
\left\|\theta(t) \partial_x \mathscr{L}_{ \pm}^\lambda f(x, t)\right\|_{C\left(\mathbb{R}_x ; H_0^{\frac{s}{3}}\left(\mathbb{R}_t^{+}\right)\right)} \leq c\|f\|_{H_0^{\frac{s+1}{3}}\left(\mathbb{R}_t^{+}\right)}
$$
\item[(d)] (Bourgain space estimate) If $s-1 \leq \lambda<s+\frac{1}{2}, \lambda<\frac{1}{2}, \alpha \leq \frac{s-\lambda+2}{3}$, and $0 \leq b<\frac{1}{2}$, then $$\left\|\theta(t) \mathscr{L}_{ \pm}^\lambda f(x, t)\right\|_{X_{s, b} \cap D_\alpha} \leq c\|f\|_{H_0^{\frac{s+1}{3}\left(\mathbb{R}_t^{+}\right)}}.$$
\end{enumerate}
\end{lemma}

\end{document}